\documentclass[11pt]{article}
\usepackage{amssymb,latexsym,amsmath, amsthm}
\usepackage{amsfonts,amssymb, latexsym, amsmath, amsthm,mathrsfs, verbatim,geometry}

\usepackage{amsfonts, graphics}
\usepackage{mathtools}
\usepackage{hyperref} 
\usepackage{csquotes} 
\usepackage{enumitem,moreenum} 
\usepackage{float}
\usepackage{dsfont} 
\numberwithin{equation}{section}
\def\R{\mathbb R}
\def\C{\mathbb C}
\def\N{\mathbb N}

\hypersetup{
    colorlinks,
    citecolor=red,
    filecolor=green,
    linkcolor=blue,
    urlcolor=black
}

\def\supp{\mathrm{supp}\,}
\def\div{\mathrm{div}\,}
\def\divx{\mathrm{div}_x}
\def\divv{\mathrm{div}_v}

\def\A{{\mathcal A}}
\def\D{{\mathcal D}}
\def\B{{\mathcal B}}
\def\J{{\mathcal J}}
\def\Ri{{\mathcal R}}
\def\sign{{\mathrm{sign}}}
\newcommand{\id}{\mathds{1}}
\newcommand{\weakto}{\rightharpoonup}
\def\pa{\partial}
\def\l{\lambda}
\newcommand{\inter}{{\mathrm{int}}}

\def\pOmega{\bar\Omega}
\def\pE{E} 
\def\pvarphi{\varphi}
\def\pPhi{\Phi}
\def\pU{U}
\def\pUmin{\pU_0(0)}
\def\prho{\rho}
\def\pM{M}
\def\pR{R}
\def\pT{T}
\def\pG{G}
\def\pH{\bar H}
\def\pK{\bar K}
\def\pN{\bar N}
\def\pI{\bar I}
\def\pD{\bar{\mathcal D}}
\def\pB{\bar{\mathcal B}}
\def\pA{\bar{\mathcal A}}
\def\pJ{\bar{\mathcal J}}
\def\pv{\bar{v}}
\def\palpha{\alpha}
\def\pbeta{\beta}
\def\pf{f}
\newcommand{\p}[1]{\bar{#1}}
\def\F{\mathcal F}
\def\M{\mathcal M}
\def\pQ{\bar{Q}}
\def\pS{\bar{S}}

\def\tx{\tilde x}
\def\H{\mathcal H}
\def\Ltwo{H} 

\let\originalleft\left
\let\originalright\right
\renewcommand{\left}{\mathopen{}\mathclose\bgroup\originalleft}
\renewcommand{\right}{\aftergroup\egroup\originalright}

\newcommand*\diff{\mathop{}\!\mathrm{d}} 

\newtheorem{theorem}{Theorem}[section]
\newtheorem{defn}[theorem]{Definition}
\newtheorem{prop}[theorem]{Proposition}

\newtheorem{cor}[theorem]{Corollary}

\newtheorem{lemma}[theorem]{Lemma}

\newtheorem{remark}[theorem]{Remark}
\newtheorem{defnrem}[theorem]{Definition \& Remark}

\usepackage{color}
\def\bcr{\begin{color}{red}}
\def\bcb{\begin{color}{blue}}
\def\ec{\end{color}}
\def\be{\begin{equation}}
\def\ee{\end{equation}}

\begin{document}

\title{On the existence of linearly oscillating galaxies}
\author{Mahir~Had\v{z}i\'{c}
	\thanks{University College London, UK. Email: m.hadzic@ucl.ac.uk}~, Gerhard~Rein
	\thanks{University of Bayreuth, Germany. Email: gerhard.rein@uni-bayreuth.de}~, and Christopher~Straub
	\thanks{University of Bayreuth, Germany. Email: christopher.straub@uni-bayreuth.de}}

\maketitle

\begin{abstract}
	We consider two classes of steady states of the three-dimensional,
	gravitational Vlasov-Poisson system: the spherically symmetric Antonov-stable
	steady states (including the polytropes and the King model) and their
	plane symmetric analogues. We completely describe the essential spectrum of
	the self-adjoint operator governing the linearized dynamics in the neighborhood of these
	steady states. We also show that for the steady states under consideration,
	there exists a gap in the spectrum.
	We then use a version of the Birman-Schwinger
	principle first used by Mathur to derive a general criterion for the existence of
	an eigenvalue inside the first gap of the essential spectrum, which corresponds to
	linear oscillations about the steady state. It follows in particular that no linear
	 Landau damping can occur in the neighborhood of steady states satisfying our criterion. 
	 Verification of this criterion requires
	a good understanding of the so-called period function associated with each
	steady state. In the plane symmetric case we verify the criterion rigorously,
	while in the spherically symmetric case we do so under a natural monotonicity assumption
	for the associated period function. Our results
	explain the pulsating behavior triggered by perturbing
	such steady states, which has been observed numerically.
\end{abstract}

\tableofcontents

\section{Introduction}

\subsection{The basic set-up and main objective}
The three-dimensional gravitational Vlasov-Poisson system is the fundamental system
of equations used in astrophysics to describe galaxies~\cite{BiTr}. The unknowns are
the phase-space density function $f\colon\mathbb R\times\mathbb R^3\times\R^3\to[0,\infty[$ and the
gravitational potential $U\colon\mathbb R\times\mathbb R^3\to\mathbb R$, and they satisfy
\begin{equation}
  \partial_tf + v\cdot\nabla_xf-\nabla_xU\cdot\nabla_vf = 0, \label{E:FEQN}
\end{equation}
\begin{equation}
  \Delta U = 4\pi\rho,\label{E:POISSONEQN}
\end{equation}
\begin{equation}  
  \rho(t,x) =\int_{\R^3}f(t,x,v) \diff v, \label{E:RHOEQN}
\end{equation}
where $\rho(t,\cdot)$ is the macroscopic density of the stars at time $t\in\R$
and $\cdot$ denotes the Euclidean scalar product.
Equation~\eqref{E:FEQN} is the Vlasov equation and~\eqref{E:POISSONEQN}
is the Poisson equation. Isolated systems are characterized by the boundary condition
\begin{equation}\label{E:AF}
  \lim_{\vert x\vert\to\infty}U(t,x)=0,
\end{equation}
which uniquely specifies the solution of the Poisson equation~\eqref{E:POISSONEQN}
for a given right-hand side. We refer to the system~\eqref{E:FEQN}--\eqref{E:RHOEQN}
as the {\em Vlasov-Poisson system}.

The Vlasov-Poisson system possesses a plethora of spatially localized steady
states which serve as models of stationary galaxies~\cite{BiTr,FrPo1984}.
The dynamic stability of such equilibria 
has attracted a lot of interest in both the physics and mathematics
communities; we will review some of the corresponding literature below.
However, even when a steady state has been shown to be stable the actual dynamical
behavior triggered by a small perturbation of it is not determined.
In a numerical investigation of the Einstein-Vlasov system, which is
the relativistic version of the Vlasov-Poisson system, it was observed that
such perturbations lead to solutions which oscillate about the steady state
\cite{AR}. In these oscillations the spatial support of the solutions expands
and contracts in a time-periodic way, i.e., after perturbation the state starts
to pulse. The same behavior was observed numerically for the Vlasov-Poisson
system in~\cite{RaRe2018}, and again
for the Einstein-Vlasov system in~\cite{Gue_e_a}.

Such pulsating solutions are classical for the Euler-Poisson system
and have been used to explain the Cepheid variable stars \cite{edd,ross}.
While a general oscillatory behavior in the context of the Vlasov matter model
has long been discussed \cite{BiTr,LoGe1988},
the above numerical investigations are to our knowledge the first instance that
such pulsating behavior has been observed in the mathematics or astrophysics
literature, with one notable exception.
In \cite{Ku78}, R.~Kurth constructs
a semi-explicit, one-parameter family of solutions to the Vlasov-Poisson system
which show exactly such pulsating behavior and which are small
perturbations of a particular
steady state when the parameter is small.

For the Euler-Poisson system a mathematically rigorous analysis of pulsating
solutions is provided in \cite{jang,mak}, and the main objective of the present paper
is to understand the corresponding issue for the Vlasov-Poisson system.
A natural first step in this undertaking is to linearize the system about some
steady state $f_0$, i.e., to substitute
\[
f = f_0 + \delta f
\]
into the system and drop all terms which are of higher order than linear in
the perturbation $\delta f$. For this to be justified $\delta f$ must be small
compared to $f_0$, in particular $\delta f$ must vanish outside the support
of $f_0$. So even if one finds an eigenvalue $i \omega$ (with $\omega >0$)
of the linear equation
which governs the dynamics of $\delta f$ and a corresponding eigenfunction
$f^\ast$, this does not explain the observed pulsations, since
$f_0 + e^{i\omega t} f^\ast$ is supported in the support of $f_0$.
However, if we linearize the Vlasov-Poisson system not starting from the Eulerian
picture as above, but in the Lagrangian picture or alternatively in suitable
mass-Lagrange coordinates, then the corresponding linearized analysis does
indeed capture the pulsating behavior which was observed numerically,
even though the linear operator which governs these alternative linearizations
is (essentially) the same as before. One should at this point note that the
numerical observations mentioned above refer to the non-linear systems as they
stand, and it is a-priori not clear that the observed pulsations can be
captured in a linear analysis and are not a purely non-linear effect.

The main objective of the present paper is to understand the pulsating perturbations
of certain steady states of the Vlasov-Poisson system on the linearized level.
In the next sections of this introduction we outline our paper in
more detail and put it into perspective.

\subsection{Symmetry classes}\label{ssc:symmetries}

The steady states which are studied in this paper and their perturbations will be
either spherically symmetric or plane symmetric. We need
to make precise what we mean by these symmetries, and we begin with spherical symmetry.

A phase-space density function $f$ is {\em spherically symmetric} if
\begin{equation} \label{ss_symm_def}
f(t,Ax,Av)=f(t,x,v)\  \mbox{for all}\ x,v \in\R^3 \ \mbox{and all rotations}\
A\in \text{SO}(3).
\end{equation}
In this case, it is sometimes convenient to use coordinates adapted to spherical symmetry,
\begin{equation}\label{eq:rwL}
  r = |x|, \ \ w = \frac{x\cdot v}{r}, \ \ L = |x\times v|^2,
\end{equation}
where $r$ is the spatial radius, $w$ is the radial velocity, and $L$ is the modulus of the angular
momentum squared. In these variables the Vlasov-Poisson system takes the form
\begin{equation}\label{ss_vlasov}
  \pa_t f + w\, \pa_r f + \left(\frac{L}{r^3} - \frac{m(t,r)}{r^2}\right)\,\pa_w f =0,
\end{equation}
\begin{equation} \label{ss_m}
  m(t,r) = 4\pi \int_0^r s^2\rho(t,s)\diff s,
\end{equation}
\begin{equation} \label{ss_rho}
  \rho(t,r) = \frac{\pi}{r^2} \int_{-\infty}^\infty \int_0^\infty f(t,r,w,L)\diff L\diff w;
\end{equation}
for the corresponding solution $U$ of the Poisson equation, $\pa_r U = m/r^2$.

A phase-space density function $f$ is {\em plane symmetric} if
\begin{equation} \label{pl_symm_def}
  f(t,x,v)=f(t,x_1,v)=f(t,-x_1,-v_1,\pv)\  \mbox{for all}\ x,v \in\R^3,
\end{equation}
which implies that $\rho(t,x)=\rho(t,x_1)=\rho(t,-x_1)$
and $U$ can be chosen such that $U(t,x)=U(t,x_1)=U(t,-x_1)$.
Since in the case of planar symmetry the effective spatial variable is only $x_1$,
we drop the subscript and view $x\in \R$ in the planar case, while
$v=(v_1,\bar v) \in \R^3$.
With this symmetry assumption and notational convention
the Vlasov-Poisson system can be written as
\begin{equation} 
  \partial_t f + v_1\,\partial_{x}f-\partial_{x} U\, \partial_{v_1} f = 0, \label{pl_vlasov}
\end{equation}
\begin{equation} \label{pl_poisson}
  U(t,x) = 2\pi \int_\R |x-y|\, \rho(t,y) \diff y,
\end{equation}
\begin{equation}  
  \rho(t,x) =\int_{\R^3}f(t,x,v) \diff v \label{pl_rho}.
\end{equation}
In the context of our investigation the system
\eqref{pl_vlasov}--\eqref{pl_rho} is equivalent to the one-dimensional Vlasov-Poisson system,
since one can simply integrate out the $\bar v$-dependence, but we prefer to keep to the
above form which appears in the astrophysics literature~\cite{An1971,KuMa1970}, and the stability analysis of
which differs from the purely one-dimensional one. 
Clearly, within this symmetry class we cannot require the boundary condition \eqref{E:AF},
but \eqref{pl_poisson} implies that
\begin{equation} \label{pl_dxU}
  \partial_x U(t,x) = 2\pi \int_\R \sign(x-y)\, \rho(t,y) \diff y,
\end{equation}
and hence
$\lim_{x\to-\infty}\partial_x U(t,x) = - \lim_{x\to\infty}\partial_x U(t,x)$,
which is a natural substitute
for \eqref{E:AF}. The additional reflection symmetry which we included in \eqref{pl_symm_def}
implies that $\partial_x U(t,0)=0$ and
\begin{equation} \label{pl_dxU_sym}
  \partial_x U(t,x) = 4\pi \int_0^x\rho(t,y) \diff y;
\end{equation}
this identity will be important in what follows.

A short comment on our (ab)use of notation seems in order. We will throughout the paper
identify $f(t,x,v)=f(t,r,w,L)$ and analogously for other functions or variables.
However, we will distinguish different representations of the same
set---for example the support of some steady state---in different coordinates.
Moreover, while we do not notationally distinguish between a general, or spherically symmetric,
or plane symmetric phase space density $f$, we will notationally distinguish between
function spaces and operators consisting of or acting on spherically symmetric,
or plane symmetric functions respectively; the latter will be distinguished by a bar,
so $\A$ will act on spherically symmetric functions, while $\bar \A$ will act on plane symmetric 
ones.

\subsection{Steady states, linearization, and stability}

The most important class of steady states to the Vlasov-Poisson system are
{\em isolated, spherically symmetric} steady states,
which are time-independent solutions of \eqref{ss_vlasov}--\eqref{ss_rho} with the boundary condition~\eqref{E:AF}.
Physically 
less important but mathematically easier to analyze are {\em plane symmetric} steady
states solving the system \eqref{pl_vlasov}--\eqref{pl_rho}. 
The latter have nevertheless been widely used 
in past investigations~\cite{KuMa1970,Louis1992,Mark1971,Ma,Weinberg1991}. 

To find a spherically symmetric steady state $f_0$ of the Vlasov-Poisson system
one prescribes a {\em microscopic} equation of state $\varphi$ and
seeks a solution of the form
\begin{equation}\label{E:MICRO}
	f_0(r,w,L) = \varphi(E,L),\quad E(r,w,L) = \frac12w^2 + \Psi_L(r),
\end{equation}
where
\begin{equation}\label{psildef}
	\Psi_L(r) = \frac{L}{2r^2} + U_0(r)
\end{equation}
is the effective potential and $U_0$ is the gravitational potential
induced by $f_0$ via the Poisson equation.
Both $E$ and $L$ are preserved by the characteristic flow of the
spherically symmetric Vlasov equation~\eqref{ss_vlasov},
so any sufficiently regular $f_0$ of the form \eqref{E:MICRO}
is a solution to the Vlasov equation. For a wide range of $\varphi$ one
finds time-independent solutions of the form~\eqref{E:MICRO} with finite
mass and compact support in the phase space \cite{RaRe13}. The most representative are the polytropes
\begin{equation}\label{E:POLYTROPESINTRO}
	\varphi(E,L) = (E_0-E)^k_+(L-L_0)_+^l,
\end{equation}
and the King model
\begin{equation}\label{E:KINGINTRO}
	\varphi(E) = (e^{E_0-E}-1)_+,
\end{equation}
which are used extensively in the astrophysics literature~\cite{BiTr}.
Here, $E_0<0$, $L_0\geq0$, and $k,l\in\R$ are suitably chosen parameters,
see the discussion in the paragraph after~\eqref{E:KING} in Section~\ref{ssc:ststradial}.
If not stated explicitly otherwise 
we employ the following notational convention throughout the paper. For $t,k\in \R$,
\begin{align} \label{E:+conv}
	t_+^k \coloneqq \left\{ \begin{array}{cc}t^k,&t>0,\\ 0,&t\leq 0.\end{array}\right.
\end{align}
In the planar case the ansatz
\begin{align}
	\pf_0(x,v)=\pvarphi(\pE,\pv), \qquad \pE(x,v)=\frac12v_1^2+\pU_0(x) ,\qquad (x,v)\in\R\times\R^3,
\end{align}
leads to a solution of the Vlasov equation~\eqref{pl_vlasov},
since $\pE$ and $\pv=(v_2,v_3)$ are conserved quantities of the characteristic flow.
Analogously to the radial situation, 
there exist plane symmetric steady states of polytropic form
\begin{align}\label{E:POLYTROPESPLANARINTRO}
\pvarphi(\pE,\pv) = (\pE_0-\pE)^k_+\,\pbeta(\pv), 
\end{align}
and of King type
\begin{align}\label{E:KINGPLANARINTRO}
\pvarphi(\pE,\pv) = (e^{\pE_0-\pE}-1)_+\,\pbeta(\pv).
\end{align}
For the discussion of the range of exponents $k$ in~\eqref{E:POLYTROPESPLANARINTRO}, the choice of $\pE_0$ as well as the $\pv$-dependency $\pbeta$, see Section \ref{ssc:ststp}.
We refer to the above steady states as the {\em plane symmetric polytropes} and the {\em plane symmetric King} solution respectively.
All of the following can be done for a much larger class of steady states, but for the sake of clarity we limit the discussion to these classical models.

A formal linearization of the Vlasov-Poisson system about some radial steady state
$f_0$ with potential $U_0$ yields a linearized equation of the form
\begin{align}\label{E:LINDYNINTRO}
	\pa_t^2 h + \A h =0,
\end{align}
where $h=h(x,v)$ is odd in $v$ and the {\em Antonov operator} or {\em linearized operator} $\A$ is given as
\begin{align}\label{E:ABDEF}
	\A\coloneqq - \D^2 - \B,
\end{align} 
where $\B$ is a bounded, symmetric operator, and
\begin{align}\label{E:Ddef}
	\D: = v\cdot\nabla_x - \nabla_x U_0 \cdot\nabla_v
\end{align}
is the transport operator associated with the characteristic flow of the steady state,
whose functional analytical properties have been investigated in \cite{GuLi08,LeMeRa11}
and recently in~\cite{ReSt20}. 
In Section~\ref{S:LINEARISATION} we carry out this linearization and the derivation of
the operator $\A$, starting from the Eulerian and the Lagrangian formulations
of the system, and also in the so-called mass-Lagrange coordinates.

A similar analysis around the plane symmetric steady states
\eqref{E:POLYTROPESPLANARINTRO}--\eqref{E:KINGPLANARINTRO} leads to
the linearization analogous to~\eqref{E:LINDYNINTRO} with the Antonov operator $\A$
replaced by a related operator $\pA$, see Section~\ref{S:LINEARISATION}.

The basic criterion for linear stability is the absence of growing modes.
Starting in the 1960's, a simple criterion for linear stability was proposed and
formally derived in the astrophysics literature~\cite{An1961,DoFeBa,KS}:
if 
\begin{align}\label{E:ANTONOVSCRITERION}
	\varphi'(E,L)\coloneqq\pa_E\varphi(E,L) <0
\end{align} 
then the corresponding steady state $f_0$ is linearly stable.
This is equivalent to the non-negativity of the quadratic form 
\be\label{E:QUADRATICFORM}
\langle\A g, g\rangle_H,
\ee
where $H$ is a suitable Hilbert space consisting of spherically symmetric functions
on the support of the steady state which are square-integrable with respect to a certain weight;
the latter is chosen such that $\A$ is symmetric on $\Ltwo$, see Section~\ref{sc:operator}.
We refer to the non-negativity of~\eqref{E:QUADRATICFORM} as {\em Antonov coercivity}
and for the reader's convenience this bound is stated in Proposition~\ref{antonovcoerc}. 
The equilibria satisfying~\eqref{E:ANTONOVSCRITERION} are therefore called {\em Antonov-stable} or {\em linearly stable}. The condition~\eqref{E:ANTONOVSCRITERION} has a simple interpretation---a steady galaxy of the form~\eqref{E:MICRO} is stable if the concentration of ever more energetic stars is decreasing on the phase-space support of the galaxy. If~\eqref{E:ANTONOVSCRITERION} is not satisfied the steady state may in fact be linearly unstable, i.e., there exist growing modes, cf.~\cite{GuLi08,GuoLinetal2013}.

A planar analogue of Antonov's coercivity bound has been shown in~\cite{KuMa1970} and is proven in Proposition~\ref{antonovcoerc1d} using the techniques from \cite{GuRe2007,LeMeRa11}. The stability of planar steady states is e.g.~discussed in \cite{Mark1971}.

In the mathematics literature a lot of effort went into the rigorous non-linear stability analysis, see~\cite{GuLi08,GuRe2001,GuRe2007,LeMeRa11,LeMeRa12,Re07} and the references therein. In particular, Lemou, M\'ehats, and Rapha\"el~\cite{LeMeRa12} showed non-linear orbital stability against general perturbations for a wide range of spherically symmetric, isotropic steady states. 

By contrast, very little is rigorously known about the linear or non-linear
{\em asymptotic stability} properties of the Antonov-stable steady states of the
form~\eqref{E:MICRO}.
The main motivating question for us is whether one can find linear oscillatory solutions
(normal modes) of the linearized dynamics around the above steady states. 
This asks for a
refined understanding of the linearized operator and we shall address it combining ideas
from spectral theory and Hamiltonian mechanics.

\subsection{Main results}
In this section we focus the discussion on the harder, spherically symmetric situation,
but emphasize that similar results are obtained in the planar setting as well.

\subsubsection*{The connection between eigenvalues and pulsating modes}
Linearizing the Vlasov-Poisson system about a fixed radial steady state translates
the existence of linearly oscillating modes
into the existence of positive eigenvalues of $\A$ with eigenfunction odd in $v$,
and different ways of linearizing yield different interpretations of oscillating modes.
The classical Eulerian way (Section~\ref{ssc:euler}) explains the oscillations
of the kinetic and potential energy respectively,
while the use of Mass-Lagrange and Lagrangian coordinates (Sections~\ref{ssc:masslagrange}
and~\ref{ssc:lagrange})
yields the oscillation of the support of the solution,
i.e., the pulsating behavior discussed above.
In all cases, an eigenvalue $\lambda>0$ of $\A$ and the period $P$
of the corresponding oscillation/pulsation are related via
\begin{align}\label{eq:periodeigenvalue}
	P=\frac{2\pi}{\sqrt\lambda}.
\end{align}
In the context of pulsating modes for the Euler-Poisson system and the Cepheid variable
stars the period of the
pulsation is related to the central density of the corresponding steady state
by the so-called Eddington-Ritter relation $P \rho(0)^{1/2} = const$, cf.\ \cite{edd,mak,ross}, and in
Section~\ref{ssc:eddritter} we establish an analogous result for the Vlasov-Poisson system
and thus for the pulsations of galaxies, which was observed numerically in \cite{RaRe2018} at the non-linear level.

\subsubsection*{The essential spectrum}
Next we provide a sharp description of the essential spectrum of $\A$. It is shown in Theorem~\ref{essspecA} that $\B$ is relatively compact to $\D^2$ and therefore the essential spectrum of $\A=-(\D^2+\B)$ coincides with the essential spectrum of $-\D^2$. A fundamental ingredient in our analysis are the {\em action-angle variables}. For any choice of the energy level $E$ and the angular momentum $L$, the corresponding particle motion in the gravitational well defined by the effective potential $\Psi_L$, cf.~\eqref{psildef}, is periodic, and the action variable $\theta\in\mathbb S^1$ parametrizes the phase of the oscillation along the particle orbit.
The radial particle periods are given by the {\em period function}
\be
T(E,L)\coloneqq\sqrt2 \int_{r_-(E,L)}^{r_+(E,L)} \frac1{\sqrt{E-\Psi_L(r)}} \diff r,
\ee
where $r_-(E,L)<r_+(E,L)$ solve $\Psi_L(r_\pm(E,L))=E$, see Lemma~\ref{effpot}.
The corresponding phase space change of variables 
\[
(r,w,L) \mapsto (\theta, E, L)
\]
is not quite volume preserving---as it would be in the case of actual action-angle variables \cite{Arnold,LaLi,LB1994}, but we use this terminology anyway---but instead satisfies $\diff r\diff w\diff L = T(E,L)\,\diff \theta\diff E\diff L$. In action-angle variables the transport operator $\D$ transforms into the particularly simple form
\begin{align}\label{E:TRANSPORTANGLE}
\D = \frac1{T(E,L)} \pa_\theta,
\end{align}
see Lemmas~\ref{trans1dsmooth} and~\ref{trans1d} for more precise statements, and therefore the Antonov operator~\eqref{E:ABDEF} can be rewritten as
\begin{align}\label{E:ANTONOVAA}
\A = - \frac1{T(E,L)^2}\pa_{\theta}^2 -  \B.
\end{align}
It is tempting to express the gravitational \enquote{response} operator $\B$
in the new variables as well, but the resulting expression does not lead to any clear insights.
This antagonism between transport
and gravity is well explained in Lynden-Bell's notes~\cite{LB1994} and also mentioned in \cite{BiTr}.
Using~\eqref{E:TRANSPORTANGLE}, we prove in Theorems~\ref{transsquarespectrum} and~\ref{essspecA} that
\begin{align}\label{E:ESSENTIALSPECTRUM}
\sigma_{ess}(\A)=\sigma_{ess}(-\D^2) = \overline{ \left(\frac{2\pi\N_0}{T(\mathring\Omega_0^{EL})}\right)^2 }= \overline{ \left\{  \frac{4\pi^2k^2}{T^2(E,L)}  ~\Big|~ k\in\N_0,~(E,L)\in\mathring\Omega_0^{EL} \right\} },
\end{align}
where $\mathring\Omega_0^{EL}$ is the interior of the $(E,L)$-support of the steady state. 
Related results regarding the essential spectrum of $\D$ instead of $-\D^2$ were stated previously in the physics literature, see e.g.~\cite{Ma}.
By Proposition~\ref{Tbounded}, the period function $(E,L)\mapsto T(E,L)$ is bounded from above and bounded away from zero on the support of the steady state, and thus \eqref{E:ESSENTIALSPECTRUM} in particular shows that the essential spectrum contains a gap between the $0$-eigenvalue and the value $\frac{4\pi^2}{\sup^2(T)}$. Following Mathur~\cite{Ma}, we refer to this gap as the {\em principal gap} and denote it by
\begin{align}\label{E:PRINCIPALGAP}
\mathcal G \coloneqq \left]0, \frac{4\pi^2}{\sup_{\mathring\Omega_0^{EL} }^2(T)}\right[.
\end{align}
%
%
Remarkably, the structure of the essential spectrum is entirely encoded in the properties of the period function $T(E,L)$. Another simple consequence of~\eqref{E:ESSENTIALSPECTRUM} is that further gaps in the spectrum are possible depending on the relative size of the maximum and the minimum of the period function on the steady state support, see Remark~\ref{essspecform}.
The gap structure is also mentioned in the physics literature, see~\cite{BiTr} and references therein. 

In the plasma-physics context, action-angle variables have been recently used in the important work of Guo and Lin~\cite{GuLi2017}, where the spectrum of the linearized operator around a certain class of (space-inhomogeneous) BKG steady states was completely described. We also mention a recent work of Pausader and Widmayer~\cite{PaWi2020} on small data global-in-time asymptotics for the Vlasov-Poisson system with a point charge, wherein action-angle variables are also used.

\subsubsection*{The gap in the spectrum}

It is natural to ask if there can exist eigenvalues of $\A$ in the principal gap $\mathcal G$.
In particular, \eqref{eq:periodeigenvalue} shows that an eigenvalue in $\mathcal G$ corresponds to an oscillating mode with period larger than $\sup(T)$, i.e., larger than the supremum of the radial particle periods in the steady state configuration.

One obvious attempt is to obtain such an eigenvalue via a variational principle, more precisely,
to minimize the quadratic form~\eqref{E:QUADRATICFORM} over the set of $g$ orthogonal to
$\ker(\A)$ satisfying $\|g\|_H=1$.
Fortunately, Antonov's coercivity bound---as stated in Proposition~\ref{antonovcoerc}---yields a complete characterization of the null-space of $\A$; it coincides with the null-space of $\D$, see Corollary~\ref{antonovev0}. Due to Jeans' theorem, $\ker(\D)\subset\Ltwo$ consists of all functions of $E$ and $L$ that belong to $H$ \cite{BaFaHo86,GuRe2007,ReSt20,St19}, a statement which also follows from the formula~\eqref{E:TRANSPORTANGLE}. Furthermore, Lemma~\ref{transinversetauel} implies that the range of $\D$ is closed, i.e., the orthogonal complement of the kernel of $\D$ equals $\mathrm{im}(\D)$. 

Nonetheless, this minimization problem is in general hard. In Theorem~\ref{antonovgap} we can show that 
\begin{align}\label{eq:variationalintro}
	\inf_{g\in\ker(\A)^{\perp}\setminus\{0\}} \frac{\langle\A g,g\rangle_{\Ltwo}}{\|g\|_{\Ltwo}^2}
\end{align}
is positive by considering an intermediate variational problem, cf.~Proposition~\ref{helpervariation}. The positivity of \eqref{eq:variationalintro} follows immediately from Antonov's coercivity bound~\eqref{eq:antonovcoerc} in the case of an isotropic model, but is harder to obtain for polytropic steady states with an inner vacuum region. This proves the existence of a gap at the origin in the spectrum of $\A$ and particularly shows that no point spectrum can accumulate at $0\in\sigma_{ess}(\A)$.

Unfortunately, we have neither quantitative control on the size of the infimum
\eqref{eq:variationalintro} nor do we obtain a minimizer. We therefore do not pursue
this argument further to decide
whether there exists an eigenvalue of $\A$ inside the principal gap $\mathcal G$.

\subsubsection*{The Birman-Schwinger principle and the existence of oscillating modes}

To address the problem of the existence of eigenvalues we resort to a version of the  Birman-Schwinger principle. The latter has been used classically to prove the existence of bound states below the essential spectrum of Schr\"odinger operators \cite{LiLo01} and one is tempted to apply these methods to~\eqref{E:ANTONOVAA}. We adopt the approach developed by Mathur~\cite{Ma}, originally used to find normal modes for the linearized Vlasov-Poisson system in the presence of an external potential. 

We restrict the operators to $\Ltwo^{odd}$---the subspace of $\Ltwo$ consisting of functions odd in $v$/$w$---since only spherically symmetric, odd-in-$v$ eigenfunctions yield the existence of an oscillating mode. Now it is easily checked that $\lambda\in\mathcal G$ is an eigenvalue of $\A$ if and only if $1$ is an eigenvalue of the operator 
\[
Q_\l = \B\, \left(-\D^2 - \lambda \right)^{-1} \colon \Ltwo^{odd} \to \Ltwo^{odd}.
\]
Lemmas~\ref{L:BS} and~\ref{L:BSONE} further show that the existence of some $\lambda\in\mathcal G$ such that $Q_\lambda$ possesses an eigenvalue in $[1,\infty[$ implies that $\A$ has an eigenvalue in the principal gap $\mathcal G$. 
The operator $Q_\l$ is however not easy to analyze, but the crucial simplification comes from the fact that the image of $\B$
is of the form $\vert\varphi'(E,L)\vert\,w\,F(r)$ for some function $F$. The latter \enquote{separation-of-variables} property allows us to restrict the operator $Q_\l$ to the subspace of functions of the form $|\varphi'(E,L)|\,w\,F(r)$ with $F$ in a suitable Hilbert space $\mathcal F$. We refer to the resulting operator 
\be
\M_\l:\F\to\F
\ee 
as the {\em Mathur operator} and refer to Definition~\ref{defmathur} for a rigorous definition.
By Lemma~\ref{spectraequivalent}, any eigenvalue of $\M_\l$ gives an eigenvalue of $Q_\l$. 
Furthermore, Proposition~\ref{Klambdadef} yields an integral kernel representation of the Mathur operator first proposed in~\cite{Ma}, and we show that $\M_\lambda$ is indeed a symmetric Hilbert-Schmidt operator---when considered on the \enquote{right} function space---and that its spectrum is non-negative. Thus, the largest element in the spectrum of the Mathur operator is given by
\begin{align}
	M_\lambda= \|\M_\lambda\|_{\F_1\to\F_1} = 
	\sup\left\{ \langle G,\M_\lambda G\rangle_{\mathcal F_1} \mid G\in\mathcal F_1,\,\|G\|_{\mathcal F_1}=1\right\}
\end{align}
for a suitably chosen function space $\mathcal F_1$, and $M_\lambda$ is an eigenvalue of $\M_\lambda$.
We therefore arrive at the next key result of this work---Theorem~\ref{T:CRITERION}. It states that the existence of a $\l\in\mathcal G$ such that 
\begin{align*}
	M_\lambda>1
\end{align*}
is equivalent to the linearized operator $\A$ having an eigenvalue within the principal gap $\mathcal G$. 

\subsubsection*{Examples of steady states with linearly oscillating solutions}

Finally, we prove in Theorems~\ref{T:EXAMPLEPLANAR} and~\ref{T:EXAMPLERADIAL} that
\begin{itemize}
\item  there exist classes of planar steady states such that the associated linearized
  operator has an eigenvalue in the principal gap;
\item assuming that the maximal value of the period function on the steady state support
  is attained at the maximal energy and minimal angular momentum, i.e.,
  \begin{align}\label{E:PERIODMONINTRO}
    \sup_{\mathring\Omega_0^{EL}}(T) = T(E_0,L_0),
  \end{align}
  and that $T$ is sufficiently regular~\eqref{E:Tdiffbar}, there are classes of
  spherically symmetric steady states such that there exists an eigenvalue in the
  principal gap for the associated linearized operator.
\end{itemize}
In fact, the assumption~\eqref{E:PERIODMONINTRO} is implied by the stronger {\em 
monotonicity} assumptions 
\be\label{E:MONASS2}
\pa_E T\geq0 \ \text{ and }\ \pa_L T\leq0 \  \text{ on } \mathring\Omega_0^{EL},
\ee
which are expected to hold for a wide class of steady states studied in this paper. 
The assumption~\eqref{E:MONASS2} and therefore~\eqref{E:PERIODMONINTRO} has been verified numerically, and it will be rigorously established in future work. 
In general, monotonicity properties of the period function are 
an important topic in dynamical systems, especially in connection to bifurcation theory for
Hamiltonian dynamical systems~\cite{Ch85,ChWa86,HaKo1991}. Monotonicity of the period function
plays an important role for phase mixing (see e.g.~\cite{RiSa2020}). However, in the context of
verifying the criterion of Theorem~\ref{T:CRITERION}, it is used in an entirely different
way---and for the opposite purpose.
Section~\ref{sc:T} in the appendix is devoted to the properties of the period function in
the radial setting.

\subsection{Background and motivation}
It is an important question in astrophysics to understand the oscillatory behavior of
galaxies~\cite{BiTr,LoGe1988}. We cite here directly
from~\cite[Section~5.5.3]{BiTr}: \enquote{\em Stability theory addresses only one aspect of
  how a stellar system responds to external forces. A natural next question is whether a stable
  stellar system can sustain undamped oscillations, that is, can it ring like a bell?}

Of course, the question arises what the exact nature of these oscillations is
(besides being undamped), and we recall the numerical investigations \cite{AR,Gue_e_a,RaRe2018}
which show that these oscillations are periodically repeated expansions and contractions
of the configuration in (phase) space.
We are aware of only one example of such time-periodic solutions to the Vlasov-Poisson
system \cite{Ku78}. In Section~\ref{sc:kurth} we remind the reader of the spherically symmetric
Kurth solutions and introduce a new class of plane symmetric Kurth-type solutions.
The Kurth flows are essentially a kinetic generalization of periodically oscillating
self-gravitating incompressible gas balls. 
The phase-space density of the Kurth solution blows up at the boundary of its support,
which raises doubts as to the physical relevance of such solutions. 
In view of the spectral analysis from Section~\ref{sc:ess}, the Kurth steady state is distinguished by the fact that the particle period is constant on the whole steady state support.
Comparing the oscillation period of time-periodic Kurth solutions close to the Kurth steady state with the essential spectrum, we see that these solutions indeed ``correspond" to an eigenvalue within the principal gap.
Observe however that the time-periodic Kurth solutions are solutions to the non-linear
Vlasov-Poisson system, but we neglect the transition from non-linear to linear regime at 
this point as well as the fact that the Kurth steady states do not satisfy the assumptions
required for the results in Section~\ref{sc:ess} in order to gain an understanding of the possible structure of the spectrum of $\A$.

Linearizing the flow around a given steady state, it is natural to look for linearly oscillating modes.
In Section~\ref{S:LINEARISATION} we present a detailed derivation of the Antonov operator from different
view-points, which will in particular allow us to interpret the normal modes as the oscillations of
the support of the galaxy.
The goal is to compute the \enquote{response} \cite[Chapter~5]{BiTr} of the spherical or
the plane symmetric steady state. Using action-angle variables, one can formally derive a dispersion relation, yet it is  rather complicated. Important contributions in the physics literature
in this direction were among others made by Kalnajs~\cite{Kalnajs1971,Kalnajs1977},
Fridman and Polyachenko~\cite{FrPo1984}, Louis~\cite{Louis1992}, Mathur~\cite{Ma},
Vandervoort~\cite{Vandervoort2003}, Weinberg~\cite{Weinberg1991,Weinberg1994}.
We refer the reader to~\cite{BiTr} for more background on the problem.

Another impetus for the study of the existence of normal modes is
to understand the (in)validity of the so-called linear Landau damping.
We observe that the steady states~\eqref{E:MICRO} are not spatially homogeneous,
which significantly complicates the linearized dynamics with respect to perturbations of spatially
homogeneous steady states. In fact, spatial homogeneity is not a natural assumption for isolated self-gravitating systems, but it is by contrast often studied in the plasma physics context. It was observed by Landau~\cite{Landau1946} that, despite the absence of dissipation, small perturbations of space-homogeneous steady plasmas can lead to asymptotic-in-time equilibriation of macroscopic observables---e.g.\ the macroscopic density---at the linearized level. This process is related to phase mixing and it is referred to as Landau damping. It was shown to hold nonlinearly by Mouhot and Villani~\cite{MoVi2011}, see also the work of Bedrossian, Masmoudi, and Mouhot~\cite{BeMaMo2016} and a recent proof by Grenier, Nguyen, and Rodnianski~\cite{GrNgRo2020}.  

By contrast, realistic galaxies are spatially inhomogeneous and the linearized dynamics combines
the phase-mixing effects with the gravitational response of the background. This can be viewed as the interaction between the operators $\D^2$ and $\B$ in the definition~\eqref{E:ABDEF} of $\A$.
In order to make sense of linear Landau damping around steady galaxies of the form~\eqref{E:MICRO}, we must first mod out the perturbations that belong to the null-space of the linearized operator if any decay is to happen. In the pioneering contribution of Lynden-Bell~\cite{LB1962} the author considered this question, however neglecting the crucial gravitational self-interaction term $\B$ in the linearized operator. He argued informally that there exists a mechanism for the damping of macroscopic, i.e., velocity-averaged, quantities, such as the spatial density. This is to our knowledge the first instance where action-angle variables were used to gain intuition about the phase-mixing process. We also mention the important early work of Antonov~\cite{An1961} and subsequent theory of violent relaxation by Lynden-Bell~\cite{LB1967}, where the mixing processes plays a crucial role. Since then, a lot of work has been devoted to a better understanding of phase-mixing/gravitational Landau damping in the physics literature, see the textbook by Binney and Tremaine~\cite{BiTr} for an overview. Recently, Rioseco and Sarbach~\cite{RiSa2020} studied phase-mixing for Vlasov equations in a {\em given} external potential using action-angle variables.
In light of this discussion, our results in Section~\ref{sc:mathur} imply that
{\em no} linear Landau damping occurs around a large class of plane symmetric steady
states and the same conclusion holds in the spherically symmetric case under the additional assumption~\eqref{E:PERIODMONINTRO} on the period function. This does not mean that such modes are not damped at the non-linear level, but such non-linear damping will work on time scales
which are much longer than the oscillation periods.
It must involve new mechanisms---the original Landau damping {\em is} already present at the linear
level---and it should be a challenging topic for future work.

Furthermore, as stated above, numerical simulations \cite{AR,Gue_e_a} indicate that the pulsating behavior is also present in the context of the asymptotically flat Einstein-Vlasov system and many of our methods extend to this case~\cite{GRS}.


\subsection{Outline of the paper} 
In Section \ref{sc:stst} we state the steady state class, consisting of spherically symmetric and plane symmetric ones, and also introduce some key quantities as the period function. In the following we always treat the radial and planar case separately. In Section \ref{S:LINEARISATION} we linearize the Vlasov-Poisson system around such steady states, which translates the existence of oscillating modes into the existence of eigenvalues of a certain operator, and derive an Eddington-Ritter type scaling law for the eigenvalues in Subsection~\ref{ssc:eddritter}.
This linearized operator is then carefully defined in Section \ref{sc:operator},
its essential spectrum is explicitly determined in Section \ref{sc:ess}. To gain an intuition of where the eigenvalues of the linearized operator can be located relative to the essential spectrum, we consider semi-explicit Kurth-type solutions of the Vlasov-Poisson system in Section \ref{sc:kurth}. In Section \ref{sc:gap} we use Antonov's coercivity bound to characterize the nullspace of the operator and show that the spectrum possesses a gap at the origin. This is needed for the Birman-Schwinger principle we apply in Section \ref{sc:mathur} to derive criteria for the existence of oscillating modes. Finally, we apply the aforementioned criteria in Section~\ref{ssc:examples} to obtain examples of linearly oscillating modes in the radial and planar setting---in the former setting
under the assumptions~\eqref{E:PERIODMONINTRO} and~\eqref{E:Tdiffbar} on $T$.

Appendix \ref{sc:potential} is devoted to the properties of gravitational potentials induced by dynamically accessible perturbations, which are used throughout the whole paper. In the second part of the appendix, Appendix \ref{sc:T}, we summarize the properties of the period function $T$ in the radial setting.

\vspace*{.5cm}\noindent
During the writing of this manuscript we have been informed that
related results were obtained independently by M.~Kunze \cite{MK}.

\vspace*{.5cm}
\noindent
{\bf Acknowledgments.}
M.~Had\v zi\'c's research is supported by the EPSRC Early Career Fellowship EP/S02218X/1.
The authors thank Thomas~Kriecherbauer, Alexander~Pushnitski, and Manuel~Schaller for helpful discussions.

\section{Steady states}\label{sc:stst}

\subsection{Spherically symmetric steady states}\label{ssc:ststradial}

We consider steady states of the three-dimensional
Vlasov-Poisson system~\eqref{E:FEQN}--\eqref{E:RHOEQN}
with boundary condition~\eqref{E:AF}
of the form
\begin{align*}
f_0 = \varphi(E,L),
\end{align*}
where $\varphi\colon\R\times\R\to[0,\infty[$ is a suitable ansatz function,
$E$ is the particle energy induced by the stationary potential $U_0=U_0(x)$ of the
steady state, i.e., 
\begin{align*}
E=E(x,v) = \frac12\vert v\vert^2 + U_0(x)
\end{align*}
as above, and $L$ is the modulus of the angular momentum squared defined in \eqref{eq:rwL}.
The particle energy $E$ is conserved along characteristics of the Vlasov equation,
provided $U_0$ is time-independent, while $L$ is conserved, provided $U_0$ is spherically
symmetric.
The stationary Vlasov-Poisson system is then reduced to the following equation
for the potential:
\begin{align}\label{eq:statVP}
\Delta U_0(x)=4\pi\int_{\R^3}
\varphi\left(\frac12\vert v\vert^2+U_0(x),\vert x\times v\vert^2\right)
\diff v\text{ for }x\in\R^3, \quad \lim_{\vert x\vert\to\infty}U_0(x)=0.
\end{align}
In the isotropic case where by definition $\varphi$ depends only on the particle energy
$E$, every solution $U_0\in C^2(\R^3)$ of this equation is spherically symmetric,
cf.~\cite{GiNiNi79}, while this symmetry must be assumed a-priori when $\varphi$
depends also on $L$.

As for the ansatz function, we focus on the polytropes
\begin{align}\label{E:POLYTROPE}
\varphi(E,L) = (E_0-E)^k_+ (L-L_0)_+^l,
\end{align}
and the King model 
\begin{align}\label{E:KING}
\varphi(E,L) = \varphi(E) = (e^{E_0-E}-1)_+,
\end{align}
both of which play a prominent role in the astrophysics literature, cf.~\cite{BiTr}; we recall~\eqref{E:+conv}.
In both cases, $E_0<0$ is a cut-off energy, which is necessary in order that
the resulting steady state has compact support and finite mass.
In the polytropic case, $L_0\geq0$ gives a lower bound for the angular momentum.
In particular, $L_0>0$ leads to a steady state with a vacuum region at the center.
In this case the parameters $k>0$ and $l>-1$ have to be chosen such that $k<l+\frac72$
in order for a steady state to exist and have finite mass and compact support, and we also require $k+l+\frac12\geq0$.
In the case of no vacuum region, i.e., $L_0=0$, we restrict ourselves to $l=0$, i.e., $L$-independent isotropic models; we use the convention $0^0=1$ 
in \eqref{E:POLYTROPE} in this case. Again, $0<k<\frac72$.
For the existence of steady states under the above (and more general) assumptions
we refer to \cite{RaRe13} and the references there.

Now let 
\begin{align*}
{\Omega_0}\coloneqq\{(x,v)\in\R^3\times\R^3\mid f_0(x,v) > 0\}
\end{align*}
be the (interior of the) support of the steady state in $(x,v)$-coordinates. 
For the steady states mentioned above $\Omega_0$ is open and bounded;
by $R_0\coloneqq\sup\{\vert x\vert\mid(x,v)\in\Omega_0\}\in]0,\infty[$ we denote the maximal occurring radius in the steady state support.
We add an upper index when expressing this set in different coordinates:
\begin{align*}
{\Omega_0^r} &\coloneqq \{ (r,w,L)\in\R^3\mid \exists(x,v)\in\Omega_0
\colon r=\vert x\vert,~ w=\frac{x\cdot v}r, L=L(x,v) \},\\
{\Omega_0^{EL}} &\coloneqq \{ (E,L)\in\R^2 \mid \exists(x,v)\in\Omega_0
\colon E=E(x,v),~L=L(x,v) \}.
\end{align*}
The derivative ${\varphi'}\coloneqq\partial_E\varphi$ exists on $\mathring\Omega_0^{EL}$ with
\begin{align}\label{eq:linstab}
\varphi'<0 \text{ on } \mathring\Omega_0^{EL},
\end{align}
which is the usual condition for linear or non-linear stability of the steady state,
encountered both in the astrophysics and in the mathematics literature,
cf.~\cite{BiTr,Re07} and the references there.
Here, $\mathring\Omega_0^{EL}$ denotes the interior of the set $\Omega_0^{EL}$.

All of the following can be done for a much larger class of steady states.
In fact, it is only essential that the ansatz function satisfies the conditions of the
existence theory \cite{RaRe13} and that the steady state satisfies the stability
condition \eqref{eq:linstab}. Furthermore, for the spectral analysis we require that
\begin{align}\label{eq:A7prime}
\int_{\R^3} \vert\varphi'(E(x,v),L(x,v))\vert \diff v \leq C
\end{align}
for some $C>0$ independent of $x$, where we extend $\varphi'=\partial_E\varphi$ by $0$
to the whole space. 
The assumption~\eqref{eq:A7prime} however is of mere technical nature and it is expected that it can be relaxed.
In the case of an isotropic steady state, i.e.,
$\varphi(E,L)=\varphi(E)$, \eqref{eq:A7prime} follows if  
\begin{align}\label{eq:critiso}
\int_{U_0(0)}^{E_0}\vert\varphi'(E)\vert\diff E < \infty,
\end{align}
cf.~\cite{BaMoRe95}. For polytropic ansatz functions, our choice of parameters also yields
\eqref{eq:A7prime}, since
\begin{align}
\int_{\R^3} \vert\varphi'(E(x,v),L(x,v))\vert \diff v
= c_{k,l}\,r^{2l}\left(E_0-U_0(r)-\frac{L_0}{2r^2}\right)^{k+l+\frac12}_+
\end{align}
for $x\in\R^3\setminus\{0\}$ and $r=\vert x\vert$, where $c_{k,l}>0$ is some constant
depending on the indices $k$ and $l$. We refer to \cite{RaRe13,RaRe2018,Re99} where similar
calculations have been performed in order to represent $\rho_0$ in terms of $E_0-U_0$;
note that only the exponents differ by $1$ in the formulas for $\rho_0$ and
$\int \vert\varphi'(E,L)\vert \diff v$.

An important quantity for the analysis of spherically symmetric
steady states of the Vlasov-Poisson system is the effective potential
\begin{align}\label{eq:defeffpot}
{\Psi_L} \colon \left] 0, \infty \right[ \to \R ,~ \Psi_L(r) \coloneqq U_0(r) + \frac L{2 r^2},
\end{align}
where $L>0$ and we identified $U_0(x)=U_0(\vert x\vert)$. We claim the following properties:

\begin{lemma} \label{effpot}
  \begin{enumerate}[label=(\alph*)]
  \item For any $L>0$ there exists a unique ${r_L} > 0$ such that 
    \begin{align*}
      \min_{\left]0, \infty \right[}( \Psi_L ) = \Psi_L(r_L) < 0.
    \end{align*}
    Moreover, the mapping $\left]0, \infty \right[ \ni L \mapsto r_L$ is continuously differentiable.
  \item For any $L>0$ and $E \in \left] \Psi_L (r_L) , 0 \right[$
    there exist two unique radii ${r_\pm (E,L)}$ satisfying
    \begin{align*}
      0 < r_-(E,L) < r_L < r_+(E,L) < \infty
    \end{align*}
    and such that $\Psi_L (r_\pm(E,L)) = E $. In addition, the functions
    \begin{align*}
      \{ (E,L) \in \left]- \infty, 0 \right[ \times \left]0, \infty \right[ \mid
    \Psi_L(r_L) < E \} \ni (E,L) \mapsto r_\pm(E,L)
    \end{align*}
    are continuously differentiable.
  \item For any $L>0$ and $E \in \left] \Psi_L (r_L) , 0 \right[$,
    \begin{align}
      r_+(E,L) < - \frac{M_0}{E},
    \end{align}
    where $M_0:=\|f_0\|_1\in \,]0,\infty[$ denotes the total mass of the steady state.
  \item For any $L>0$,
    $E \in \left] \Psi_L (r_L) , 0 \right[$ and
    $r \in \left[ r_-(E,L) , r_+(E,L) \right]$
    the following estimate holds:
    \begin{align}
      E - \Psi_L(r) \geq L \, \frac{(r_+(E,L) - r) \, (r - r_-(E,L))}{2 r^2 r_-(E,L)  r_+(E,L)}.
    \end{align}
  \end{enumerate}
\end{lemma}
\begin{proof} We refer to \cite{GuRe2007,LeMeRa11} or more recently,
  \cite{ReSt20,St19}. In these references, similar and further properties have been shown
  for other classes of steady states. However, the proofs only depend on the properties of
  the stationary potential $U_0$ and can therefore be adapted word by word. \end{proof} 

Note that the effective potential appears in the particle energy when expressed in
$(r,w,L)$-coordinates:
\begin{align*}
E=E(r,w,L) = \frac12w^2 + \frac L{2r^2} + U_0(r) = \frac12w^2 + \Psi_L(r).
\end{align*}
Therefore, for fixed $L>0$, the particle trajectories of the steady state $f_0$ are governed by the characteristic system
\begin{align*}
\dot r=w,\qquad \dot w=-\Psi_L'(r).
\end{align*}
Let $\R \ni t \mapsto (r(t), w(t), L)$ be an arbitrary global solution of this system. Since the particle energy is conserved along these characteristics, there exists $E \in\R$ such that $E = E(r(t), w(t), L)$ for all $t\in\R$. We assume that the solution satisfies $\Psi_L(r_L)<E < 0$, otherwise it is not of interest. For any $t\in\R$ we then have
\begin{align*}
\Psi_L (r_L) \leq \Psi_L (r(t)) \leq \frac 12 w^2 (t) + \Psi_L (r(t)) = E
\end{align*}  
and thus $r_-(E,L) \leq r(t) \leq r_+ (E,L) $. Furthermore, solving for $w$ yields
\begin{align*}
\dot r (t) = w(t)  = \pm \sqrt{ 2 E - 2 \Psi_L (r(t)) }
\end{align*}
for $t\in\R$.  Therefore, $r$ oscillates between $r_-(E,L)$ and $r_+(E,L)$, where $\dot r = 0$
is equivalent to $r=r_\pm (E,L)$ and $\dot{r}$ always switches its sign when reaching $r_\pm(E,L)$.
By applying the inverse function theorem and integrating, we obtain that the period of the
$r$-motion, i.e., the time needed for $r$ to travel from $r_-(E,L)$ to $r_+(E,L)$ and
back to $r_-(E,L)$, is given by following expression:
\begin{defn}\label{defT}
  For $L>0$ and $\Psi_L(r_L)<E<0$ let
  \begin{align}\label{eq:defT}
    {T(E,L)} \coloneqq 2 \int_{r_-(E,L)}^{r_+(E,L)} \frac{\diff r}{\sqrt{2E - 2\Psi_L(r)}},
  \end{align}
  which is referred to as the period function of the steady state.
\end{defn}
Using Lemma~\ref{effpot} (c), (d), it can be shown that the above integral is finite with
\begin{align}\label{eq:Tboundnaive}
T(E,L) \leq 2\pi \frac{M_0^2}{E^2\sqrt L},
\end{align}
see \cite{ReSt20,St19} for a detailed proof.
We only consider $T$ on the interior of $\Omega_0^{EL}$,
since the boundary may contain points with $E=\Psi_L(r_L)$, i.e., $r_\pm$
may not be defined there. However, it is easy to see that $\partial \Omega_0^{EL}$
is a set of measure zero and therefore not of interest later on.

We shall see that the spectrum of the operator $\A$ is closely connected to the period
function $T$. In particular, we establish criteria for the existence of linearly oscillating
galaxies which solely depend on the behavior of the period function of the steady state.
A detailed discussion of the properties of $T$ can therefore be found in Appendix~\ref{sc:T}.

\subsection{Plane symmetric steady states}\label{ssc:ststp}

In the plane symmetric case, we look for stationary solutions of the
Vlasov-Poisson system in the form \eqref{pl_vlasov}--\eqref{pl_rho},
and we recall that for this symmetry class, $x\in \R$ and $v=(v_1,\bar v)\in \R^3$,
cf.~Subsection~\ref{ssc:symmetries}.

The conserved quantities associated with the characteristic flow are  $v_2$, $v_3$,
and the energy 
\begin{align}
\pE(x,v_1) = \frac12v_1^2 + \pU_0(x),
\end{align}
where $\pU_0\colon\R\to\R$ is the stationary potential.
We seek steady states of the form 
\begin{align}\label{eq:iso1D}
\pf_0(x,v)
= \pvarphi(\pE, \pv) = \palpha(\pE)\,\pbeta(\pv),\quad (x,v)\in\R\times\R^3.
\end{align}
This ansatz turns the mass density into a functional of
the potential $\pU_0$,
\begin{align}
  \prho_0(x)
  \coloneqq& \int_{\R^3} \pvarphi\left(\frac12v_1^2+\pU_0(x),\pv\right) \diff v
  \nonumber\\
  =& 2\int_{\R^2}\pbeta(\pv)\diff\pv \;
  \int_{\pU_0(x)}^\infty \frac{\palpha(\pE)}{\sqrt{2\pE-2\pU_0(x)}} \diff\pE
  \eqqcolon h(\pU_0(x)),\quad x\in\R,\label{eq:defh1D}
\end{align}
and the stationary Vlasov-Poisson system is reduced to the equation
\begin{align}\label{eq:1DstatVP}
\pU_0'' = 4\pi h(\pU_0) \text{ on }\R.
\end{align}
Solutions of this equation resulting in steady states with compact support and finite mass
are much easier to obtain than in the spherically symmetric setting,
and we briefly outline the arguments. 
The only requirement on the ansatz function is that the resulting function $h$ in
\eqref{eq:defh1D} is $C^1$, vanishes on $[\pE_0,\infty[$ and is positive and decreasing
on $]-\infty,\pE_0[$,
where $\pE_0$ is some cut-off energy. 
We assume  $\pbeta$ to be continuous with compact support and
\begin{align}\label{E:BETABARASSUMPTION}
\int_{\R^2}\pbeta(\pv)\diff\pv=1,
\end{align}
and for the sake of definiteness and simplicity we require $\palpha$ to be either polytropic
\begin{align}\label{eq:poly1d}
\palpha(\pE) = (\pE_0-\pE)^k_+ 
\end{align} 
for some $k>\frac12$ or of King-type, i.e.,
\begin{align}\label{eq:king1d}
\palpha(\pE) = \left(e^{\pE_0-\pE}-1\right)_+.
\end{align}
Again, $(\ldots)_+$ denotes the positive part. 
As in \cite{RaRe13}, it is convenient to reformulate the problem in terms
of $y\coloneqq \pE_0-\pU_0$. Let $\pPhi(\eta)\coloneqq\palpha(\pE_0-\eta)$,
i.e., $\eta=\pE_0-\pE$. Then $y$ solves
\begin{align}\label{eq:1DstatVPy}
y''= - 4\pi \tilde h(y),
\end{align}
where
\begin{align}
\tilde h(z)\coloneqq 2\int_0^z \frac{\pPhi(\eta)}{\sqrt{2z-2\eta}} \diff\eta,\quad z\in\R.
\end{align}
In order to see the required regularity of $\tilde h$  we rewrite it, using
integration by parts:
\begin{align}
\tilde h(z) = -2 \int_0^z \pPhi(\eta)\;\partial_\eta
\left[\sqrt{2z-2\eta}\right]\diff\eta
= 2 \int_0^z \pPhi'(\eta)\,\sqrt{2z-2\eta}\, \diff\eta,\quad z>0.
\end{align}
Now $\tilde h$ has the same form as in the spherically symmetric case
(cf.~\cite{RaRe13}), with the sole difference that the microscopic equation of
state $\Phi$ contains a derivative under the integral sign.
For our two examples \eqref{eq:poly1d} and \eqref{eq:king1d}, 
$\tilde h\in C^1(\R)$, $\tilde h$
is strictly increasing on $[0,\infty[$ and $\tilde h=0$ on $]-\infty,0]$; for the
polytropic case \eqref{eq:poly1d}, $\tilde h(z)= c_k z_+^{k+1/2}$. We define
\[
H(z) \coloneqq 4\pi \int_0^z \tilde h(s)\diff s
\]
and observe that
\begin{equation}\label{y_levels}
\frac{1}{2} (y')^2 + H(y)
\end{equation}
is a conserved quantity for the autonomous, planar system corresponding to
\eqref{eq:1DstatVPy} in the $(y,y')$-plane. The form of the level sets of this
conserved quantity implies immediately that any non-trivial solution $y$
to \eqref{eq:1DstatVPy} exists globally on $\R$, and there exists a unique
$x^\ast\in \R$ such that $y'(x^\ast)=0$ and $y(x^\ast) >0$; in accordance with
the reflection symmetry contained in \eqref{pl_symm_def} we take
$x^\ast =0$. Since $y(-\cdot)$ is a solution of \eqref{eq:1DstatVPy}
with the same data at $x=0$ it follows that $y(-x) = y(x)$; $y$ is even in $x$
as required by  \eqref{pl_symm_def}.
The form of the level sets of the conserved quantity \eqref{y_levels}
implies that the limits $\lim_{x\to\infty} y'(x) = - \lim_{x\to-\infty} y'(x)$ exist,
and $\lim_{x\to\pm\infty} y(x)=-\infty$. Hence there exists $\pR_0>0$ such that
$\prho_0 \coloneqq \tilde h(y)$ has compact support $[-\pR_0,\pR_0]$,
and
\[
\lim_{x\to\infty} y'(x) = \int_0^\infty y''(x)\diff x = - 4 \pi \int_0^\infty \prho_0(x)\diff x
\eqqcolon- 2 \pi \pM_0;
\]
the non-trivial solutions of \eqref{eq:1DstatVPy} can be uniquely parametrized by
the mass $\pM_0>0$ of the resulting steady state. It remains to recover $\pU_0$
from $y$.
At this point we recall that in the plane symmetric case
the usual boundary condition~\eqref{E:AF} at spatial infinity makes no sense, and instead,
$\pU_0$ is to obey \eqref{pl_poisson}. 
If we take \eqref{pl_poisson} as the definition of
$\pU_0$---notice that $\prho_0$ is already defined---,
then $\lim_{x\to\infty} \pU_0'(x) = -\lim_{x\to\infty} y'(x)$ so that
$\pU_0'+y' = 0$ and $\pU_0 +y = const$. If we evaluate this identity at $x=\pR_0$, it
follows that the proper choice of the cut-off energy is given by
\[
\pE_0 \coloneqq \pU_0(\pR_0) =   2 \pi\int_{-\pR_0}^{\pR_0} (\pR_0-y) \prho_0(y)\diff y
= 2 \pi \pR_0 \pM_0 .
\]
With this choice \eqref{eq:defh1D} holds, and all together we have proven the following.
\begin{prop}\label{existenceplanarstst}
  Let an ansatz of the form \eqref{eq:iso1D} with \eqref{eq:poly1d} or \eqref{eq:king1d}
  be fixed. Then for each
  $\pM_0>0$ there exists a unique corresponding
  steady state $(\pf_0,\pU_0,\prho_0)$ of the plane symmetric Vlasov-Poisson
  system~\eqref{pl_vlasov}--\eqref{pl_rho} with the following properties:
  \begin{enumerate}[label=(\alph*)]
  \item
    $\pM_0=\int_\R \prho_0(x)\diff x$ is the mass of the steady state.
  \item
    $\prho_0\in C^1(\R)$ has compact support $[-\pR_0,\pR_0]$  and is
    strictly decreasing on $[0,\pR_0]$.
  \item
    $\pU_0$ is convex on $\R$ and strictly increasing on $[0,\infty[$,
      $\pU_0(x)=2\pi \pM_0 x$ for $x\geq \pR_0$, and $\pU_0(x) = -2\pi \pM_0 x$
      for $x\leq \pR_0$.		
  \end{enumerate}
\end{prop}
As in the spherically symmetric setting, let
\begin{align*}
  \pOmega_0 \coloneqq \{(x,v)\in\R\times\R^3\mid \pf_0(x,v)\neq0\} =
  \{ (x,v_1)\in\R^2\mid \pE(x,v_1)<E_0 \} \times \{\pbeta\neq0\}
\end{align*}
be the (interior of the) support of the steady state in $(x,v)$-coordinates.
The finite cut-off energy ensures that $\pOmega_0$ is bounded and $\pOmega_0$
is open for the above ansatz functions. We again add an upper index when
expressing this set in different coordinates:
\begin{align*}
  \pOmega_0^{\pE\pv} \coloneqq \{ (\pE(x,v),\pv)\mid(x,v)\in\pOmega_0\}=
         [\pUmin,\pE_0[\times\{\pbeta\neq0\}
\end{align*}
Next, note that $\pvarphi'\coloneqq\partial_{\pE}\pvarphi$ exists on $\inter(\pOmega_0^{\pE\pv})$ with
\begin{align}\label{E:MONPLANAR}
\pvarphi' < 0 \text{ on } \inter(\pOmega_0^{\pE\pv}).
\end{align}
Here, $\inter(\ldots)$ denotes the interior of a set.
Condition~\eqref{E:MONPLANAR} is the analogue of the monotonicity
assumption~\eqref{eq:linstab} in the radial case. In view of a linear stability
analysis it leads to an Antonov-type coercivity bound
(proved later in Proposition~\ref{antonovcoerc1d}) which implies linear stability 
against perturbations which do not exhibit any dependence on $x_2$ or $x_3$.
For the linear stability analysis against general perturbations see~\cite{KuMa1970}.

Before proceeding, we note that we also have an analogue of \eqref{eq:A7prime}
in the plane symmetric setting, more precisely, there exists $C>0$ such that for all $x\in\R$
\begin{align}\label{eq:A7prime1d}
\int_{\R^3} \vert\pvarphi'(\pE(x,v),\pv)\vert\diff v \leq C.
\end{align}
The assumption~\eqref{eq:A7prime1d} however is of mere technical nature and it is expected that it can be relaxed.
While \eqref{eq:A7prime1d} is obvious in the King case, a straight-forward computation yields
\begin{align*}
\int_{\R^3} \vert\pvarphi'(\pE(x,v),\pv)\vert\diff v = 2k \int_{\pU_0(x)}^{\pE_0}\frac{\left(\pE_0-\pE\right)^{k-1}}{\sqrt{2\pE-2\pU_0(x)}}\diff\pE= c_k (\pE_0-\pU_0(x))^{k-\frac12}
\end{align*}
in the polytropic case if $\pU_0(x)<\pE_0$, where $c_k>0$ is some constant depending on $k>\frac12$.

We now consider the characteristic system corresponding to the steady state, i.e.,
\begin{align}\label{eq:charsys1D}
\dot x=v_1,\qquad \dot v_1=-\pU_0'(x).
\end{align}
We left out the trivial $v_2$ and $v_3$ equations. To analyze the solutions of this system we first introduce the following notation similar to Lemma \ref{effpot}:
\begin{lemma}\label{L:XMINUSPLUS}
	For all $\pE>\pUmin=\min(\pU_0)$ there exist unique $x_-(\pE)<0<x_+(\pE)$ satisfying $\pU_0(x_\pm(\pE))=\pE$. $x_\pm$ have the following properties:
	\begin{enumerate}[label=(\alph*)]
		\item $\pU_0(x)< \pE$ is equivalent to $x_-(\pE)<x<x_+(\pE)$.
		\item $x_\pm$ are continuously differentiable on $]\pUmin,\infty[$ with
		\begin{align*}
		x_\pm'(\pE) = \frac1{\pU_0'(x_\pm(\pE))},\quad \pE>\pUmin.
		\end{align*}
		\item $x_+=-x_-$.
		\item $x_+$ is strictly increasing on $]\pUmin,\infty[$, $x_-$ strictly decreasing.
		\item $\lim_{\pE\to\pUmin} x_\pm(\pE)=0$, which is why we set $x_\pm(\pUmin)\coloneqq0$.
	\end{enumerate}
\end{lemma}

Now consider a global solution $\R\ni t\mapsto(x(t),v_1(t))$ of \eqref{eq:charsys1D}. Since $\pE$ is a conserved quantity of the system, there exists $\pE\geq\pUmin$ such that $\pE=\pE(x(t),v_1(t))$ for all $t\in\R$. Solving for $v_1$ yields
\begin{align*}
v_1(t)=\pm\sqrt{2\pE-2\pU_0(x(t))},
\end{align*}
i.e., the solution $(x,v)$ is periodic and $x$ oscillates between $x_-(\pE)$ and $x_+(\pE)$.
\begin{defn} For $\pE>\pUmin$ define
	\begin{align}\label{eq:defT1D}
	\pT(\pE)\coloneqq 2 \int_{x_-(\pE)}^{x_+(\pE)} \frac{\diff x}{\sqrt{2\pE-2\pU_0(x)}} = 4 \int_0^{x_+(\pE)} \frac{\diff x}{\sqrt{2\pE-2\pU_0(x)}}.
	\end{align}
\end{defn}
Then $\pT(\pE)$ is the period of any solution of \eqref{eq:charsys1D} having energy $\pE$, i.e., the time needed for the $x$-component of the solution to travel from $x_-(\pE)$ to $x_+(\pE)$ and back to $x_-(\pE)$, see \cite{BiTr}. Since $\pU_0'(x)>0$ for $x>0$, the integral \eqref{eq:defT1D} exists for every $\pE>\pUmin$. 

We shall see in Section~\ref{ssc:essplane} that the properties of $\pT$ are strongly
related to the spectrum of the planar Antonov operator. In fact, the period functions of systems like \eqref{eq:charsys1D} have widely been studied, see \cite{Bo05,Ch85,ChWa86,Sc85}. A question of particular interest, which is also crucial for the existence of oscillating modes in Section~\ref{sc:mathur}, is whether or not the period function is monotone as a function of the energy. To study this monotonicity we first compute the derivative of $\pT$:
\begin{lemma}\label{Tderplane}
	$\pT$ is continuously differentiable on $]\pUmin,\infty[$ with
	\begin{align*}
	\pT'(\pE)
	=\frac2{\pE-\pUmin} \int_0^{x_+(\pE)}
	\frac{\left(\pU_0'(x)\right)^2-2(\pU_0(x)-\pU_0(0))\pU_0''(x)}{\left(\pU_0'(x)\right)^2}
	\frac{\diff x}{\sqrt{2\pE-2\pU_0(x)}}
	\end{align*}
	for $\pE>\pUmin$.
\end{lemma}
For details on how to calculate this derivative we refer to \cite[Theorem 2.1]{ChWa86}. The continuity of $\pT'$ follows by a straight-forward application of the dominated convergence theorem; note that the fraction in the integral above is bounded for $x\to0$.
Now let
\begin{align}\label{eq:defG}
\pG(x) \coloneqq \left(\pU_0'(x)\right)^2-2(\pU_0(x)-\pU_0(0))\pU_0''(x) ,\quad x\in\R.
\end{align}
Obviously, the non-negativity of $\pG$ will instantly imply the monotonicity of $\pT$.
The former is easy to verify in the planar case, since $\pG(0)=0$ and
\begin{align*}
\pG'(x) = -2 (\pU_0(x)-\pU_0(0)) \pU_0'''(x)
= -8\pi (\pU_0(x)-\pU_0(0))\prho_0'(x) \text{ for }x\in\R,
\end{align*}
i.e., $\pG$ is strictly increasing on $[0,\pR_0]$ and constant on $[\pR_0,\infty[$. Thus,
\begin{prop}\label{Tincreasingplane}
	$\pT$ is strictly increasing on $]\pUmin,\infty[$.
\end{prop}
From the monotonicity we easily obtain the boundedness of $\pT$ on the energy-support of the steady state:
\begin{prop}\label{Tboundsplane}
	For all $\pE\in]\pUmin,\pE_0[$,
	\begin{align*}
	0<\frac{2\pi}{\sqrt{\pU_0''(0)}} = \sqrt{\frac\pi{\prho_0(0)}} \eqqcolon \pT(\pUmin) < \pT(\pE)<\pT(\pE_0)<\infty.
	\end{align*}
\end{prop}
\begin{proof}
	The monotonicity of $\pT$ immediately gives the upper bound. For the lower bound it remains to show that
	\begin{align*}
	\lim\limits_{\pE\to\pUmin} \pT(\pE) = \frac{2\pi}{\sqrt{\pU_0''(0)}}.
	\end{align*}
	First, we change variables via $\eta=\pU_0(x)$ to rewrite $\pT(\pE)$ for fixed $\pE>\pUmin$ as follows:
	\begin{align*}
	\pT(\pE) = 2\sqrt2 \int_{\pUmin}^{\pE} \frac{\diff \eta}{\sqrt{\pE-\eta}\;\pU_0'\left(x_+(\eta)\right)},
	\end{align*}
	note that $x_+$ inverts $\pU_0\colon[0,x_+(\pE)]\to[\pUmin,\pE]$. Next, observe that for every $\eta\in ]\pUmin,\pE[$ there exists $s=s(\eta)\in]0,x_+(\eta)[$ such that 
	\begin{align*}
	\frac{\left( \pU_0'\left(x_+(\eta)\right) \right)^2}{\eta} = \frac{2\pU_0''(s)\pU_0'(s)}{\pU_0'(s)}
	\end{align*} 
	by the extended mean value theorem, and hence
	\begin{align*}
	\pU_0'\left(x_+(\eta)\right) = \sqrt{2\eta\,\pU_0''(s)} .
	\end{align*}
	As $\pE\to\pUmin$, we also have $\pU_0''(s)\to\pU_0''(0)>0$ uniformly in $\eta\in]\pUmin,\pE[$, and therefore
	\begin{align*}
	\pT(\pE) \to \frac2{\sqrt{\pU_0''(0)}}\int_{\pUmin}^{\pE} \frac{\diff\eta}{\sqrt{(\pE-\eta)\eta}} = \frac{2\pi}{\sqrt{\pU_0''(0)}}.&\qedhere
	\end{align*}
\end{proof}

We note that the finite extent of the steady state
(i.e., the existence of the cut-off energy $\pE_0$) causes the period to be bounded
from above. Conversely, the steady state having a \enquote{smooth core}
(i.e., $\pU_0''(0)=4\pi\prho_0(0)<\infty$) implies that $\pT$ is bounded away from zero.
These interpretations can also be found in the physics literature, cf. \cite{BiTr}.

\section{Linearization}\label{S:LINEARISATION}

In this section we consider three different ways to linearize the
Vlasov-Poisson system about
some steady state $(f_0,\rho_0,U_0)$ and to derive the corresponding Antonov operator
$\A$. The different approaches yield (essentially) the same operator $\A$, but
each allows for a different interpretation of oscillatory modes of $\A$.
We carry out the computations for three dimensional, spherically symmetric steady states
and only state the results for the case of planar symmetry, since the arguments are
analogous and somewhat simpler in that case.

Our purpose is to arrive at the operator $\A$ by some convincing,
but not necessarily rigorous, manipulations; a rigorous derivation
would only be necessary for deducing results for the non-linear Vlasov-Poisson system
from the linearized spectral analysis. Hence for this section we dispense with
rigor and formulate our findings not in the form of theorems and such.

In Subsection~\ref{ssc:eddritter} we derive an Eddington-Ritter type relation for the oscillation
period in the case of spherically symmetric, polytropic steady states.

\subsection{The Eulerian picture}\label{ssc:euler}

In the literature~\cite{GuLi08,GuRe2001,GuRe2007,IT68,LeMeRa12}
the Vlasov-Poisson system is usually linearized in Eulerian variables,
starting with the formal expansion of the form
$f = f_0 + \varepsilon \delta f + O(\varepsilon^2)$.
This leads to the linearized system
\[
\pa_t \delta f + \D \delta f - \nabla\delta U  \cdot \nabla_v f_0=0, 
\]
\[
\Delta \delta U = 4 \pi \delta\rho,\quad\lim_{\vert x\vert\to\infty}\delta U(t,x)=0,
\]
\[
\delta\rho(t,x) = \int_{\R^3} \delta f(t,x,v) \diff v,
\]
where we recall the definition \eqref{E:Ddef} of the transport operator $\D$.
In the manipulations to follow it will be convenient to use the notations
\begin{align}\label{E:subs_def}
	\rho_g(x) \coloneqq \int_{\R^3} g(x,v) \diff v,\quad j_g (x) \coloneqq\int_{\R^3} v\, g(x,v) \diff v,\quad
	U_g (x) \coloneqq - \frac{1}{|\cdot|}\ast \rho_g
\end{align}
so that $\delta \rho(t) = \rho_{\delta f(t)}$ and $\delta U(t) = U_{\delta f(t)}$.
Following Antonov~\cite{An1961} we write $\delta f = \delta f_+ + \delta f_-$,
where
\begin{align}\label{evenodd}
	\delta f_\pm (t,x,v) = \frac{1}{2} \left(\delta f(t,x,v) \pm \delta f(t,x,-v)\right)
\end{align}
denote the even and odd parts of $\delta f$ with respect to $v$.
Then $\delta U (t) = U_{\delta f_+(t)}$, and
\begin{align}
	\pa_t \delta f_- + \D \delta f_+ &= \nabla U_{\delta f_+} \cdot \nabla_v f_0,
	\label{linEul1}\\
	\pa_t \delta f_+ + \D \delta f_- &=0.
	\label{linEul2}
\end{align}
In order to obtain a second order equation for $\delta f_-$ we wish to differentiate the
first equation with respect to $t$ and substitute the second one. To do so we observe that
$\pa_t U_{\delta f_+} = U_{\pa_t\delta f_+} = - U_{\D \delta f_-}$. Since
$\rho_{\D \delta f_-}=\div j_{\delta f_-}$, we arrive at
\[
\pa_t^2 \delta f_- - \D^2 \delta f_- + \nabla U_{\div j_{\delta f_-}}\cdot \nabla_v f_0 = 0.
\]
For the latter equation to be of the form \eqref{E:LINDYNINTRO}, we define the Antonov operator $\A$ as $\A\coloneqq-\D^2-\B$, where
\begin{align} \label{E:B-def-3d}
	(\B g)(x,v) = - \nabla_v f_0(x,v) \cdot
	\int_{\R^3} \frac{x-y}{|x-y|^3}\,\div j_g (y)\diff y.
\end{align}
If $g=g(r,w,L)$ is spherically symmetric, then
\[
\nabla U_{\div j_g} (x) = \frac{1}{r^2} \int_{|y| \leq r}\div j_g(y) \diff y \frac{x}{r}
= 4\pi j_g(r)\cdot\frac{x}{r} \frac{x}{r},\quad r=\vert x\vert;
\]
notice that $j_g$ is still the vector field defined in \eqref{E:subs_def}. Since
\[
\nabla_v f_0(x,v) = \pa_E \varphi(E,L) v + \pa_L \varphi(E,L) (2 r^2 v - 2 x\cdot v x)
\]
and $(2 r^2 v - 2 x\cdot v x)\cdot \frac{x}{r} = 0$, $\B g$
is spherically symmetric as well, more precisely
\begin{align} \label{E:B-def-ss}
  (\B g)(r,w,L) = -4\pi w\,\varphi'(E,L)\,j_g(r)
  = -\frac{4\pi^2}{r^2}\,w\,\varphi'(E,L) \int_{\R} \int_0^\infty \tilde w\,
  g(r,\tilde w,\tilde L)\diff \tilde L\diff \tilde w,
\end{align}
where $\varphi'(E,L)=\pa_E \varphi(E,L)<0$ and $E=E(r,w,L)=\frac12w^2+\Psi_L(r)$.

The same arguments as above
can be applied to the plane symmetric case, resulting in
\begin{align} \label{E:B-def-pl}
  (\pB g)(x,v) = -  4 \pi \,v_1\, \varphi'(E,\bar v) 
  \int_{\R^3} \tilde v_1\, g(x,\tilde v)\diff \tilde v;
\end{align}
here we recall that the reflection symmetry included in the definition \eqref{pl_symm_def}
of planar symmetry implies the formula \eqref{pl_dxU_sym} for the potential induced by $g$.

We now examine whether the spectral properties of $\A$ obtained in the Eulerian
linearization can explain the oscillations about the steady state which were observed
numerically in~\cite{RaRe2018}. For example the spatial support of the solution
should expand and contract, but also its kinetic and potential energies should
oscillate; see e.g.~\cite{Re07} for the definitions of the energies.
As already observed in the introduction, the former cannot be seen
in the present, Eulerian linearized picture, but the latter can.
Assume that we have a positive eigenvalue $\lambda >0$ of $\A$ with
spherically symmetric eigenfunction
$g_-$; the latter will be real-valued and odd in $w$. Let $\omega\coloneqq\sqrt\lambda$. Then 
\[
\delta f_- = \cos(\omega t) g_-
\]
is a solution of the linearized equation $\pa_t^2\delta f_- + \A \delta f_-=0$;
replacing $\cos(\omega t)$ by $\sin(\omega t)$ gives another one.
According to \eqref{linEul2} the corresponding part which is even in $v$
is given as
\[
\delta f_+ = - \frac{1}{\omega} \sin(\omega t) \D g_-,
\]
and only this part contributes to the kinetic and potential energies:
\begin{align*}
  E_{\mathrm{kin}}(f_0 +\varepsilon \delta f(t))
  &= E_{\mathrm{kin}}(f_0)
  -\frac{\varepsilon}{2\omega}\sin(\omega t) \int_{\R^6} |v|^2 \D g_-(z)\diff z\\
  &= E_{\mathrm{kin}}(f_0)
  -\frac{\varepsilon}{\omega}\sin(\omega t) \int_{\R^3} \nabla U_0\cdot j_{g_-}\diff x,
\end{align*}
which in general oscillates with period $2\pi/\omega$ about the kinetic energy
of the steady state. By the analogous computation for the potential energy,
\[
E_{\mathrm{pot}}(f_0 +\varepsilon \delta f(t)) = E_{\mathrm{pot}}(f_0)
+\frac{\varepsilon}{\omega}\sin(\omega t) \int_{\R^3} \nabla U_0\cdot j_{g_-}\diff x
+ O(\varepsilon^2).
\]
\subsection{Mass-Lagrange coordinates} \label{ssc:masslagrange}
The following derivation is based on the so-called mass-Lagrangian coordinates
which are often used in spherical symmetry for the Euler-Poisson system,
cf~\cite{JaMa2009,mak}.
We restrict ourselves from the start
to the spherically symmetric situation, cf.~Section~\ref{ssc:symmetries}.

We fix a steady state $(f_0,U_0,\rho_0)$ which for the sake of simplicity we
take isotropic, i.e., $f_0 = \varphi(E)$.
We consider a spherically symmetric
solution to the Vlasov-Poisson system with data which are dynamically accessible
from $f_0$ so that in particular the solution $f(t)$ has the same mass $M>0$
as the steady state, and we assume that
$\{r>0\mid\rho(t,r)>0\}=]0,R(t)[$.
Then the mapping
\begin{align} \label{mcoord}
	[0,R(t)] \ni r\mapsto m(t,r) \in[0,M]
\end{align}
is one-to-one, since $\pa_r m = 4\pi r^2 \rho >0$
on $]0,R(t)[$.
Let  
\begin{align} \label{rtcoord}
	[0,M] \ni m\mapsto \tilde r(t,m) \in[0,R(t)]
\end{align}
denote its inverse.
We introduce the new dependent variables
\[
\tilde f(t,m,w,L) = f(t,r,w,L), \quad \tilde \rho(t,m) = \rho(t,r).
\]
For a function $g=g(\cdot,w,L)$ we introduce the abbreviations
\begin{align}\label{RJdef}
  \Ri(g) \coloneqq
  \int_{-\infty}^\infty\int_0^\infty g(\cdot,w,L) \diff L \diff w, \ \
  \J(g) \coloneqq
  \int_{-\infty}^\infty\int_0^\infty w\, g(\cdot,w,L) \diff L \diff w
\end{align}
so that
\begin{align*}
  \rho(t) = \frac{\pi}{r^2} \Ri(f(t)),\quad
  \tilde \rho(t) = \frac{\pi}{\tilde r^2} \Ri(\tilde f(t)).
\end{align*}
In order to rewrite the Vlasov-Poisson system in terms of $(m,w,L)$ we compute
\begin{align*}
	\partial_t m(t,r)
	&=
	4\pi \int_0^r s^2\pa_t \rho(t,s)\, \diff s
	=
	-4\pi^2 \int_0^r \pa_r \int_{-\infty}^\infty\int_0^\infty w\, f(t,r,w,L) \,\diff L\diff w
	\diff s\\
	&=
	-4\pi^2 \J(f(t))(r).
\end{align*}
Hence
\[
\pa_t f  = \pa_t \tilde f + \pa_m\tilde f\,\pa_t m 
= \pa_t \tilde f - 4\pi^2 \J(\tilde f(t))\,\pa_m\tilde f,\quad
\pa_r f  = \pa_m \tilde f\,\pa_r m 
= 4\pi^2 \Ri(\tilde f(t))\,\pa_m \tilde f ,
\]
and the Vlasov equation \eqref{ss_vlasov} takes the form
\begin{align} \label{ss_vlasov_mass}
	\pa_t \tilde f + 4\pi^2  \left(w\, \Ri(\tilde f(t)) -
	\J(\tilde f(t))\right) \pa_m \tilde f +
	\left(\frac{L}{\tilde r^3} - \frac{m}{\tilde r^2}\right) \pa_w \tilde f = 0.
\end{align}
It needs to be supplemented with the relation between the function
$\tilde r(t,m)$ and $\tilde \rho$ respectively $\tilde f$.
Since
\[
\pa_m \tilde r = \frac{1}{4 \pi \tilde r^2 \tilde \rho}
= \frac{1}{4 \pi^2 \Ri(\tilde f)},
\]
it follows that
\begin{align} \label{ss_poisson_mass}
	\tilde r(t,m)  
	=\frac{1}{4\pi^2} \int_0^m \frac {\diff\eta}{\Ri(\tilde f(t))(\eta)},
\end{align}
and \eqref{ss_vlasov_mass}, \eqref{ss_poisson_mass} constitute the
spherically symmetric Vlasov-Poisson system in the coordinates $(m,w,L)$,
where we need to recall the abbreviations introduced in \eqref{RJdef}.

We notice that $\tilde r(t,m)$ is the radius of the ball about the origin
in which mass $m$ of the solution $f(t)$ is contained. In a second change of
variables we relate this to the steady state configuration as follows.
The map
\[
[0,R_0] \ni r\mapsto m_0(r) = 4\pi \int_0^r s^2\rho_0(s)\,\diff s \in[0,M]
\]
is again one-to-one, and
\[
[0,M] \ni m\mapsto r_0(m) \in[0,R_0]
\]
denotes its inverse; $R_0>0$ is the radius of the spatial support of the steady state.
Now we introduce a new radial variable via
\[
[0,M] \ni m \mapsto \bar r := r_0(m) \in [0,R_0],
\]
and new dependent variables
\[
\hat f(t,\bar r,w,L) = \tilde f(t,m,w,L) = f(t,r,w,L), \ \
\hat \rho(t,\bar r) =\tilde \rho(t,m) = \rho(t,r).
\]
If we let
$\hat r(t,\bar r) = \tilde r(t,m_0(\bar r))$,
the spherically symmetric Vlasov-Poisson system becomes
\begin{equation} \label{ss_vlasov_massr}
	\pa_t \hat f +  \left(w\, \frac{\Ri(\hat f(t))}{\Ri(f_0)} -
	\frac{\J(\hat f(t))}{\Ri(f_0)}\right) \pa_{\bar r} \hat f +
	\left(\frac{L}{\hat r^3} - \frac {m_0(\bar r)}{\hat r^2}\right) \pa_w \hat f = 0, 
\end{equation}
\begin{equation} \label{ss_poisson_massr}
	\hat r (t,\bar r)  
	= \int_0^{\bar r} \frac{\Ri(f_0)}{\Ri(\hat f(t))}\,\diff s;
\end{equation}
notice that $\pa m/\pa\bar r = 4\pi^2 \Ri(f_0)$ and
\eqref{ss_poisson_massr} is obtained from \eqref{ss_poisson_mass}
via the change of variables $s\mapsto \eta = m_0(s)$.
The ball of radius $\hat r (t,\bar r)$ about the origin
contains for the time dependent solution $f(t)$
the same amount of mass as the ball of radius $\bar r$ does for the steady state.
In this way the steady state mass distribution is used as the reference frame for
describing the mass distribution of the time dependent solution.
In particular, should $\hat r (t,\bar r)$ on the linearized level oscillate
about $\bar r$, this would exactly explain the pulsating behavior of the
perturbed steady state which was observed numerically.

In order to linearize~\eqref{ss_vlasov_massr}, \eqref{ss_poisson_massr}
about the steady state we write 
\[
\hat f = f_0 + \varepsilon \delta \hat f + O(\varepsilon^2), 
\quad \hat\rho = \rho_0 +\varepsilon \delta \hat\rho + O(\varepsilon^2)
\]
and
\begin{align}\label{rhatexpansion}
	\hat r(t,\bar r) = \bar r + \varepsilon \delta\hat r(t,\bar r)
	+ O(\varepsilon^2),
\end{align}
and expand \eqref{ss_vlasov_massr}, \eqref{ss_poisson_massr} to the first order.
When doing so we observe that the transport operator $\D$ along the characteristic flow
of the steady state now takes the form
\[
\D f 
=w\,\pa_{\bar r} f +  
\left(\frac{L}{\bar r^3}-\frac {m_0(\bar r)}{\bar r^2}\right) \pa_wf.
\]
This implies that
\begin{equation} \label{ss_vlasov_massr_lin}
	\pa_t \delta\hat f + \D \delta\hat f
	+  \left(w\frac{\Ri(\delta\hat f)}{\Ri(f_0)}-\frac{\J(\delta\hat f)}{\Ri(f_0)}\right)
	\pa_{\bar r} f_0
	+ \delta\hat r\, \left(\frac{2m_0(\bar r)}{\bar r^3}-\frac{3L}{\bar r^4}\right)
	\pa_w f_0 = 0,
\end{equation}
\begin{equation} \label{ss_poisson_massr_lin}
	\delta\hat r(t,\bar r) = -\int_0^{\bar r}\frac{\Ri(\delta\hat f(t))}{\Ri(f_0)}\,\diff s.
\end{equation}
Before we turn this system into a second order one for the
odd-in-$w$ part of $\delta \hat f$
to obtain the corresponding Antonov operator we compute the time derivative of
$\delta\hat r(t,\bar r)$.
Using the linearized Vlasov equation \eqref{ss_vlasov_massr_lin}
a short computation shows that
\begin{align} \label{dtRf}
	\pa_t \Ri(\delta\hat f) = - \Ri(f_0) \pa_{\bar r} \left(\frac{\J(\delta\hat f)}{\Ri(f_0)}\right)
\end{align}
and hence
\begin{align} \label{dtrhat}
	\pa_t \delta\hat r = \frac{\J(\delta\hat f)}{\Ri(f_0)} .
\end{align}
Now we again split  $\delta\hat f= \delta\hat f_+ + \delta\hat f_-$
into its even and the odd part with respect to $w$. 
They satisfy the following linear system:
\begin{align}
	\pa_t \delta\hat f_+ + \D  \delta\hat f_- - \frac{\J(\delta\hat f_-)}{\Ri(f_0)}
	\pa_{\bar r} f_0 & = 0,
	\label{ss_vlasov_massr_lin+}\\
	\pa_t \delta\hat f_- + \D  \delta\hat f_+ + w\frac{\Ri(\delta\hat f_+)}{\Ri(f_0)}
	\pa_{\bar r} f_0 +
	\delta\hat r\, \left(\frac{2m_0(\bar r)}{\bar r^3}-\frac{3L}{\bar r^4}\right)
	\pa_w f_0 & = 0,
	\label{ss_vlasov_massr_lin-}
\end{align}
\begin{equation} \label{ss_poisson_massr+}
	\delta\hat r(t,\bar r) = -\int_0^{\bar r}\frac{\Ri(\delta\hat f_+(t))}{\Ri(f_0)}\,\diff s.
\end{equation}
In order to differentiate \eqref{ss_vlasov_massr_lin-}
with respect to $t$ we observe from
\eqref{dtRf} and \eqref{dtrhat} that
\begin{align*} 
	\pa_t \Ri(\delta\hat f_+) =
	- \Ri(f_0)\,\pa_{\bar r} \left(\frac{\J(\delta\hat f_-)}{\Ri(f_0)}\right),\quad
	\pa_t \delta\hat r = \frac{\J(\delta\hat f_-)}{\Ri(f_0)}.
\end{align*}
Thus
\begin{align}\label{secorder_massr}
	0 &  = \pa_t^2 \delta\hat f_- - \D^2 \delta\hat f_-
	+ \D \left(\frac{\J(\delta\hat f_-)}{\Ri(f_0)}\pa_{\bar r}f_0\right) 
	- w\,\pa_{\bar r}\left(\frac{\J(\delta\hat f_-)}{\Ri(f_0)}\right)\pa_{\bar r}f_0\notag\\
	&\qquad \qquad\qquad\qquad\qquad
	+  \left(\frac{2m_0(\bar r)}{\bar r^3}-\frac{3L}{\bar r^4}\right)
	\frac{\J(\delta\hat f_-)}{\Ri(f_0)} \pa_wf_0 \notag  \\
	& = \pa_t^2 \delta\hat f_- - \D^2  \delta\hat f_-  - \B \delta\hat f_-,
\end{align}
where
\begin{align*}
	\B g
	& =- \D \left(\frac{\J(g)}{\Ri(f_0)}\pa_{\bar r} f_0\right) 
	+w\,\pa_{\bar r}\left(\frac{\J(g)}{\Ri(f_0)}\right)\pa_{\bar r} f_0 
	-  \left(\frac{2m_0(\bar r)}{\bar r^3}-\frac{3L}{\bar r^4}\right)
	\frac{\J(g)}{\Ri(f_0)} \pa_w f_0\\
	&=
	\varphi'(E) \left[-\D\left(\frac{\J(g)}{\Ri(f_0)} \pa_{\bar r} E\right)
	+ w \pa_{\bar r} \left(\frac{\J(g)}{\Ri(f_0)}\right)\pa_{\bar r} E
	- w \left(\frac{2m_0(\bar r)}{\bar r^3}-\frac{3L}{\bar r^4}\right)
	\frac{\J(g)}{\Ri(f_0)}\right]\\  
	&=
	\varphi'(E) \left[-w \frac{\J(g)}{\Ri(f_0)} \pa^2_{\bar r} E
	- w \left(\frac{2m_0(\bar r)}{\bar r^3}-\frac{3L}{\bar r^4}\right)
	\frac{\J(g)}{\Ri(f_0)}\right]\\
	& = -4\pi^2 \varphi'(E) \frac{w}{\bar r^2}\J(g); 
\end{align*}
for the last identity we have used the fact that
\[
\pa^2_{\bar r} E =
\pa^2_{\bar r}\left( \frac{1}{2}w^2 + \frac{L}{2\bar r^2} + U_0(\bar r)\right)
= \frac{3L}{\bar r^4}-\frac{2m_0(\bar r)}{\bar r^3}+ 4\pi \rho_0(\bar r).
\]
We have therefore shown that the linearized dynamics in what we refer
to as mass-Lagrange coordinates is governed by \eqref{secorder_massr},
and hence by an equation of the form \eqref{E:LINDYNINTRO} where
the corresponding Antonov operator $\A$ is exactly the same as obtained
in the Eulerian picture under the assumption of spherical symmetry,
cf.~\eqref{E:B-def-ss}. The analogous arguments imply the analogous
result for the plane symmetric case.

Assume now that we have an eigenvalue $\lambda >0$ and a corresponding,
spherically symmetric eigenfunction $g_-$, which must be odd in $w$.
Then $\delta \hat f = \cos(\omega t) g_-$ with $\omega = \sqrt{\lambda}$
solves the above linearized dynamics in mass-Lagrange coordinates \eqref{secorder_massr}.
In view of~\eqref{rhatexpansion} and~\eqref{dtrhat} if follows that
to linear order,
\begin{align} \label{radosc}
	\hat r(t,\bar r) = \bar r +
	\frac{\varepsilon}{\omega}\sin(\omega t) \frac{\J(g_-)}{\Ri(f_0)}.
\end{align}
Hence to linear order the radius of the ball containing a certain mass
$m$ oscillates with period $2\pi/\omega$ about the corresponding radius
for the steady state. Of course the details of this oscillation---like which portions
of the configuration take part in it---depend on the actual eigenfunction,
but \eqref{radosc} nicely explains the numerically observed pulsation behavior.
Notice that while the even part of the solution to the linearized
dynamics governs the oscillation of the kinetic and potential energies as seen at
the end of the previous subsection, its odd part governs the oscillation of the spatial
support here.
\subsection{The Lagrangian picture}\label{ssc:lagrange}
The viewpoint of this section is to focus on the map which re-distributes the particles
in phase space as time proceeds.
We again consider a fixed, spherically symmetric steady state $(f_0,\rho_0,U_0)$ and let
\[
\Omega_0 \coloneqq \{ z=(x,v)\in \R^6 \mid f_0(z)>0\}.
\]
We wish to study the initial value problem for the
Vlasov-Poisson system with data
\[
\mathring f = f_0 \circ \mathring Z^{-1},
\]
extended by $0$ to all of $\R^6$, where
\[
\mathring Z \colon \Omega_0 \to \Omega^0
\]
is a diffeomorphism onto some open set $\Omega^0\subset \R^6$, which is
measure-preserving, i.e., $\det \pa_z \mathring Z = 1$ on $\Omega_0$.
Such data are dynamically accessible from $f_0$,
and physically viable perturbations of $f_0$,
say, by some exterior force acting on the galaxy, are of this form.
It is easy to see that for such data
the initial value problem is equivalent to the following one:
We look for a time dependent family
\[
Z(t)= (X(t),V(t)) \colon \Omega_0 \to \Omega^t
\]
of measure-preserving diffeomorphisms onto some open set $\Omega^t\subset \R^6$ such that
\begin{align}
	\dot X & =  V,  \label{E:CHAR1}\\
	\dot V & = -\nabla U_{Z(t)}(X), \label{E:CHAR2}
\end{align}
where 
\begin{align}\label{E:POTENTIALX}
	U_{Z(t)}(x)
	= - \int_{\mathbb R^6}\frac{1}{|x-X(t,\tilde z)|} f_0(\tilde z)\diff \tilde z, \quad
	\nabla U_{Z(t)}(x)
	= \int_{\mathbb R^6} \frac{x-X(t,\tilde z)}{|x-X(t,\tilde z)|^3} f_0(\tilde z)\diff \tilde z.
\end{align}
We refer to \eqref{E:CHAR1}--\eqref{E:POTENTIALX}
as the Lagrangian formulation of the Vlasov-Poisson system
and supplement it with the initial condition
\begin{align}
	Z(0) = \mathring Z.\label{E:INCOND}
\end{align}
If $Z(t)$ is a solution of this initial value problem,
then $f(t,z)=f_0(Z(t)^{-1}(z))$,
extended by $0$ outside $\Omega^t = Z(t,\Omega_0)$,
solves the corresponding initial value problem for the Vlasov-Poisson system,
and $\overline{\Omega^t}$ is the support of $f(t)$.
The choice $\mathring Z = \mathrm{id}$ leads to the steady state,
and we denote the solution to \eqref{E:CHAR1}--\eqref{E:POTENTIALX}
with initial data $Z(0) = \mathrm{id}$ by $Z_0(t)$, which is the characteristic
flow of the steady state.

In the above reformulation of the Vlasov-Poisson system the basic unknown is the
flow map $Z(t)$ and the configuration space of the system is the set of measure-preserving
(or symplectic) diffeomorphisms. The mathematical structure of this configuration space
has been studied in \cite{EbMar}, see also~\cite{Mar}.
We wish to linearize the system \eqref{E:CHAR1}--\eqref{E:POTENTIALX};
a similar approach was followed  by Vandervoort~\cite{Va1983}. 

For small perturbations of the steady state we expect the composition
$Z(t)\circ Z_0(t)^{-1}$ to remain close to the identity so that
the natural definition of the perturbation is given through the expansion
\begin{align*}
	Z(t)\circ Z_0(t)^{-1} = \mathrm{id} + \varepsilon \delta Z(t) + O(\varepsilon^2),
\end{align*}
or equivalently,
\begin{align}\label{E:PERTURBTWO}
	Z(t) = Z_0(t) + \varepsilon \delta Z(t)\circ Z_0(t) + O(\varepsilon^2).
\end{align}
Differentiating both sides of~\eqref{E:PERTURBTWO} with respect to $t$ we find that
\begin{align}\label{E:PARTIALT}
	\dot Z =
	\dot Z_0 + \varepsilon \left[\left(\pa_t + \mathcal D\right)\delta Z\right] \circ Z_0,
\end{align}
where $\mathcal D = v\cdot\nabla_x - \nabla_xU_0\cdot \nabla_v$ as before,
and where we have dropped higher-order terms in $\varepsilon$.
In order to derive the system governing the evolution of $\delta Z(t)$ we need to
expand the right hand side of \eqref{E:CHAR1} and \eqref{E:CHAR2}. The former is
trivial, but the latter requires some thought:
\begin{align*}
  &
  \nabla U_{Z_0(t) + \varepsilon \delta Z(t,Z_0(t))}(X_0(t) + \varepsilon \delta X(t,Z_0(t)))
  = \nabla U_{Z_0(t)}(X_0(t))\\
  &\qquad + \varepsilon \frac{d}{d\varepsilon}
  \nabla U_{Z_0(t)}(X_0(t) + \varepsilon \delta X(t,Z_0(t)))_{|\varepsilon=0}
  + \varepsilon\frac{d}{d\varepsilon}
  \nabla U_{Z_0(t) + \varepsilon \delta Z(t,Z_0(t))}(X_0(t))_{|\varepsilon=0}\\
  &
  = \nabla U_0(X_0(t)) + \varepsilon I_1 + \varepsilon I_2.
\end{align*}
Since $U_{Z_0(t)} = U_0$ is sufficiently smooth,
\[
I_1 = D^2 U_0(X_0(t))\, \delta X(t,Z_0(t)).
\]
In order to compute the second term $I_2$ we multiply it with a test function
$\psi \in C^\infty_c(\Omega_0)$ and integrate with respect to $z\in \Omega_0$. Then
\begin{align*}
  \int_{\R^6} I_2(z)\, \psi(z)\, \diff z
  &= - \int_{\R^6}\frac{d}{d\varepsilon}
  \nabla U_{Z_0(t)}^\psi (\tilde x + \varepsilon \delta X(t,\tilde z))_{|\varepsilon=0}
  f_0(\tilde z)\, \diff\tilde z \\
  &= - \int_{\R^6} D^2 U_{Z_0(t)}^\psi (\tilde x)\, \delta X(t,\tilde z)\, f_0(\tilde z)\, \diff\tilde z,
\end{align*}
where
\[
U_{Z_0(t)}^\psi(\tilde x) \coloneqq - \int_{\R^6} \frac{1}{|\tilde x - x|}\psi(Z_0(t)^{-1}(z))\, \diff z,
\]
and we changed variables with the measure-preserving diffeomorphism $Z_0(t)$;
note that the latter potential is again smooth. Now we reverse the change of variables,
change the order of integration, and finally drop the integration against the test
function $\psi$ to arrive at
\[
I_2 = - \nabla \div \tilde U_{\delta X(t)} (X_0(t)), \ \mbox{where}\
\tilde U_{\delta X(t)} (x)
\coloneqq - \int_{\R^6} \frac{1}{|x-\tilde x|}\delta X (t,\tilde z)\, f_0(\tilde z)\, \diff\tilde z.
\]
Putting these results into \eqref{E:PARTIALT} and using the fact that
$Z_0(t) \colon \Omega_0 \to \Omega_0$ is a diffeomorphism we conclude that
\begin{align}
  \left(\pa_t + \mathcal D\right)\delta X & = \delta V \label{E:DELTAREQN1},\\
  \left(\pa_t + \mathcal D\right)\delta V & =
  -D^2 U_0\,  \delta X + \nabla \div \tilde U_{\delta X(t)},  \label{E:DELTAWEQN1}
\end{align}
which is the Lagrangian linearization of the Vlasov-Poisson system around the
steady state $f_0$. This system is equation (6) in~\cite{Va1983}.

We will need certain commutator relations between
$\D$ and derivatives with respect to $x$ or $v$. For smooth vector fields
$F$ or functions $f$ it follows by direct computation that
\begin{align}
  \D\,\divv F &= \divv \D F -\divx F, \label{E:COMM1}\\
  \D\,\divx F &= \divx \D F + D^2 U_0 \cdot D_v F, \label{E:COMM2} \\
  \D\,\nabla_v f &= \nabla_v\D f -\nabla_x f, \label{E:COMM3}\\
  \D\,\nabla_x f &= \nabla_x\D f + D^2 U_0\nabla_vf. \label{E:COMM4}
\end{align}
It turns out that phase-space volume is preserved at the linear level
in the following sense. Let $\delta Z$ be a solution to the linearized
dynamics~\eqref{E:DELTAREQN1}, \eqref{E:DELTAWEQN1}. Then 
by \eqref{E:COMM1} and~\eqref{E:COMM2},
\begin{align*}
  &\left(\pa_t + \mathcal D\right)\left(\divx\delta X + \divv \delta V  \right) \\
  & \qquad
  =  D^2 U_0\cdot D_v\delta X + \divx \left(\pa_t + \mathcal D\right)\delta X
  - \divx\delta V + \divv \left(\pa_t + \mathcal D\right)\delta V \\
  & \qquad
  = D^2 U_0\cdot D_v\delta X +\divx \delta V
  -  \divx\delta V + \divv \left(-D^2U_0\delta X +
  \nabla\div\tilde U_{\delta X(t)}\right) = 0.
\end{align*}
This implies that
$\left(\divx\delta X + \divv \delta V\right)(t)=0$ for all $t$,
provided this is true initially.

In what follows we restrict $\delta Z$ to a form which in particular guarantees
this preservation of phase-space volume, namely, we assume that
the perturbation is a skew-gradient, i.e.,
there exists a smooth function $h\colon \R\times \Omega_0\to \R$ such that 
\begin{align}\label{E:HODGE0}
	\delta X = \nabla_v h,\ \delta V = -\nabla_x h. 
\end{align}
We can rewrite the linearized, Lagrangian dynamics
\eqref{E:DELTAREQN1}, \eqref{E:DELTAWEQN1} in terms of the generating
function $h$:
\begin{align*}
  \left(\pa_t + \mathcal D\right) \nabla_v h & = - \nabla_x h, \\
  \left(\pa_t + \mathcal D\right) \nabla_x h & =
  D^2 U_0  \nabla_v h - \nabla\div \tilde U_{\nabla_v h}.
\end{align*}
Using~\eqref{E:COMM3} and \eqref{E:COMM4} this system is equivalent to 
\begin{align*}
  \nabla_v\left(\pa_t + \mathcal D\right)  h &= 0, \\
  \nabla_x\left(\pa_t + \mathcal D\right)  h & = - \nabla \div \tilde U_{\nabla_v h},
\end{align*}
which in turn is equivalent to
\begin{align} \label{fform}
  \left(\pa_t + \mathcal D\right)  h = f(x)\ \mbox{with}\
  \nabla f =- \nabla \div \tilde U_{\nabla_v h};
\end{align}
notice that for every $x\in \R^3$ the set $\{v\in \R^3 \mid (x,v)\in \Omega_0\}$ is connected.
We again split $h=h_++h_-$ into its even and odd parts
with respect to $v$, cf.~\eqref{evenodd}.
Then
\begin{align*}
  \pa_t h_- + \D h_+ &= 0, \\
  \pa_t h_+ + \D h_- &= f(x).
\end{align*}
Differentiating the first equation with respect to $t$, inserting the second one
and using the formula for $\nabla f$ from \eqref{fform}
we find that
\begin{align} 
  \pa_t^2 h_- - \D^2 h_- - v\cdot\nabla \div \tilde U_{\nabla_v h_-} =0 ; \label{E:HODDDYNAMICS}
\end{align}
note that $\nabla_v h_+$ is odd in $v$ and does not contribute to
\[
\tilde U_{\nabla_v h}(x) = \tilde U_{\nabla_v h_-}(x)
=  -\int_{\mathbb R^6}\frac{1}{|x-\tx|} \nabla_v h(t,\tilde z)\,
f_0(\tilde z)\diff \tilde z
\]
since $f_0$ is even in $v$.
In order to proceed we assume that the steady state
is isotropic, i.e., $f_0=\varphi(E)$. Then we integrate by parts in
the integral above, and \eqref{E:HODDDYNAMICS} can be written
in the form
\begin{align}\label{lindynlagr}
	\pa_t^2 h + \tilde\A h =0,
\end{align}
where the corresponding Antonov operator is given as
\[
(\tilde \A g)(x,v) = - (\D^2 g)(x,v)  -  v\cdot \nabla_x \divx
\int_{\mathbb R^6}\frac{1}{|x-\tx|} g(\tilde z)\, \varphi'(\tilde E) \,\tilde v\,\diff \tilde z. 
\]
for suitable $g\colon\R^3\times\R^3\to\R$.
If we restrict ourselves to spherical symmetry
so that by abuse of notation $g(x,v)= g(r,w,L)$, then the integral above
defines a spherically symmetric vector field $F (x)=\frac{x}{r}F(r)$
for which $\nabla\div F= \Delta F$. Thus
\begin{align}\label{A_lagr_ss}
  (\tilde \A g)(x,v)
  &= - (\D^2 g)(x,v)  + 4\pi v\cdot
  \int_{\mathbb R^3} g(x,\tilde v)\, \varphi'(\tilde E) \,\tilde v\,\diff \tilde v \nonumber\\
  &= - (\D^2 g)(r,w,L)  + \frac{4\pi^2}{r^2} w
  \int_{-\infty}^\infty \int_0^\infty  g(r,\tilde w,\tilde L)\, \varphi'(\tilde E) \,\tilde w\,
  \diff \tilde L \diff \tilde w 
\end{align}
is the Antonov operator which governs the linearized, spherically symmetric
dynamics in the Lagrangian formulation of the Vlasov-Poisson system.

Strictly speaking, $\tilde \A$ differs from $\A$ in the exact form of the
integral operator part, but the two are equivalent in the following sense:
Assume that two functions, both of which are odd in $w$, are related through
\[
f(r,w,L) = \varphi'(E)\, g(r,w,L).
\]
Then $f$ is an eigenfunction of the operator $\A$ with eigenvalue $\lambda$,
iff $g$ is an eigenfunction of the operator $\tilde \A$ to the same eigenvalue.

Again, the same arguments apply in the plane symmetric case and result in
\begin{align}\label{A_lagr_pl}
  (\tilde \A g)(x,v)
  &= - (\pD^2 g)(x,v)  + 4\pi v_1 
  \int_{\R^3} g(x,\tilde v)\, \varphi'(\tilde E) \,\tilde v_1 \diff \tilde v . 
\end{align}

Like the previous two subsections,
we conclude the present one by examining its consequences for possible
oscillations of a steady state after perturbation. To this end, let $g$ denote an
eigenfunction of the operator $\tilde \A$ with eigenvalue $\lambda>0$. Then the
closure of $\Omega^t = Z(t,\Omega_0)$ is the support in phase space
of the perturbed solution, and to linear order
\[
\supp f(t) = \{ z + \varepsilon \cos(\omega t) (\nabla_v g(z),-\nabla_xg(z))
\mid z\in \overline{\Omega}_0\},
\]
where $\omega = \sqrt{\lambda}$, so that the phase-space support of the perturbed
solution oscillates about the support of the steady state with period $2\pi/\omega$,
which again fits with the numerical observations.
\subsection{An Eddington-Ritter type relation} \label{ssc:eddritter}
The Eddington-Ritter relation connects the period of the pulsation
of a perturbed steady state to its central density in the context
of the Euler-Poisson system, cf.~\cite{edd,mak,ross}.
In this section we establish its analogue for the case of the Vlasov-Poisson system.
To this end we consider a spherically symmetric,
polytropic steady state $(f_0, U_0, \rho_0)$
with ansatz function of the form
\[
\varphi(E,L) = (E_0 - E)_+^k L^l,
\]
with $k$ and $l$ fixed. Then the function $y(r) = E_0 - U_0(r)$ satisfies the equation
\begin{equation}\label{y_eq_poly}
\frac{1}{r^2}(r^2 y'(r))' = - 4 \pi c_{k,l}\, r^{2l} y(r)_+^{k+l+3/2}
\end{equation}
with $c_{k,l} >0$ determined by $k$ and $l$; this is the spherically symmetric,
semilinear Poisson equation \eqref{eq:statVP}, rewritten in terms of $y$.
Its solutions
are uniquely determined by their value $y(0)>0$ at the origin, and
the induced spatial density is given by
\[
\rho_0(r) = c_{k,l} r^{2l} y(r)_+^{k+l+3/2}.
\]
If $y$ solves \eqref{y_eq_poly}, then for any $\sigma >0$ the rescaled function
\[
\sigma ^{-2} y(\sigma^{-\frac{2(k+l)+1}{2l+2}}r)
\]
solves \eqref{y_eq_poly} as well. This implies that
if $f_0^\ast$ is the steady state
induced by taking $y^\ast (0)=1$, then all other steady states
with the same ansatz function
can be obtained via
\[
f_0(x,v)= \sigma^{-\frac{2k+l}{l+1}} f_0^\ast(\sigma^{-\frac{2(k+l)+1}{2l+2}}x,\sigma v)
\]
and
\[
y(r)= \sigma^{-2} y^\ast(\sigma^{-\frac{2(k+l)+1}{2l+2}}r),
\]
in particular, $\sigma = y(0)^{-1/2}$.

Let $\mathcal A^\ast$ denote the Antonov operator corresponding to the steady state
$y^\ast$, and let $\lambda^\ast>0$ denote its smallest, positive eigenvalue so that
\[
\mathcal A^\ast h = \lambda^\ast h
\]
for some eigenfunction $h$; the central aim of the present paper is to establish the existence
of this eigenvalue. We want to rescale this eigenvalue equation in such a
way that it becomes the eigenvalue equation for a general
steady state obtained from $y^\ast$ by the rescaling above, which results in the function
$y$ and the induced operators $\mathcal A$ and $\mathcal B$.
First we note that
\begin{align*}
E^\ast(\sigma^{-\frac{2(k+l)+1}{2l+2}}x,\sigma v) 
&=
\sigma^2 E(x,v),\\
w(\sigma^{-\frac{2(k+l)+1}{2l+2}}x,\sigma v)
&=
\sigma w(x,v),\\
L(\sigma^{-\frac{2(k+l)+1}{2l+2}}x,\sigma v)
&=
\sigma^{-\frac{2k-1}{l+1}} L(x,v);
\end{align*}
the local energy $E^\ast$ corresponds to $y^\ast$, and $E$ corresponds to
the general steady state $y$.
A lengthy computation shows that
\[
(\mathcal B^\ast h)(\sigma^{-\frac{2(k+l)+1}{2l+2}}x,\sigma v)
= \sigma^{\frac{2k+4l+3}{l+1}} (\mathcal B h_\mathrm{resc})(x,v)
\]
where 
\[
h_{\mathrm{resc}}(x,v)  = h(\sigma^{-\frac{2(k+l)+1}{2l+2}}x,\sigma v).
\]
The same scaling law is true for the operator $\D^2$.
Hence we find that
\[
\mathcal A h_{\mathrm{resc}} = 
\sigma^{-\frac{2k+4l+3}{l+1}}(\mathcal A^\ast h)_{\mathrm{resc}}
= \sigma^{-\frac{2k+4l+3}{l+1}} \lambda^\ast h_{\mathrm{resc}}
\]
so that 
\[
\lambda = \sigma^{-\frac{2k+4l+3}{l+1}} \lambda^\ast
\]
is the smallest, positive eigenvalue of $\mathcal A$;
it has to be the smallest since the scaling preserves order.
Let $P$ denote the period corresponding to the eigenvalue
$\lambda$. Then $P = 2\pi/\sqrt{\lambda}$ and
expressing $\sigma$ by $y(0)$ we find that
\[
P y(0)^{\frac{k+2l+3/2}{2l+2}} = const;
\]
this relation was observed numerically in \cite{RaRe2018}.
If we consider the special case $l=0$ and recall that
then $\rho(0)=c\,y(0)^{k+3/2}$
we find that
\[
P \rho(0)^{1/2} = const,
\]
which is the Eddington-Ritter relation known from the Euler-Poisson case.


\section{The Antonov operators}\label{sc:operator}

We now give a precise definition of the Antonov operators
in the spherically symmetric and in the plane symmetric case
which have been derived in the previous section, 
including the function spaces they act on and some first properties. As explained in Section~\ref{S:LINEARISATION}, their eigenvalues, if they exist, correspond to linearized galaxy oscillations.


\subsection{The radial Antonov operator}\label{ssc:operatorradial}


The operator will be defined on a suitable subspace of the weighted, real-valued $L^2$-space
\begin{align*}
	{L^2_{\frac1{\vert\varphi'\vert}}}(\Omega_0) \coloneqq \{g\colon\Omega_0\to\R\text{ measurable}\mid \|g\|_{\frac1{\vert\varphi'\vert}}<\infty\},
\end{align*}
where
\begin{align*}
	{\|g\|^2_{\frac1{\vert\varphi'\vert}}} \coloneqq \int_{\Omega_0} \frac1{\vert\varphi'(E,L)\vert}\,\vert g(x,v)\vert^2 \diff (x,v);
\end{align*}
recall $\varphi'=\partial_E\varphi<0$ on the steady state support.
The scalar product $\langle\cdot,\cdot\rangle_{\frac1{\vert\varphi'\vert}}$ is defined accordingly. Later, we work on the radial subspace
\begin{align*}
  {L^2_{\frac1{\vert\varphi'\vert},r}(\Omega_0)}
  \coloneqq\{g\in L^2_{\frac1{\vert\varphi'\vert}}(\Omega_0)\mid
  g\text{ is spherically symmetric a.e.~on }\Omega_0\}
\end{align*}
and with odd functions
\begin{align*}
  {L^{2,odd}_{\frac1{\vert\varphi'\vert},r}}(\Omega_0)\coloneqq
  \{g\in L^2_{\frac1{\vert\varphi'\vert},r}(\Omega_0)\mid g\text{ is odd in }v \text{ a.e.~on }\Omega_0\}.
\end{align*}
Spherical symmetry on $\Omega_0$ is defined similarly to~\eqref{ss_symm_def}; note that the set $\Omega_0$ is spherically symmetric.
We will always use a lower index $r$ when restricting some function space to its radial subspace.
For an a.e.~spherically symmetric function $g\colon\Omega_0\to\R$,
we write $g(x,v)={g}(r,w,L)$ with slight abuse of notation; recall \eqref{eq:rwL}.
Similarly to \cite{ReSt20} we use the abbreviations
\begin{align}\label{E:SPACESONE}
  \Ltwo\coloneqq {L^2_{\frac1{\vert\varphi'\vert},r}(\Omega_0)}, \qquad
  \|\cdot\|_{\Ltwo}\coloneqq\|\cdot\|_{\frac1{\vert\varphi'\vert}},  \qquad
  \Ltwo^{odd}\coloneqq {L^{2,odd}_{\frac1{\vert\varphi'\vert},r}}(\Omega_0).
\end{align}

The first part of the linearized operator consists of the squared transport operator which we
define in a weak sense similar to \cite{ReSt20,St19}.
For a smooth function $g\in C^1(\Omega_0)$,
\begin{align*}
  {\D} g(x,v)\coloneqq
  v\cdot\partial_xg(x,v)-\partial_x U_0(x)\cdot\partial_vg(x,v), \quad (x,v)\in\Omega_0.
\end{align*} 
\begin{defn}\label{D:OPERATORD}
  For a spherically symmetric function $g\in L^1_{loc}(\Omega_0)$, ${\D g}$ {\em exists weakly}
  if there exists some spherically symmetric $\mu\in L^1_{loc}(\Omega_0)$ such that for
  every test function $\xi\in C^1_{c,r}(\Omega_0)$,
  \begin{align}
    \int_{\Omega_0} \frac1{\vert\varphi'(E,L)\vert}\,g\;\D\xi \diff(x,v)
    = -\int_{\Omega_0}\frac1{\vert\varphi'(E,L)\vert}\,\mu\;\xi \diff(x,v).
  \end{align}
  In this case, ${\D g\coloneqq \mu}$ {\em weakly}.
  The domain of $\D$ is defined as
  \begin{align}
    {\mathrm D(\D)}\coloneqq \{ g\in \Ltwo  \mid \D g\text{ exists weakly and }\D g\in \Ltwo  \}.
  \end{align}
\end{defn}
Obviously, the weak definition of $\D$ extends the classical one on $C^1_{c,r}(\Omega_0)$.
Furthermore, the resulting operator $\D\colon\mathrm D(\D)\to\Ltwo$ has the following properties:
\begin{prop} \label{transport}
  \begin{enumerate}[label=(\alph*)]
  \item $\D\colon\mathrm D(\D)\to \Ltwo $ is skew-adjoint as a densely defined operator
    on $\Ltwo $, i.e., $\D^*=-\D$.
  \item The kernel of $\D$ is given by
    \begin{align*}
      \ker (\D) &= \{ g\in \Ltwo \mid  \exists f \colon \R^2 \to \R \text{ s.t. }
      g(x,v) = f(E(x,v), L(x,v)) ~ \\
      & \qquad\qquad\qquad\qquad\qquad\qquad\qquad\qquad\qquad~\text{ for a.e. }
      (x,v) \in \Omega_0  \}.
    \end{align*}
  \item Let
    \begin{align*}
      {\mathrm D(\D^2)}\coloneqq \{ g\in\mathrm D(\D)\mid \D g\in\mathrm D(\D) \}.
    \end{align*}
    Then $\D^2\colon\mathrm D(\D^2)\to \Ltwo $ is self-adjoint as a densely defined
    operator on $\Ltwo $.
  \item The restricted operator $\D^2\colon \mathrm D(\D^2)\cap\Ltwo^{odd}\to \Ltwo^{odd}$
    is self-adjoint as a densely defined operator on $\Ltwo^{odd}$.
  \end{enumerate}
\end{prop}
\begin{proof}
  For parts (a) and (b) cf.~\cite{ReSt20,St19}; the same proofs work for the present steady states.
  
  Part (c) follows by von Neumann's theorem \cite[Theorem X.25]{ReSi2}
  since $\D^2=-\D^*\D$ and $\mathrm D(\D^2) = \{ g\in\mathrm D(\D)\mid \D g\in\mathrm D(\D^*) \}$
  by (a). 
 
  For (d), note that $\D$ reverses $v$-parity, which can be easily verified for smooth functions
  and then extended to the weak version of $\D$, cf. \cite[Corollary 3.17]{St19}.
  Therefore, $\D^2$ preserves $v$ parity, meaning that the operator restricted to $\Ltwo^{odd}$
  is well-defined. Its self-adjointness follows with part (c) by decomposing all functions
  involved into their even and odd parts in $v$.
\end{proof}

The domains $\mathrm D(\D)$ and $\mathrm D(\D^2)$ may seem quite abstract for now,
but we offer a more intuitive characterization in Section~\ref{ssc:essradial}.
Besides, the use of action-angle type variables in Section~\ref{ssc:essradial}
also offers alternate ways to prove the skew-adjointness and the representation of the kernel.

We now get to the second part of the linearized operator. For $g\in \Ltwo $ let
\begin{align}\label{E:BDEF}
\left( {\B g} \right)(r,w,L) \coloneqq 4\pi^2 \vert \varphi'(E,L)\vert ~\frac w{r^2}~ \J (g)(r)
\end{align}
for a.e.\ $(r,w,L)\in\Omega_0^r$, where
\begin{align*}
{\J} (g)(r) \coloneqq \int_0^\infty\int_\R w~ g(r,w,L) \diff w\diff L ,\quad \text{a.e. } r>0,
\end{align*}
$g$ is extended by $0$ to $\R^3\times\R^3$, and we used the abbreviation $E=E(r,w,L)$.
\begin{lemma} \label{bdef}
  The operator $\B\colon \Ltwo \to \Ltwo $ is well-defined, linear,
  continuous, and symmetric.
\end{lemma}
\begin{proof}
  For $g\in \Ltwo $ we have
  \begin{align*}
    \|\B g\|_{\Ltwo}^2 &= (2\pi)^6\int_0^\infty \vert \J(g)(r)\vert^2\frac1{r^4} \int_0^\infty\int_{\R} w^2\vert\varphi'(E,L)\vert\diff w\diff L \diff r\\
		&\leq C \int_0^\infty \frac1{r^2}\, \vert\J(g)(r)\vert^2\diff r\\
		&\leq C \int_0^\infty \frac1{r^2}\,\left(\int_0^\infty\int_{\R} w^2\vert\varphi'(E,L)\vert\diff w\diff L\right)\left(\int_0^\infty\int_{\R}\frac{\vert g(r,w,L)\vert^2}{\vert\varphi'(E,L)\vert}\diff w\diff L\right)\diff r\\&\leq C\|g\|_{\Ltwo}^2.
	\end{align*}
  Here, $C>0$ changed from line to line, but depends only on the fixed steady
  state $f_0$. In the first and third inequality, we used that	
	\begin{align}\label{eq:wintrho}
		\frac\pi{r^2} \int_0^\infty&\int_{\R} w^2\vert\varphi'(E,L)\vert\diff w\diff L = -\frac\pi{r^2} \int_0^\infty\int_{\R} w\,\partial_w\left[\varphi(E,L)\right]\diff w\diff L\nonumber\\&= \frac\pi{r^2} \int_0^\infty\int_{\R} \varphi(E,L)\diff w\diff L = \int_{\R^3} \varphi(E,L)\diff v = \rho_0(r)
	\end{align}
	for $r>0$ and that $\rho_0$ is bounded on the support of the steady state. The symmetry of $\B$ follows by
	\begin{align*}
		\langle \B g, h\rangle_{\Ltwo} &= (2\pi)^4 \int_0^\infty \frac1{r^2}~ \J(g)(r) ~ \J(h)(r)\diff r
	\end{align*}
	and the symmetry of the latter expression for $g,h\in \Ltwo $.
\end{proof}

Obviously, $\B g$ is odd in $v$ for every $g\in \Ltwo $, which implies that the restriction $\B\colon \Ltwo^{odd}\to \Ltwo^{odd}$ onto odd function also has the properties from above.

Equation \eqref{eq:uprime} from the appendix yields the following alternate representation of $\B g$ in the case $g\in \mathrm D(\D)$:
\begin{align} \label{eq:bpot}
	\left( \B g \right)(r,w,L) = \vert\varphi'(E,L)\vert~ w ~ U_{\D g}'(r) 
\end{align}
for a.e.~$(r,w,L)\in\Omega_0^r$.

We are now in the position to define the Antonov operator:
\begin{defn} \label{antonovdef}
  Let
  \begin{align*}
    \A \colon \mathrm D(\D^2)\to \Ltwo ,~ \A \coloneqq - \D^2 - \B .
  \end{align*}
  For $g\in \mathrm D(\D^2)$,
  \begin{align*}
    \langle \A g,g\rangle_{\Ltwo}
    &= \|\D g\|_{\Ltwo}^2 -  (2\pi)^4 \int_0^\infty \frac1{r^2}~ \left(\J(g)(r)\right)^2\diff r\\
    &= \|\D g\|_{\Ltwo}^2 - \int_0^\infty r^2 \left(U_{\D g}'(r)\right)^2\diff r\\
    &= \|\D g\|_{\Ltwo}^2 -\frac1{4\pi} \|\partial_x U_{\D g}\|_2^2.
  \end{align*}
  The latter expression is also defined on $\mathrm D(\D)$ which motivates the definition
  \begin{align*}
    \langle \A g,g\rangle_{\Ltwo}
    \coloneqq \|\D g\|_{\Ltwo}^2 - \frac1{4\pi} \|\partial_x U_{\D g}\|_2^2 \quad
    \text{for } g\in\mathrm D(\D).
  \end{align*}
\end{defn}
Up to some factor, the quadratic form of $\A$ equals the second order
variation of an important energy-Casimir functional (see \cite{GuRe2007}),
also known \cite{GuLi08,LeMeRa11} as the Antonov functional.
This is why we call $\A$ the {Antonov operator} or the linearized operator.


Before proceeding we observe that $\A$ inherits the self-adjointness of $\D^2$ and $\B$:

\begin{lemma} \label{antonovselfadjoint}
	The operator $\A\colon \mathrm D(\D^2)\to \Ltwo $ is self-adjoint as a densely defined operator on $\Ltwo $. Its restriction  $\A\colon \mathrm D(\D^2)\cap \Ltwo^{odd}\to \Ltwo^{odd}$ to odd functions is also self-adjoint as a densely defined operator on $\Ltwo^{odd}$.
\end{lemma}
\begin{proof}
	This follows by the self-adjointness of $\D^2$ on $\mathrm D(\D^2)$ and $\mathrm D(\D^2)\cap \Ltwo^{odd}$, see Proposition \ref{transport}, together with the Kato-Rellich theorem (see \cite[Chapter~13]{HiSi}) since $\B$ is bounded and symmetric due to Lemma \ref{bdef}.
\end{proof}

As explained in Section~\ref{S:LINEARISATION}, a positive eigenvalue $\lambda>0$ of the operator $\A$ with spherically symmetric, odd-in-$v$ eigenfunction induces an oscillating or pulsating solution of the linearized Vlasov-Poisson system with period $P$ given by
\begin{align}\label{eq:period}
	P = \frac{2\pi}{\sqrt\lambda}.
\end{align}

\subsection{The planar Antonov operator}\label{ssc:operatorplane}

The function spaces and operators in the plane symmetric setting are defined similarly
to the spherically symmetric case, i.e.,
\begin{align*}
  \pH\coloneqq \{g\colon\pOmega_0\to\R\text{ measurable}\mid \|g\|_{\pH}<\infty\},
\end{align*}
where
\begin{align*}
  \|g\|^2_{\pH} \coloneqq
  \int_{\pOmega_0} \frac1{\vert\pvarphi'(\pE,\pv)\vert}\,\vert g(x,v)\vert^2 \diff (x,v).
\end{align*}
The scalar product $\langle\cdot,\cdot\rangle_{\pH}$ is defined accordingly. 
Functions in $\pH$ are in general not plane symmetric in the sense of~\eqref{pl_symm_def},
but for the sake of generality we consider the operators on $\pH$ too.
Eqn.~\eqref{pl_symm_def} is valid for functions in
\begin{align}
  \p\H\coloneqq&\{g\in \pH
  \mid g\text{ is odd in }v_1 \text{ and }x \text{ a.e.~on }\pOmega_0\}\notag \\
  =&\{g\in \pH\mid \text{ for a.e. }(x,v)\in\pOmega_0\colon g(x,v)=-g(-x,v)=-g(x,-v_1,\pv)\}.
  \label{E:HODDDEF}
\end{align}
Observe that oddness in $x$ and $v_1$ is stronger than the symmetry required
in~\eqref{pl_symm_def}. However, imposing the above symmetry condition simplifies the
following analysis, and, as we shall see in Section~\ref{sc:mathur}, we obtain eigenvalues
of $\A$ in $\p\H$ nonetheless. 

As to the operators, we again start by defining the planar transport operator in a weak sense
similarly to the spherically symmetric setting in the previous section.
For a smooth function $g\in C^1(\pOmega_0)$,
\begin{align*}
  {\pD} g(x,v)\coloneqq v_1\;\partial_xg(x,v)-\pU_0'(x)\;\partial_{v_1}g(x,v), \quad (x,v)\in\pOmega_0.
\end{align*}
\begin{defn}\label{deftransportplanar}
  For $g\in L^1_{loc}(\pOmega_0)$, ${\pD g}$ {\em exists weakly} if there exists some
  $\mu\in L^1_{loc}(\pOmega_0)$ such that for every test function $\xi\in C^1_c(\pOmega_0)$,
  \begin{align}
    \int_{\pOmega_0} \frac1{\vert\pvarphi'(\pE,\pv)\vert} g\;\pD\xi \diff(x,v)
    = -\int_{\pOmega_0}\frac1{\vert\pvarphi'(\pE,\pv)\vert}\mu\;\xi \diff(x,v).
  \end{align}
  In this case, ${\pD g \coloneqq \mu}$ {\em weakly}. The domain of $\pD$ is defined as
  \begin{align}
    {\mathrm D(\pD)}\coloneqq \{ g\in \pH \mid \pD g\text{ exists weakly and }\pD g\in \pH \}.
  \end{align}
\end{defn}

The resulting operator again extends $\pD$ for smooth functions and has the
following further properties:

\begin{prop}\label{transport1D}
	\begin{enumerate}[label=(\alph*)]
		\item $\pD\colon\mathrm D(\pD)\to \pH$ is skew-adjoint as a densely defined operator on $\pH$, i.e., $\pD^*=-\pD$.
		\item The kernel of $\pD$ is given by
		\begin{align*}
			\ker (\pD) &= \{ g\in \pH \mid  \exists f \colon \R^3 \to \R \text{ s.t. }  g(x,v) = f(\pE(x,v_1), \pv) \text{ for a.e. } (x,v) \in \pOmega_0  \}.
		\end{align*}
		\item Let
		\begin{align*}
			{\mathrm D(\pD^2)}\coloneqq \{ g\in\mathrm D(\pD)\mid \pD g\in\mathrm D(\pD) \}.
		\end{align*}
		Then $\pD^2\colon\mathrm D(\pD^2)\to \pH$ is self-adjoint as a densely defined operator on $\pH$.
		\item $\pD$ reverses $v_1$-parity and $x$-parity respectively, i.e., $\pD^2$ preserves these parities. Furthermore, the restricted operator 
		$\pD^2\colon \mathrm D(\pD^2)\cap\p\H\to\p\H$ is self-adjoint as well.
	\end{enumerate}
\end{prop}
\begin{proof}
  All of the above statements can be proven similarly to the spherically symmetric setting,
  see Proposition \ref{transport}. The plane symmetric case however is not covered in
  \cite{ReSt20,St19}. 	
  The use of action-angle type variables in Section~\ref{ssc:essplane}
  offers alternative and more direct proofs, see Proposition~\ref{proptransplane}.
\end{proof}

The second part of the linearized operator is given by
\begin{align}\label{E:BARBDEF}
	\left( {\pB g} \right)(x,v) \coloneqq 4\pi \vert \pvarphi'(\pE,\pv)\vert ~v_1~ \pJ (g)(x)
\end{align}
for $g\in\pH$ and a.e.~$(x,v)\in\pOmega_0$, where
\begin{align}
	{\pJ} (g)(x) \coloneqq \int_{\R^3} v_1\, g(x,v) \diff v ,\quad \text{a.e. } x\in\R,
\end{align}
and $g$ is extended by $0$ to $\R\times\R^3$.
The resulting operator has the following properties:

\begin{lemma} \label{bdef1D}
  $\pB\colon \pH\to\pH$ is well-defined, linear, continuous, and symmetric.
\end{lemma}
\begin{proof}
  For $g\in \pH$,
  \begin{align*}
    \|\pB g\|_{\pH}^2
    &= 16\pi^2\int_{\pOmega_0} \vert\pvarphi'(\pE,\pv)\vert\; v_1^2\;
    \vert \pJ(g)(x)\vert^2 \diff (x,v)\\
    &= 16\pi^2\int_\R \vert \pJ(g)(x)\vert^2\int_{\R^3} \vert\pvarphi'(\pE,\pv)\vert\;
    v_1^2\diff v \diff x\\
    &\leq C\int_\R\vert \pJ(g)(x)\vert^2\diff x =
    C\int_\R\left( \int_{\R^3} v_1\, g(x,v)\diff v \right)^2\diff x\\
    &\leq C\int_\R \left(\int_{\R^3}v_1^2\,\vert\pvarphi'(\pE,\pv)\vert\diff v \right)
    \left(\int_{\R^3} \frac{\vert g(x,v)\vert^2}{\vert\pvarphi'(\pE,\pv)\vert}
    \diff v \right)\diff x\leq C\|g\|_{\pH}^2,
  \end{align*}
  where in the first and third inequality we used that
  \begin{align}\label{eq:v1intrho}
    \int_{\R^3} v_1^2&\;\vert\pvarphi'(\pE(x,v_1),\pv)\vert\diff v
    = - \int_{\R} v_1^2\; \palpha'(\pE(x,v_1))\diff v_1
    = - \int_{\R} v_1\; \partial_{v_1}\left[\palpha(\pE(x,v_1))\right]\diff v_1\nonumber\\
    &=\int_{\R} \palpha(\pE(x,v_1))\diff v_1=\prho_0(x),\quad x\in\R,
  \end{align}
  together with the fact that $\prho_0$ is bounded by $\prho_0(0)<\infty$.
  The symmetry of $\pB$ follows by
  \begin{align*}
    \langle \pB g, h\rangle_{\pH}
    &= 4\pi\int_{\pOmega_0} v_1\,\pJ(g)(x)\,h(x,v) \diff (x,v)
    = 4\pi\int_{\R} \pJ(g)(x)\,\pJ(h)(x)\diff x
  \end{align*}
  and the symmetry of the latter expression for $g,h\in \pH$.
\end{proof}

Since $\pB$ preserves $x$-parity and the image of $\pB$ is always an odd function in $v_1$, the restricted operator 
$\pB\colon\p\H\to\p\H$ has the analogous properties.

Eqn.~\eqref{eq:uprimeplane} yields the following alternative representation
of $\pB g$ in the case $g\in \mathrm D(\pD)$:
\begin{align} \label{eq:bpot1D}
	\left( \pB g \right)(x,v) = \vert\pvarphi'(\pE,\pv)\vert\; v_1\; \pU_{\pD g}'(x) 
\end{align}
for a.e.~$(x,v)\in\pOmega_0$. We refer to Section~\ref{ssc:potplane} in the appendix for a short discussion of the potentials induced by images of the planar transport operator.

We are now in the position to define the linearized operator in the planar setting:
\begin{defn} \label{antonovdef1D}
  Let
  \begin{align*}
    \pA \colon \mathrm D(\pD^2)\to \pH,~ \pA \coloneqq - \pD^2 - \pB .
  \end{align*}
  For $g\in \mathrm D(\pD^2)$,
  \begin{align*}
    \langle \pA g,g\rangle_{\pH} &= \|\pD g\|_{\pH}^2 -\frac1{4\pi} \| \pU_{\pD g}'\|_2^2.
  \end{align*}
  The latter expression is also defined for $g\in\mathrm D(\pD)$.
  $\pA$ is the (planar) Antonov or linearized operator.
\end{defn}

We again have the following properties:

\begin{lemma} \label{antonovselfadjoint1D}
  The operator $\pA\colon \mathrm D(\pD^2)\to \pH$ is self-adjoint as a densely defined
  operator on $\pH$. Its restriction 
  $\pA\colon \mathrm D(\pD^2)\cap \p\H\to \p\H$ to odd functions is also self-adjoint.
\end{lemma}
\begin{proof}
  We use the self-adjointness of $\pD^2$ from Proposition \ref{transport1D}
  and the symmetry and boundedness of $\pB$ from Lemma \ref{bdef1D}.
\end{proof}

\section{The essential spectra of the Antonov operators} \label{sc:ess}

In this section we rigorously describe the essential spectra of $\A$ and $\pA$, see Theorems~\ref{essspecA} and~\ref{T:ESSENTIALSPECTRUMABAR} respectively. 
Broadly speaking, the essential spectrum contains all elements of the spectrum which are not isolated eigenvalues of finite multiplicity. A key consequence of our results is the existence of the so-called principal gap in the essential spectra of $\A$ and $\pA$, which plays a crucial role in the study of the existence of isolated eigenvalues in Section~\ref{sc:mathur}.

\subsection{The essential spectrum of the radial Antonov operator}\label{ssc:essradial}

Let $\A$ denote the self-adjoint operator $\A\colon\mathrm D(\D)\to \Ltwo $
introduced in Definition~\ref{antonovdef}; most of the results below also hold
true for its restriction to odd functions.

We will see in Theorem~\ref{essspecA} that the essential spectrum of $\A=-(\D^2+\B)$
is solely determined by the one of $-\D^2$. This is why we first analyze the squared
transport operator. We do this by introducing one of the key tools in our work,
the action-angle variables.

For fixed $(r,w,L)\in \Omega_0^r$ let $\R\ni t\mapsto (R,W)(t,r,w,L)$
be the unique global solution to the characteristic system
\begin{align}\label{eq:charsystrad}
\dot R = W,\quad \dot W = -\Psi_L'(R)
\end{align}
satisfying the initial condition $(R,W)(0,r,w,L)=(r,w)$.
As discussed in Section \ref{ssc:ststradial}, $(R,W)(\cdot,r,w,L)$
is periodic with period $T(E,L)$, where $E=E(r,w,L)=\frac12w^2+\Psi_L(r)$.
We now use the variables $\theta\in[0,1]$, $(E,L)\in\Omega_0^{EL}$ given by
\begin{align*}
(r,w,L) = \left( (R,W)(\theta T(E,L), r_-(E,L),0,L), L \right).
\end{align*}
For functions $g\in \Ltwo $ we write
\begin{align*}
{g} (\theta,E,L) = g\left( (R,W)(\theta T(E,L), r_-(E,L),0,L), L \right)
\end{align*}
for a.e.~$(\theta,E,L)\in{\Omega_0^\theta}\coloneqq [0,1]\times \Omega_0^{EL}$
by slight abuse of notation. Integrals change via
\begin{align}\label{eq:actionanglevolume}
\diff x\diff v = 4\pi^2 T(E,L) \diff\theta\diff E\diff L;
\end{align}
note that the mapping
$[0,\frac12]\ni\theta\mapsto R(\theta T(E,L),r_-(E,L),0,L)\in[r_-(E,L),r_+(E,L)]$
is bijective for $(E,L)\in\mathring\Omega_0^{EL}$ with inverse given by 
\begin{align}\label{eq:tauofr}
\theta(r,E,L) \coloneqq \frac1{T(E,L)} \int_{r_-(E,L)}^r \frac{\diff s}{\sqrt{2E-2\Psi_L(s)}},
\end{align}
$r_-(E,L)\leq r\leq r_+(E,L)$.
In particular, the transformation $(x,v)\mapsto(\theta,E,L)$ is not measure preserving
as it would be in the case of \enquote{true} action-angle variables, see \cite{Arnold,LaLi,LB1994}.
However, $(\theta,E,L)$ have the same interpretation as action-angle variables,
since $(E,L)$ fix a trajectory of the stationary characteristic system and $\theta\in[0,1[$
gives the position along the characteristic flow. Actual action-angle variables have
been used in \cite{GuLi08} without the restriction to spherically symmetric functions.
%

The reason why it is useful to use $(\theta,E,L)$-variables when working with the
transport operator is that $\D$ corresponds to a $\theta$-derivative in these variables:

\begin{lemma} \label{trans1dsmooth}
  For every $g\in C^1_r(\Omega_0)$,
  \begin{align}\label{E:DFORMULAONE}
    \left(\D g\right) (\theta,E,L)
    = \frac1{T(E,L)} (\partial_\theta g) (\theta, E,L), \quad
    0\leq\theta\leq1,~ (E,L)\in\mathring\Omega_0^{EL}.
  \end{align}
  Similarly, for $g\in C^2_r(\Omega_0)$,
  \begin{align*}
    \left(\D^2 g\right) (\theta,E,L) = \frac1{T(E,L)^2} (\partial_\theta^2 g) (\theta, E,L), \quad
    0\leq\theta\leq1,~ (E,L)\in\mathring\Omega_0^{EL}.
  \end{align*}
\end{lemma}
\begin{proof}
  The assertions follow by the chain rule.
\end{proof} 

To analyze the operator $-\D^2$ and its spectrum, we need a formula similar to the ones
derived in Lemma \ref{trans1dsmooth} for all functions in $\mathrm D(\D)$ and
$\mathrm D (\D^2)$. We therefore define
\begin{align}\label{eq:defh1tau}
{H^1_\theta}\coloneqq \left\{ y\in H^1(]0,1[) \mid y(0)=y(1) \right\};
\end{align}
note that $H^1(]0,1[)\hookrightarrow C([0,1])$,
i.e., the boundary condition is imposed for the continuous representative. 
In the following lemma we provide an alternate description of $\mathrm D(\D)$
(see Definition~\ref{D:OPERATORD}) and the extension of the
formula~\eqref{E:DFORMULAONE} to $\mathrm D(\D)$.

\begin{lemma} \label{trans1d} It holds that
  \begin{align*}
    \mathrm D(\D) = \{ g\in \Ltwo  \mid&\; \text{for a.e. }(E,L)\in\Omega_0^{EL},\ 
    g(\cdot,E,L)\in H^1_\theta, \\
    &\;\text{and } \int_{\Omega_0^{EL}} \frac{T(E,L)^{-1}}{\vert \varphi'(E,L)\vert}
    \int_0^1 \vert \partial_\theta g(\theta,E,L)\vert^2\diff\theta\diff(E,L)<\infty  \}.
  \end{align*}
  If $g\in\mathrm D(\D)$,
  \begin{align*}
    \left(\D g\right) (\theta,E,L) = \frac1{T(E,L)} (\partial_\theta g) (\theta, E,L)
  \end{align*}
  for a.e.~$ (\theta,E,L)\in\Omega_0^\theta$.
\end{lemma}
\begin{proof}
	First, consider $g\in\mathrm D(\D)$. By the weak definition of $\D$, a change of variables and Lemma \ref{trans1dsmooth}, we obtain the following for every test function $\xi\in C^1_{c,r}(\Omega_0)$:
	\begin{align*}
	4\pi^2 \int_{\Omega_0^{EL}} & \frac1{\vert\varphi'(E,L)\vert} \int_0^1 \partial_\theta\xi (\theta,E,L)\; g(\theta,E,L)\diff\theta\diff(E,L)  \\
	&=4\pi^2 \int_{\Omega_0^{EL}} \frac{T(E,L)}{\vert\varphi'(E,L)\vert} \int_0^1 (\D \xi) (\theta,E,L)\; g(\theta,E,L)\diff\theta\diff(E,L)  \\
	&= \int_{\Omega_0} \frac1{\vert\varphi'(E,L)\vert} \D \xi \; g \diff (x,v) = - \int_{\Omega_0} \frac1{\vert\varphi'(E,L)\vert} \xi \; \D g \diff (x,v)  \\
	&= - 4\pi^2 \int_{\Omega_0^{EL}} \frac{T(E,L)}{\vert\varphi'(E,L)\vert} \int_0^1 \xi (\theta,E,L)\; (\D g)(\theta,E,L)\diff\theta\diff(E,L).
	\end{align*}
	We now choose $\xi$ to be factorized in $\theta$ and $(E,L)$, i.e.,
	\begin{align*}
	\xi(\theta,E,L) = \zeta(\theta)\;\chi(E,L) ,\quad (\theta,E,L)\in\Omega_0^\theta,
	\end{align*}
	where $\zeta\in C^\infty_c(]0,1[)$ and $\chi\in C^\infty_c(\mathring\Omega_0^{EL})$. Note that every such choice of $\zeta$ and $\chi$ induces some $\xi$ in $C^1_{c,r}(\Omega_0)$. Inserting this into the above calculation yields
	\begin{align*}
	\int_{\Omega_0^{EL}} & \chi(E,L) \frac1{\vert\varphi'(E,L)\vert} \int_0^1 \dot\zeta(\theta)\; g(\theta,E,L)\diff\theta\diff(E,L)\\
	&= - \int_{\Omega_0^{EL}} \chi(E,L) \frac{T(E,L)}{\vert\varphi'(E,L)\vert} \int_0^1 \zeta(\theta)\; (\D g)(\theta,E,L)\diff\theta\diff(E,L).
	\end{align*}
	Since this holds true for every $\chi\in C^\infty_c(\mathring\Omega_0^{EL})$, it follows that
	\begin{align}\label{eq:helpertrans1d}
	\int_0^1 \dot\zeta(\theta)\; g(\theta,E,L)\diff\theta = - T(E,L) \int_0^1 \zeta(\theta)\; (\D g)(\theta,E,L)\diff\theta
	\end{align}
	for a.e.\footnote{The set of measure zero can be chosen independently of the test function $\zeta\in C^\infty_c(]0,1[)$ by considering a countable subset of $C^\infty_c(]0,1[)$ which is dense with respect to $\|\cdot\|_{H^1(]0,1[)}$.}~$(E,L)\in\Omega_0^{EL}$. Since $\zeta\in C^\infty_c(]0,1[)$ is arbitrary, this in turn means that $g(\cdot,E,L)$ is weakly differentiable for a.e.~$(E,L)\in\Omega_0^{EL}$ with
	\begin{align*}
	\partial_\theta g(\cdot,E,L) = T(E,L) ~ (\D g)(\cdot,E,L).
	\end{align*} 
	In addition, 
	\begin{align*}
	4\pi^2\int_{\Omega_0^{EL}}& \frac{T(E,L)^{-1}}{\vert \varphi'(E,L)\vert} \int_0^1 \vert \partial_\theta g(\theta,E,L)\vert^2\diff\theta\diff(E,L)  \\
	&= 4\pi^2\int_{\Omega_0^{EL}} \frac{T(E,L)}{\vert \varphi'(E,L)\vert} \int_0^1 \vert (\D g)(\theta,E,L)\vert^2\diff\theta\diff(E,L)= \|\D g\|_{\Ltwo}^2<\infty,
	\end{align*}
	in particular, $\partial_\theta g(\cdot,E,L)\in L^2(]0,1[)$ for a.e.~$(E,L)\in\Omega_0^{EL}$. What remains to show is the boundary condition $g(0,E,L)=g(1,E,L)$. Observe that \eqref{eq:helpertrans1d} also holds true for $\zeta(\theta)\coloneqq1$, $0\leq\theta\leq1$, since this still leads to a test function $\xi\in C^1_{c,r}(\Omega_0)$. Therefore, integrating by parts yields
	\begin{align*}
	0 &=\frac1{T(E,L)} \int_0^1 \dot\zeta(\theta)\; g(\theta,E,L)\diff\theta = - \int_0^1 \zeta(\theta)\; \partial_\theta g(\theta,E,L)\diff\theta  \\
	&= \int_0^1 \dot\zeta(\theta)\; g(\theta,E,L)\diff\theta - \zeta (\theta) g(\theta,E,L)\Big|_{\theta=0}^{\theta=1}  \\
	&= g(0,E,L)-g(1,E,L)
	\end{align*}
	for a.e.~$(E,L)\in\Omega_0^{EL}$, i.e., $g(\cdot,E,L)\in H^1_\theta$ and we have proven the first implication.
	
	Conversely, let $g\in \Ltwo $ be such that $g(\cdot,E,L)\in H^1_\theta$ for a.e.~$(E,L)\in\Omega_0^{EL}$ and 
	\begin{align}\label{eq:helpertrans1d2}
	\int_{\Omega_0^{EL}} \frac{T(E,L)^{-1}}{\vert \varphi'(E,L)\vert} \int_0^1 \vert \partial_\theta g(\theta,E,L)\vert^2\diff\theta\diff(E,L)<\infty.
	\end{align}
	For any test function $\xi\in C^1_{c,r}(\Omega_0)$ and $(E,L)\in\mathring\Omega_0^{EL}$ we obviously have $\xi(\cdot,E,L)\in C^1([0,1])$ with $\xi(0,E,L)=\xi(1,E,L)$. Thus,
	\begin{align*}
	\int_{\Omega_0}\frac1{\vert\varphi'(E,L)\vert} g\; \D \xi \diff (x,v) & =4\pi^2 \int_{\Omega_0^{EL}} \frac1{\vert\varphi'(E,L)\vert} \int_0^1g(\theta,E,L) \; \partial_\theta\xi (\theta,E,L)\diff\theta\diff(E,L)\\
	&=-4\pi^2 \int_{\Omega_0^{EL}} \frac1{\vert\varphi'(E,L)\vert} \int_0^1\partial_\theta g(\theta,E,L) \; \xi (\theta,E,L)\diff\theta\diff(E,L)\\
	&= - \int_{\Omega_0}\frac1{\vert\varphi'(E,L)\vert}~\frac{(\partial_\theta g)}{T(E,L)} \; \xi \diff (x,v),
	\end{align*} 
	where we again used Lemma \ref{trans1dsmooth} and integrated by parts in the inner integral in the weak sense; note that the boundary terms vanish since $g(\cdot,E,L)\in H^1_\theta$. By the weak definition of $\D$, the above means that $\D g$ exists weakly and 
	\begin{align*}
	(\D g)(\theta,E,L) = \frac1{T(E,L)} \partial_\theta g(\theta,E,L)
	\end{align*}
	for a.e.~$(\theta,E,L)\in\Omega_0^\theta$.
        Eqn.~\eqref{eq:helpertrans1d2} and a change of variables
        then also show $\D g\in \Ltwo $, i.e., $g\in\mathrm D(\D)$.
\end{proof}


\begin{remark}[Oddness-in-$v$]
Recall the definition~\eqref{E:SPACESONE} of the space $\Ltwo^{odd}$ of odd-in-$v$ functions in $\Ltwo$.
It is of interest to describe the oddness with respect to the $v$-coordinate in the action-angle coordinates.
To that end let 
	\begin{align*}
	{L^{2,odd}}(]0,1[) \coloneqq \{ y\in L^2(]0,1[) \mid y(\theta) = -y(1-\theta) \text{ for a.e. } \theta\in]0,1[ \}.
	\end{align*}
It is then easy to check that  for every $g\in \Ltwo $,
	\begin{align*}
	g\in \Ltwo^{odd} \quad\Leftrightarrow\quad g(\cdot,E,L)\in L^{2,odd}(]0,1[) \text{ for a.e. }(E,L)\in\Omega_0^{EL}.
	\end{align*}
\end{remark}

To obtain a representation of the domain of $\D^2$ similar to Lemma~\ref{trans1d}, let
\begin{align}\label{eq:defh2tau}
  {H^2_\theta} \coloneqq
  \left\{ y\in H^2(]0,1[)\mid y(0)=y(1) \text{ and } \dot y(0)=\dot y(1) \right\}  
    =\left\{ y\in H^1_\theta\mid \dot y\in H^1_\theta \right\} ;
\end{align}
note $H^2(]0,1[)\hookrightarrow C^1([0,1])$, i.e., the boundary conditions are imposed for the continuously differentiable representative. Then, by analogy to 
Lemma~\ref{trans1d} we obtain

\begin{cor} \label{trans1dsquare}
  It holds that
  \begin{align*}
    \mathrm D(\D^2) =
    \{ g\in \Ltwo \mid& \; \text{for a.e. }(E,L)\in\Omega_0^{EL},\
    g(\cdot,E,L)\in H^2_\theta, \\
    &\; \text{and } \sum_{j=1}^2 \int_{\Omega_0^{EL}} \frac{T(E,L)^{1-2j}}{\vert \varphi'(E,L)\vert}
    \int_0^1 \vert \partial_\theta^j g(\theta,E,L)\vert^2\diff\theta\diff(E,L)<\infty\}.
  \end{align*}
  Furthermore, if $g\in\mathrm D(\D^2)$,
  \begin{align*}
    \left(\D^2 g\right) (\theta,E,L) = \frac1{T(E,L)^2} (\partial_\theta^2 g) (\theta, E,L)
  \end{align*}
  for a.e.~$ (\theta,E,L)\in\Omega_0^\theta$.
\end{cor}
\begin{proof}
  Both implications follow from Lemma~\ref{trans1d}, applied twice.
\end{proof}

We next apply Lemma~\ref{trans1d} to obtain the following useful lemma:
\begin{lemma} \label{transinversetauel}
  For every $h\in \Ltwo $ with $h\perp\ker(\D)$ there exists $g\in\mathrm D(\D)$
  such that $\D g=h$. In particular,
  \begin{align}
    \ker(\D)^\perp = \mathrm{im}(\D).
  \end{align}
\end{lemma}
Note that as a skew-adjoint operator $\D$ satisfies $\ker(\D)^\perp=\overline{\mathrm{im}(\D)}$,
see e.g.~\cite[Corollary~2.18]{Brezis}. Lemma~\ref{transinversetauel}
thus shows that the range of $\D$ is closed, cf.~\cite[Theorem~2.19]{Brezis}.
\begin{proof}
  We define the spherically symmetric function $g\colon\Omega_0\to\R$ via
  \begin{align*}
    g (\theta,E,L)\coloneqq T(E,L) \int_0^\theta h(s,E,L)\diff s
    \quad\text{for a.e. } (\theta,E,L)\in\Omega_0^\theta
  \end{align*}
  and apply Lemma~\ref{trans1d} to verify the claimed properties of $g$.
  First, by the weak version of the main theorem of calculus, it follows that
  $g(\cdot,E,L)$ is weakly differentiable for a.e.~$(E,L)\in\Omega_0^{EL}$ with
  \begin{align*}
    \partial_\theta g(\cdot,E,L) = T(E,L)~ h(\cdot,E,L).
  \end{align*}
  In addition,
  \begin{align*}
    \vert g(\theta,E,L)\vert^2 \leq T(E,L)^2 \int_0^1 \vert h(s,E,L)\vert^2\diff s,
  \end{align*}
  i.e., $g(\cdot,E,L)\in L^2(]0,1[)$ for a.e.~$(E,L)\in\Omega_0^{EL}$ and $g\in \Ltwo $, since
    \begin{align*}
      \|g\|_{\Ltwo}^2 &=
      4\pi^2 \int_{\Omega_0^{EL}} \frac{T(E,L)}{\vert\varphi'(E,L)\vert}
      \int_0^1 \vert g(\theta,E,L)\vert^2 \diff\theta \diff(E,L) \\
      &\leq 4\pi^2 \int_{\Omega_0^{EL}} \frac{T(E,L)^3}{\vert\varphi'(E,L)\vert}
      \int_0^1 \vert h(s,E,L)\vert^2 \diff s \diff(E,L)\\
      &\leq {\sup_{\mathring\Omega_0^{EL}}}^2(T) \|h\|_{\Ltwo}^2;
    \end{align*}
    note that $T$ is bounded by Proposition \ref{Tbounded}.
    Furthermore, $\partial_\theta g(\cdot,E,L)\in L^2(]0,1[)$,
    i.e., $g(\cdot,E,L) \in H^1(]0,1[)$ for a.e.~$(E,L)\in\Omega_0^{EL}$. Moreover,
	\begin{align*}
	g(1,E,L)= \int_0^1 h(s,E,L)\diff s = 0 = g(0,E,L),
	\end{align*}
	for a.e.~$(E,L)\in\Omega_0^{EL}$, since $h\perp\ker(\D)$. Lastly,
	\begin{align*}
	4\pi^2&  \int_{\Omega_0^{EL}} \frac{T(E,L)^{-1}}{\vert \varphi'(E,L)\vert} \int_0^1 \vert \partial_\theta g(\theta,E,L)\vert^2\diff\theta\diff(E,L) \\ 
	&= 4\pi^2 \int_{\Omega_0^{EL}} \frac{T(E,L)}{\vert \varphi'(E,L)\vert} \int_0^1 \vert h(\theta,E,L)\vert^2\diff\theta\diff(E,L) = \| h\|_{\Ltwo}^2 < \infty.\qedhere
	\end{align*}
\end{proof}

Because of Corollary~\ref{trans1dsquare} we now turn our attention to the one-dimensional
Laplacian as an operator on $H^2_\theta$:

\begin{lemma} \label{laplace1d}
  The operator
  \begin{align*}
    -\partial_\theta^2 \colon H^2_\theta\to L^2(]0,1[), ~y\mapsto-\ddot y
  \end{align*}
  is self-adjoint as a densely defined operator on the Hilbert space $L^2(]0,1[)$.
  Its spectrum is given by 
  \begin{align*}
    \sigma(-\partial_\theta^2) = (2\pi\N_0)^2 \coloneqq \{ (2\pi k)^2\mid k\in\N_0 \}.
  \end{align*}
  In particular, every element of the spectrum is an eigenvalue.
  The eigenspace of the eigenvalue $0$ consists of all constant functions.
  For $k\in\N$, the eigenspace of $(2\pi k)^2$ is
  $\{ c_1 \cos(2\pi k\cdot) + c_2\sin(2\pi k\cdot)\mid c_1,c_2\in\R\}$.
\end{lemma}
\begin{proof}
  It is straight forward to verify that $-\partial_\theta^2$ is self-adjoint;
  note that the boundary conditions built into $H^2_\theta$ cause all boundary terms
  to vanish when integrating by parts.
	
  Using basic ODE theory, it can be easily verified that the set of proper eigenvalues
  of $-\partial_\theta^2$ equals $(2\pi\N_0)^2$. What remains to show is that the spectrum
  of $-\partial_\theta^2$ does not contain any other elements. One way to do this is by
  explicitly deriving the resolvent operator for $\lambda\notin(2\pi\N_0)^2$ by expanding
  all functions involved into their Fourier series, since the eigenfunctions of the
  operator form an orthonormal basis of $L^2(]0,1[)$ (up to a factor).
  The latter techniques are also applied in the case of Riesz-spectral
  operators, see for example \cite[Theorem 2.3.5]{CuZw}.
\end{proof}

Similar to the above Lemma one can also show that the operator
\begin{align*}
-\partial_\theta^2 \colon H^2_\theta\cap L^{2,odd}(]0,1[)\to L^{2,odd}(]0,1[), ~y\mapsto-\ddot y
\end{align*}
is self-adjoint as a densely defined operator on the Hilbert space $L^{2,odd}(]0,1[)$ and that its spectrum is $(2\pi\N)^2$; note that non-zero constant functions are not part of $H^2_\theta\cap L^{2,odd}(]0,1[)$, which is why $0$ is not an eigenvalue of the operator restricted to odd functions.

Combining Corollary~\ref{trans1dsquare},
i.e., \enquote{$-\D^2 = -\frac1{T(E,L)^2}\partial_\theta^2$}, and Lemma \ref{laplace1d},
i.e., \enquote{$\sigma(-\partial_\theta^2)=(2\pi\N_0)^2$},
allows us to explicitly determine the spectrum of $-\D^2$:

\begin{theorem} \label{transsquarespectrum}
  The spectrum of the self-adjoint operator $-\D^2\colon \mathrm D(\D^2)\to \Ltwo $ is
  \begin{align*}
    \sigma(-\D^2) = \overline{ \left(\frac{2\pi\N_0}{T(\mathring\Omega_0^{EL})}\right)^2 }
    \coloneqq \overline{ \left\{ \left( \frac{2\pi k}{T(E,L)} \right)^2 ~\Big|~ k\in\N_0,~(E,L)
      \in\mathring\Omega_0^{EL} \right\} }.
  \end{align*}
  Furthermore, the spectrum is purely essential, i.e.,
  \begin{align*}
    \sigma_{ess}(-\D^2) = \sigma(-\D^2).
  \end{align*}
\end{theorem}
\begin{proof}
	We first show $\left( 2\pi k~ T(E^*,L^*)^{-1} \right)^2\in\sigma_{ess}(-\D^2)$ for fixed $k\in\N$ and $(E^*,L^*)\in\mathring\Omega_0^{EL}$. The idea is that 
	\begin{align*}
	(\theta,E,L)\mapsto \delta_{(E^*,L^*)}(E,L) \; \sin(2\pi k\theta)
	\end{align*}
	defines an \enquote{eigendistribution} of $\left( 2\pi k~ T(E^*,L^*)^{-1} \right)^2$ (where $\delta$ denotes Dirac's delta distribution) and that by approximating, we can show that $\left( 2\pi k~ T(E^*,L^*)^{-1} \right)^2$ is indeed an approximate eigenvalue.\\
	We therefore prove $\left( 2\pi k~ T(E^*,L^*)^{-1} \right)^2\in\sigma_{ess}(-\D^2)$ by verifying Weyl's criterion \cite[Theorem~7.2]{HiSi}, i.e., we construct a sequence $(g_j)_{j\in\N}\subset\mathrm D(\D^2)$ with
	\begin{enumerate}[label=(\roman*)]
		\item $\displaystyle \|g_j\|_{\Ltwo} = 1$ for every $j\in\N$,
		\item $\displaystyle \left\| -\D^2g_j - \left( 2\pi k~ T(E^*,L^*)^{-1} \right)^2 g_j \right\|_{\Ltwo}\to0$ as $j\to\infty$,
		\item $\displaystyle g_j\weakto0$ in $\displaystyle \Ltwo $ as $j\to\infty$.
	\end{enumerate}
	To construct such a Weyl sequence, we approximate the Dirac distribution as follows: For $j\in\N$ let $\chi_j\colon\R^2\to\R$ be such that
	\begin{enumerate}[label=(\greek*)]
		\item $\displaystyle\supp(\chi_j)\subset \mathring\Omega_0^{EL}\cap B_{\frac1j}(E^*,L^*)$,
		\item $\displaystyle \int_{\Omega_0^{EL}} \chi_j^2(E,L)\diff (E,L) = \frac1{2\pi^2}$.
	\end{enumerate}
	It is obvious that such $\chi_j$ exist; they can be explicitly defined by a rescaling scheme (at least for sufficiently large $j$). Note however that $\chi_j\not\to\delta_{(E^*,L^*)}$ in the distributional sense.\\ 
	For $j\in\N$ we now define the spherically symmetric function $g_j\colon\Omega_0\to\R$ by
	\begin{align*}
	g_j (\theta,E,L)\coloneqq \sqrt{\frac{\vert\varphi'(E,L)\vert}{T(E,L)}}~\chi_j(E,L)\; \sin(2\pi k\theta),\quad(\theta,E,L)\in\Omega_0^\theta.
	\end{align*}
	Then $(g_j)_{j\in\N}$ is indeed a Weyl sequence. First, $g_j\in\mathrm D(\D^2)$ for $j\in\N$ by Corollary~\ref{trans1dsquare}, note that $T>0$ is continuous on $\mathring\Omega_0^{EL}$ by Lemma \ref{partperiodcontinuous} and therefore $\frac1T$ is bounded on the support of $\chi_j$.
   Now to properties (i)--(iii):
	\begin{enumerate}[label=(\roman*)]
		\item For every $j\in\N$ changing variables yields
		\begin{align*}
		\|g_j\|_{\Ltwo}^2 &= 4\pi^2\int_{\Omega_0^{EL}} \frac{T(E,L)}{\vert\varphi'(E,L)\vert} \int_0^1 \vert g(\theta,E,L)\vert^2\diff\theta\diff(E,L) \\
		&= 4\pi^2 \int_{\Omega_0^{EL}}\chi_j^2(E,L) \int_0^1\sin^2(2\pi k\theta)\diff\theta\diff(E,L) = 1.
		\end{align*}
		\item It follows from Corollary~\ref{trans1dsquare} that
		\begin{align*}
		(-\D^2g_j) (\theta, E,L) = \left(\frac{2\pi k}{T(E,L)}\right)^2 g_j(\theta,E,L) \text{ for a.e. } (\theta,E,L)\in\Omega_0^\theta.
		\end{align*}
		Thus,
		\begin{align*}
		\big\| &-\D^2g_j - \left( 2\pi k~ T(E^*,L^*)^{-1} \right)^2 g_j \big\|_{\Ltwo}^2  \\
		&=4\pi^2 \int_{\Omega_0^{EL}} \frac{T(E,L)}{\vert\varphi'(E,L)\vert} \left| \left(\frac{2\pi k}{T(E,L)}\right)^2 - \left(\frac{2\pi k}{T(E^*,L^*)}\right)^2 \right|^2 \int_0^1 \left| g_j(\theta,E,L)\right|^2\diff\theta\diff(E,L)  \\
		&= (2\pi)^6 k^4 \int_{\Omega_0^{EL}} \chi_j^2(E,L) \left| \frac1{T(E,L)^2} - \frac1{T(E^*,L^*)^2} \right|^2 \int_0^1\sin^2(2\pi k\theta)\diff\theta\diff(E,L)  \\
		&= 2^5\pi^6k^4 \int_{\Omega_0^{EL}} \chi_j^2(E,L) \left| \frac1{T(E,L)^2} - \frac1{T(E^*,L^*)^2} \right|^2 \diff(E,L).
		\end{align*}
		Since $T>0$ is continuous on $\mathring\Omega_0^{EL}$ by Lemma \ref{partperiodcontinuous}, $\left| {T(E,L)^{-2}} - {T(E^*,L^*)^{-2}} \right|$ tends to zero for $(E,L)\in\supp(\chi_j)$ as $j\to\infty$ by ($\alpha$). Together with ($\beta$) this implies (ii).
		\item For every $h\in \Ltwo $ the Cauchy-Schwarz inequality yields
		\begin{align*}
		\left| \langle g_j,h\rangle_{\Ltwo} \right| \leq \|g_j\|_{\Ltwo} \left( 4\pi^2 \int_{\supp(\chi_j)} \frac{T(E,L)}{\vert\varphi'(E,L)\vert} \int_0^1 \vert h(\theta,E,L)\vert^2\diff\theta\diff(E,L) \right)^{\frac12},
		\end{align*}
		with $\|g_j\|_{\Ltwo}=1$ and the right integral tending to zero as $j\to\infty$. Due to the Riesz representation theorem we obtain (iii).
	\end{enumerate}
	By a completely similar proof, we also get $0\in\sigma_{ess}(\D^2)$. More precisely, we just have to replace $\sin(2\pi k\theta)$ in the definition of $g_j$ by a non-zero, constant function. Note that $0$ is in fact an eigenvalue of $-\D^2$, but its multiplicity is not finite (corresponding to the fact that the eigenspace $\ker(\D)$ is infinite dimensional).
	
	Altogether, we have proven
	\begin{align*}
	\left( \frac{2\pi\N_0}{T(\mathring\Omega_0^{EL})} \right)^2 \subset\sigma_{ess}(-\D^2).
	\end{align*}
	Since the spectrum of an operator is always closed and the boundary values being non-isolated, we obtain 
	\begin{align*}
	\overline{\left( \frac{2\pi\N_0}{T(\mathring\Omega_0^{EL})} \right)^2} \subset\sigma_{ess}(-\D^2) \subset\sigma(-\D^2).
	\end{align*}
	It remains to show that
	\begin{align*}
	\sigma(-\D^2) \subset \overline{\left( \frac{2\pi\N_0}{T(\mathring\Omega_0^{EL})} \right)^2}.
	\end{align*}
	Fix an arbitrary $\lambda\in\R\setminus \overline{\left( \frac{2\pi\N_0}{T(\mathring\Omega_0^{EL})} \right)^2}$; note $\sigma(-\D^2)\subset\R$ by the self-adjointness of $-\D^2$.
        Due to the radial particle period $T$ being bounded away from zero by
        Proposition~\ref{Tbounded}, there exists $c>0$ such that
	\begin{align} \label{eq:helper}
	\mathrm{dist} \left( \lambda\, T(E,L)^2, (2\pi\N_0)^2 \right)\geq c \text{ for every } (E,L)\in\mathring\Omega_0^{EL}.
	\end{align}
	Since $(2\pi\N_0)^2=\sigma(-\partial_\theta^2)$, we then obtain the following estimate for the one-dimensional resolvent operator $ \left( -\partial_\theta^2 - \lambda\, T(E,L)^2 \right)^{-1} \colon L^2(]0,1[)\to H^2_\theta$ by \eqref{eq:helper}:
	\begin{align*}
	\left\| \left( -\partial_\theta^2 - \lambda\, T(E,L)^2 \right)^{-1} y \right\|_2 \leq \frac1c \|y\|_2 \quad\text{for } y\in L^2(]0,1[),~ (E,L)\in\mathring\Omega_0^{EL},
	\end{align*}
	see for example \cite[Theorem 5.8]{HiSi}. Therefore, inserting $y=\left( -\partial_\theta^2 - \lambda\, T(E,L)^2 \right)z$ in the above estimate yields
	\begin{align} \label{eq:helper2}
	\|z\|_2\leq\frac1c \left\| \left( -\partial_\theta^2 - \lambda\, T(E,L)^2 \right)z \right\|_2 \quad\text{for } z\in H^2_\theta ,~(E,L)\in\mathring\Omega_0^{EL}.
	\end{align}
	Now, suppose $\lambda\in\sigma(-\D^2)$. Then, by a weaker version of Weyl's criterion (see \cite[Theorem~5.10]{HiSi}), there exists a sequence $(g_j)_{j\in\N}\subset \mathrm D(\D^2)$ such that $\|g_j\|_{\Ltwo} = 1$ for $j\in\N$ and $\left\| -\D^2g_j - \lambda g_j \right\|_{\Ltwo}\to0$ as $j\to\infty$. However,
	\begin{align*}
		\big\| -&\D^2g_j - \lambda g_j \big\|_{\Ltwo}^2 \\
		&= 4\pi^2 \int_{\Omega_0^{EL}} \frac{T(E,L)^{-3}}{\vert \varphi'(E,L)\vert} \int_0^1 \left| -\partial_\theta^2 g_j(\theta,E,L) - \lambda\, T(E,L)^2 g_j(\theta,E,L) \right|^2\diff\theta\diff(E,L)
	\end{align*}
	by Corollary~\ref{trans1dsquare}. Then, applying \eqref{eq:helper2} to
        $g_j(\cdot,E,L)\in H^2_\theta$ for a.e.~$(E,L)\in\Omega_0^{EL}$ yields
	\begin{align*}
	\big\| -\D^2g_j - \lambda g_j \big\|_{\Ltwo}^2
	&\geq 4\pi^2 c^2 \int_{\Omega_0^{EL}} \frac{T(E,L)^{-3}}{\vert \varphi'(E,L)\vert}
        \int_0^1 \left| g_j(\theta,E,L) \right|^2\diff\theta\diff(E,L)\\
	&\geq \frac{c^2}{\sup_{\mathring\Omega_0^{EL}}^4(T)} \|g_j\|_{\Ltwo}^2
        = \frac{c^2}{\sup_{\mathring\Omega_0^{EL}}^4(T)} > 0,
	\end{align*}
	and therefore the desired contradiction.
\end{proof}

In particular, we obtain the following estimate:
\begin{cor}[A Poincar\'e inequality] \label{estimatetranssquare}
	For every $g\in\mathrm D(\D)$ with $g\perp\ker(\D)$ we have
	\begin{align*}
	\|\D g\|_{\Ltwo}^2 \geq \frac{4\pi^2}{\sup^2_{\mathring\Omega_0^{EL}}(T)} \|g\|_{\Ltwo}^2.
	\end{align*}
\end{cor}
\begin{proof}
	Using the fact that $\ker(\D)$ is the eigenspace of $0$, the statement on $\mathrm D(\D^2)$ follows immediately by the skew-adjointness of $\D$ together with the spectral properties of $-\D^2$ stated in the above theorem, see for example \cite[Proposition~5.12]{HiSi}.
	
	Mollifying extends this estimate to $\mathrm D(\D)$.
        One way to do this is to express $g$ in $(\theta,E,L)$-coordinates and mollify the $\theta$-function for a.e.~$(E,L)\in\mathring\Omega_0^{EL}$, for example by using Fourier partial series. Note that the mollification has to preserve the property $g\perp\ker(\D)$. 
\end{proof} 

Alternatively, an estimate as above can also be shown by more fundamental techniques without determining the spectrum of $-\D^2$, see e.g.\ \cite[proof of Theorem~2.3]{ReSt20}.

We now use the representation of $\sigma_{ess}(-\D^2)$ from Theorem~\ref{transsquarespectrum} to explicitly determine $\sigma_{ess}(\A)$ as well. In fact, the key property of the essential spectrum is that it is stable under certain types of perturbations. In our situation, this means that the essential spectrum of $\A=-\left(\D^2+\B\right)$ (see \eqref{E:BDEF} for the definition of $\B$) is
equal to the one of $-\D^2$:

\begin{theorem}\label{essspecA}
  The operator $(-\B)$ is relatively $(-\D^2)$-compact, cf.\ \cite[Definition 14.1]{HiSi}.
  Therefore, by the Weyl theorem \cite[Theorem 14.6]{HiSi},
	\begin{align*}
	\sigma_{ess}(\A) = \sigma_{ess} (-\D^2).
	\end{align*}
	Thus, by Theorem~\ref{transsquarespectrum},
	\begin{align}\label{eq:essspecA}
	\sigma_{ess}(\A) =  \overline{ \left(\frac{2\pi\N_0}{T(\mathring\Omega_0^{EL})}\right)^2 }, 
	\end{align}
	where $\A$ denotes the unrestricted operator $\A\colon\mathrm D(\D^2)\to\Ltwo$. Similarly,
	\begin{align}\label{eq:essspecAodd}
	\sigma_{ess}(\A\big|_{\Ltwo^{odd}}) = \overline{ \left(\frac{2\pi\N}{T(\mathring\Omega_0^{EL})}\right)^2 } .
	\end{align}
	Here, we use the convention $\min(\N)=1$, $\N_0=\N\cup\{0\}$.
\end{theorem}
\begin{proof}
  By Lemma \ref{bdef}, $-\B$ is continuous on $\Ltwo $, i.e.,
  relatively $(-\D^2)$-bounded with relative bound $0$. Furthermore, $(-\D^2)$
  is self-adjoint and $\rho(-\D^2)\neq \emptyset$.
  In this situation, the relative $(-\D^2)$-compactness of $-\B$
  is equivalent to the following
  (see for example \cite[III~Definition~2.15 and III~Exercise~2.18(1)]{EnNa}):
  \begin{align*}
    -\B\colon \left( \mathrm D(\D^2), \|\D^2\cdot\|_{\Ltwo} + \|\cdot\|_{\Ltwo} \right)
    \to \Ltwo  \quad\text{ is compact};
  \end{align*}
  the domain of $-\B$ in this statement is $\mathrm D(\D^2)$,
  equipped with the graph norm of $(-\D^2)$.
	
  To prove this statement, let $(g_k)_{k\in\N}\subset\mathrm D(\D^2)$ be such that $(g_k)_{k\in\N},~(\D^2g_k)_{k\in\N} \subset \Ltwo $ are bounded. Corollary \ref{estimatetranssquare}---note $\D g_k\in\mathrm{im}(\D)\subset\ker(\D)^\perp$ for $k\in\N$---then yields that $(\D g_k)_{k\in\N} \subset \Ltwo $ is bounded as well. Thus, by \eqref{eq:poth2}, $(U_{\D g_k})_{k\in\N}\subset C\cap H^2(\R^3)$ is bounded in $H^2(\R^3)$. In addition, \eqref{eq:uprime} yields $\supp(\partial_x U_{\D g_k})\subset B_{R_0}(0)$ for every $k\in\N$. By the compact embedding $H^2(B_{R_0}(0))\Subset H^1(B_{R_0}(0))$, $(\partial_x U_{\D g_k})_{k\in\N}$ is strongly convergent in $L^2(\R^3;\R^3)$, at least after extracting a subsequence. Representing $\B g_k$ as in \eqref{eq:bpot} and rewriting $\int w^2\vert\varphi'(E,L)\vert\diff(w,L)$ as in \eqref{eq:wintrho}, it follows that $(\B g_k)_{k\in\N}$ is a Cauchy sequence in $\Ltwo $,
  and thus strongly convergent in $\Ltwo $.
	
	The analogous relative compactness result also holds true when we restrict all operators to $H^{odd}$, i.e., functions odd in $v$. In addition, it can be shown similarly to Theorem~\ref{transsquarespectrum} that the spectrum of the restricted squared transport operator is given by
	\begin{align}
		\sigma_{ess}(-\D^2\big|_{\Ltwo^{odd}}) = \overline{ \left(\frac{2\pi\N}{T(\mathring\Omega_0^{EL})}\right)^2 },
	\end{align}
	since $\sin(2\pi k\cdot)\in H^2_\theta\cap L^{2,odd}(]0,1[)$ for $k\in\N$, but non-zero constant functions---which correspond to the eigenvalue $0$---are not in $L^{2,odd}(]0,1[)$.
\end{proof}

The sets \eqref{eq:essspecA}, \eqref{eq:essspecAodd}
may look qualitatively different for different steady state models
depending on the behavior of the period function $T$:

\begin{remark}\label{essspecform}
If $\sup_{\mathring\Omega_0^{EL}}(T) \geq 2 \inf_{\mathring\Omega_0^{EL}}(T)$,
\begin{align*}
\sigma_{ess}(\A) = \{0\}\cup \left[ \frac{4\pi^2}{\sup^2_{\mathring\Omega_0^{EL}}(T)},\infty \right[.
\end{align*}
Otherwise, i.e., $\sup_{\mathring\Omega_0^{EL}}(T) < 2 \inf_{\mathring\Omega_0^{EL}}(T)$, there may appear further gaps in the essential spectrum. For example, if $\frac32\leq\frac{\sup_{\mathring\Omega_0^{EL}}(T)}{\inf_{\mathring\Omega_0^{EL}}(T)}<2$,
\begin{align*}
\sigma_{ess}(\A) = \{0\}\cup \left[ \frac{4\pi^2}{\sup^2_{\mathring\Omega_0^{EL}}(T)},\frac{4\pi^2}{\inf^2_{\mathring\Omega_0^{EL}}(T)} \right] ~\dot\cup~ \left[ \frac{16\pi^2}{\sup^2_{\mathring\Omega_0^{EL}}(T)},\infty \right[.
\end{align*}
In general, the number of gaps in the essential spectrum (including the one at zero) is given by
\begin{align*}
1+\sup \left\{ k\in\N_0 \mid (k+1){\inf_{\mathring\Omega_0^{EL}}(T)} > k~{\sup_{\mathring\Omega_0^{EL}}(T)} \right\}.
\end{align*}
Note that we have infinitely many gaps iff $T$ is constant on $\mathring\Omega_0^{EL}$. 
\end{remark}
In any case, the essential spectrum of $\A$ possesses a gap at the origin since $T$ is bounded from above, see Section~\ref{ssc:Tupper} in the appendix. Following Mathur \cite{Ma}, 
we call this the principal gap of the essential spectrum; although in \cite{Ma} this term is used for the gap of the spectrum of $-i\,\D$ around the origin.
For the existence of an eigenvalue $\lambda\in\R$ of $\A$ in the principal gap, i.e., 
\begin{align*}
	\lambda\in\mathcal G\coloneqq \left]0, \min \left( \sigma_{ess}(\A)\setminus \{0\} \right)\right[=\left]0, \frac{4\pi^2}{\sup^2_{\mathring\Omega_0^{EL}}(T)}\right[,
\end{align*}
it is now sufficient to show that there exists some $g\in \mathrm D(\D^2)$ with $g\perp\ker(\D)$ such that
\begin{align} \label{eq:goal}
\frac{\langle \A g,g\rangle_{\Ltwo}}{\| g\|_{\Ltwo}^2} < \frac{4\pi^2}{\sup^2_{\mathring\Omega_0^{EL}}(T)}.
\end{align}
If such an eigenvalue exists and the corresponding eigenfunction $g$
can be chosen odd in $v$, it would lead to an oscillating mode with period $P$ satisfying
\begin{align*}
P > \sup_{\mathring\Omega_0^{EL}}(T),
\end{align*}
see \eqref{eq:period}. The latter inequality means that the period of the oscillation
of the whole solution is greater then the radial period of every single particle motion
and also their supremum.
We refer to Section~\ref{sc:mathur} for the derivation of a criterion for the existence of an eigenvalue in the principal gap.

\subsection{The essential spectrum of the planar Antonov operator}\label{ssc:essplane}

We now show the analogues of the above results for the planar linearized operator $\pA$ introduced in Definition~\ref{antonovdef1D}. 
Since most of the proofs are similar to the radial setting, see Subsection~\ref{ssc:essradial}, we shall skip some of the details. We again start by analyzing the squared transport operator $\pD^2$. We do this by representing all the functions involved in action-angle variables.

For fixed $(x,v_1)\in\R^2$ let $\R\ni t\mapsto(X,V_1)(t,x,v_1)$
be the unique global solution to the characteristic system 
\begin{align}\label{eq:charsystplanar}
\dot X = V_1,\quad \dot V_1= -\pU_0'(X)
\end{align}  
satisfying the initial condition $(X,V_1)(0,x,v_1)=(x,v_1)$.
As derived in Section \ref{ssc:ststp}, $(X,V_1)(\cdot,x,v_1)$ is periodic with period
$\pT(\pE)$, where $\pE\equiv\pE((X,V_1)(t,x,v_1))$ for $t\in\R$.
We now use the variable $\theta\in[0,1]$ given by
\begin{align}
(x,v_1)=(X,V_1)( \theta \pT(\pE), x_-(\pE),0 )
\end{align}  
together with $\pE\in[\pUmin,\pE_0[$ and $\pv$ on the steady state support. 
The major benefit of the plane symmetric case compared to the spherically symmetric setting
is that the $(x,v_1)$-motion can be described by one angle $\theta$ and only one conserved
quantity $\pE$; it is independent of the other integrals $\pv$.
This corresponds to the fact that planar symmetry puts the problem into
a lower dimensional setting.
For functions $g\colon\pOmega_0\to\R$ we write
\begin{align*}
g (\theta,\pE,\pv) \coloneqq g((X,V_1)( \theta \pT(\pE), x_-(\pE),0 ),\pv)
\end{align*}
for $(\theta,\pE,\pv)\in\pOmega_0^\theta\coloneqq [0,1]\times\pOmega_0^{\pE\pv}=[0,1]\times[\pUmin,\pE_0[\times\{\pbeta\neq0\}$ by slight abuse of notation.
As in the spherically symmetric case
the transformation $(x,v)\mapsto(\theta,\pE,\pv)$ is not measure preserving, since
\begin{align}
\diff x\diff v= \pT(\pE) \diff\theta\diff\pE\diff\pv.
\end{align}
Regarding the latter statement, we observe that the mapping
$[0,\frac12]\ni\theta\mapsto X( \theta \pT(\pE), x_-(\pE),0)\in[x_-(\pE),x_+(\pE)]$
is bijective for $\pE>\pUmin$ and its inverse is given by
\begin{align}\label{E:THETAPLANARDEF}
\theta(x,\pE) = \frac{1}{\pT(\pE)} \int_{x_-(\pE)}^x \frac{1}{\sqrt{2\pE-2\pU_0(y)}}\diff y.
\end{align}
Related planar action-angle coordinates are also used for the study of BGK waves for the one-dimensional Vlasov-Poisson system in the plasma case~\cite{GuLi2017}.

The transport operator can now be written as a $\theta$-derivative:
\begin{lemma}
	For $g\in C^1(\pOmega_0)$ and $(\theta,\pE,\pv)\in\pOmega_0^\theta$,
	\begin{align*}
	(\pD g)(\theta,\pE,\pv) = \frac1{\pT(\pE)} (\partial_\theta g)(\theta,\pE,\pv).
	\end{align*}
	Similarly, for $g\in C^2(\pOmega_0)$,
	\begin{align*}
	(\pD^2g)(\theta,\pE,\pv) = \frac1{\pT(\pE)^2} (\partial_\theta^2g)(\theta,\pE,\pv).
	\end{align*}
\end{lemma}
\begin{proof}
  The assertions follow by the chain rule.
\end{proof}

Next we generalize the above lemma onto the weak extension of $\pD$---introduced in Definition~\ref{deftransportplanar}---and give an alternate representation of $\mathrm D(\pD)$ analogous to Lemma~\ref{trans1d}:

\begin{lemma}\label{trans1dplane}
  It holds that
  \begin{align*}
    \mathrm D(\pD) = \{ g\in \pH \mid&\; \text{for a.e. }(\pE,\pv)\in\pOmega_0^{\pE\pv},
    \ g(\cdot,\pE,\pv)\in H^1_\theta, \\
    &\;\text{and } \int_{\pOmega_0^{\pE\pv}} \frac{\pT(\pE)^{-1}}{\vert \pvarphi'(\pE,\pv)\vert}
    \int_0^1 \vert \partial_\theta g(\theta,\pE,\pv)\vert^2\diff\theta\diff(\pE,\pv)<\infty  \},
  \end{align*}
  where $H^1_\theta$ is given by \eqref{eq:defh1tau}. If $g\in\mathrm D(\pD)$,
  \begin{align*}
    \left(\pD g\right) (\theta,\pE,\pv) = \frac1{\pT(\pE)} (\partial_\theta g) (\theta, \pE,\pv)
  \end{align*}
  for a.e.~$(\theta, \pE,\pv)\in\pOmega_0^\theta$.
\end{lemma}
The proof is completely analogous to the one of Lemma~\ref{trans1d}. The
$v_1$-parity and $x$-parity can also be translated into $(\theta,\pE,\pv)$-coordinates:
\begin{remark}
	Consider $g\colon\pOmega_0\to\R$. Then
	\begin{enumerate}[label=(\alph*)]
		\item $g$ is odd in $v_1$ if and only if $g(\theta,\pE,\pv)=-g(1-\theta,\pE,\pv)$ for $(\theta,\pE,\pv)\in\pOmega_0^\theta$.
		\item $g$ is odd in $x$ if and only if $g(\theta,\pE,\pv)=-g(\frac12-\theta,\pE,\pv)$ as well as $g(1-\theta,\pE,\pv)=-g(\theta+\frac12,\pE,\pv)$ for $0\leq\theta\leq\frac12$ and $(\pE,\pv)\in\pOmega_0^{\pE\pv}$.
	\end{enumerate}
	Analogous equivalences apply for $g$ being odd almost everywhere.
\end{remark}
As to the squared transport operator in the plane symmetric setting:
\begin{cor}\label{trans1dsquareplane}
  With $H^2_\theta$ as defined in \eqref{eq:defh2tau},
  \begin{align*}
    \mathrm D(\pD^2) = \{ g\in \pH \mid&\; \text{for a.e. }(\pE,\pv)\in\pOmega_0^{\pE\pv},\
    g(\cdot,\pE,\pv)\in H^2_\theta, \\
    &\;\text{and }
    \sum_{j=1}^2\int_{\pOmega_0^{\pE\pv}} \frac{\pT(\pE)^{1-2j}}{\vert \pvarphi'(\pE,\pv)\vert}
    \int_0^1 \vert \partial_\theta^j g(\theta,\pE,\pv)\vert^2\diff\theta\diff(\pE,\pv)<\infty\}.
  \end{align*}
\end{cor}
\begin{proof}
  Both implications follow by applying Lemma~\ref{trans1dplane} twice.
\end{proof}

The above results allow us to conclude the following fundamental properties of
the transport operator and its square which have partly been stated in
Proposition~\ref{transport1D} before.
\begin{prop}\label{proptransplane}
  \begin{enumerate}[label=(\alph*)]
  \item $\pD\colon \mathrm D(\pD)\to\pH$ is skew-adjoint and
    $\pD^2\colon\mathrm D(\pD^2)\to\pH$ is self-adjoint.
    The restrictions to $\p\H$ are skew-adjoint respectively self-adjoint as well.
  \item The kernel of $\pD$ consists of all functions depending solely on $\pE$ and $\pv$ a.e..
  \item For every $h\in\pH$ with $h\perp\ker(\pD)$ there exists
    $g\in\mathrm D(\pD)$ such that $\pD g=h$. In particular,
    \begin{align}
      \ker(\pD)^\perp = \mathrm{im}(\pD).
    \end{align}
  \end{enumerate}
\end{prop}
\begin{proof}
  Interpreting $\pD$ as a $\theta$-derivative, we integrate by parts in the
  $\theta$-integral to obtain the skew-symmetry of $\pD$ and the symmetry of $\pD^2$.
  From there we easily obtain the skew-adjointness and self-adjointness respectively.
	
  Part (b) follows since the kernel of $\partial_\theta\colon H^1_\theta\to L^2(]0,1[)$
  consists of all functions which are constant a.e.. 
	
  As to the last part, we define $g\colon\pOmega_0\to\R$ via
  \begin{align*}
    g (\theta,\pE,\pv)\coloneqq \pT(\pE) \int_0^\theta h(s,\pE,\pv)\diff s
    \quad\text{for a.e. } (\theta,\pE,\pv)\in\pOmega_0^\theta.
  \end{align*}
  It is then straight-forward to verify that $g\in\pH$ with
  \begin{align*}
    \|g\|_{\pH}^2\leq\pT(\pE_0)^2\,\|h\|_{\pH}^2,
  \end{align*}
  and $g(\cdot,\pE,\pv)\in H^1_\theta$ for a.e.~$(\pE,\pv)\in\pOmega_0^{\pE\pv}$,
  in particular $g(0,\pE,\pv)=0=g(1,\pE,\pv)$ since $h\perp\ker(\D)$.
  In addition, $\pD g=h$ as required. For more details we refer to the proof of
  Lemma \ref{transinversetauel}
  in the spherically symmetric setting.
\end{proof}

Finally we combine Corollary \ref{trans1dsquareplane} and the spectral properties of $-\partial_\theta^2\colon H^2_\theta\to L^2(]0,1[)$ from Lemma \ref{laplace1d} to explicitly determine the spectrum of $-\pD^2$:

\begin{theorem}\label{transsquarespectrumplane}
	The spectrum of the self-adjoint operator $-\pD^2\colon \mathrm D(\pD^2)\to \pH$ is
	\begin{align*}
	\sigma(-\pD^2) = \left(\frac{2\pi\N_0}{\pT([\pUmin,\pE_0])}\right)^2  \coloneqq \left\{ \left( \frac{2\pi k}{\pT(\pE)} \right)^2 ~\Big|~ k\in\N_0,~\pUmin\leq\pE\leq\pE_0 \right\} .
	\end{align*}
	Furthermore, the spectrum is purely essential, i.e.,
	\begin{align*}
	\sigma_{ess}(-\pD^2) = \sigma(-\pD^2).
	\end{align*}
\end{theorem}
\begin{proof}
  Since the proof is very similar to the one of Theorem \ref{transsquarespectrum},
  we only outline it for the plane symmetric setting.
  For fixed $\pE^*\in]\pUmin,\pE_0[$ and $k\in\N$
    we show $\left(2\pi k\,\pT(\pE^*)^{-1}\right)^2\in\sigma_{ess}(-\pD^2)$
    by constructing approximate eigenfunctions to the eigendistribution
	\begin{align}\label{eq:eigendistributionplane}
	(\theta,\pE,\pv)\mapsto\delta_{\pE^*}(\pE)\,\sin(2\pi k\theta),
	\end{align}
	i.e., a \enquote{Weyl-sequence} $(g_j)_{j\in\N}\subset\mathrm D(\pD^2)$ satisfying
	\begin{enumerate}[label=(\roman*)]
		\item $\displaystyle \|g_j\|_{\pH} = 1$ for every $j\in\N$,
		\item $\displaystyle \left\| -\pD^2g_j - \left( 2\pi k\;\pT(\pE^*)^{-1} \right)^2 g_j \right\|_{\pH}\to0$ as $j\to\infty$,
		\item $\displaystyle g_j\weakto0$ in $\displaystyle \pH$ as $j\to\infty$.
	\end{enumerate}
	To this end choose $\chi_j\colon\R\to\R$ such that $\supp(\chi_j)\subset]\pE^*-\frac1j,\pE^*+\frac1j[\cap]\pUmin,\pE_0[$ and
	\begin{align*}
	\int_{\pUmin}^{\pE_0}\chi_j^2(\pE)\diff\pE=2
	\end{align*}
	for $j\in\N$. Then
	\begin{align*}
	g_j (\theta,\pE,\pv)\coloneqq \sqrt{\frac{\vert\pvarphi'(\pE,\pv)\vert}{\pT(\pE)}} \;\chi_j(\pE)\;\sqrt{\pbeta(\pv)}\;\sin(2\pi k\theta),\qquad (\theta,\pE,\pv)\in\pOmega_0^\theta,\;j\in\N,
	\end{align*}
	defines a Weyl-sequence with the three properties claimed above. In particular, $g_j\in\mathrm D(\pD^2)$ for $j\in\N$ by Corollary \ref{trans1dsquareplane} since $\pT$ is bounded away from zero, see~Proposition \ref{Tboundsplane}. Furthermore,
	\begin{align*}
	\big\| -&\pD^2g_j - \left( 2\pi k\,\pT(\pE^*)^{-1} \right)^2 g_j \big\|_{\pH} = 2^3 (\pi k)^4 \int_{\pUmin}^{\pE_0} \chi_j^2(\pE) \left| \frac1{\pT(\pE)^2} - \frac1{\pT(\pE^*)^2} \right|^2\diff\pE\to0
	\end{align*}
	as $j\to\infty$ by the continuity of $\pT>0$ on $]\pUmin,\pE_0[$, see Lemma \ref{Tderplane}. In addition, the essential spectrum contains $0$ (since the kernel of $\pD$ is an eigenspace of infinite multiplicity) and is closed, which means that we have shown 
	\begin{align*}
	\left(\frac{2\pi\N_0}{\pT([\pUmin,\pE_0])}\right)^2\subset \sigma_{ess}(-\pD^2)\subset\sigma(-\pD^2).
	\end{align*} 
	Conversely, for fixed $\lambda\in\R\setminus \left(\frac{2\pi\N_0}{\pT([\pUmin,\pE_0])}\right)^2$ there exists $c>0$ such that 
	\begin{align*}
	\mathrm{dist} \left( \lambda\; \pT(\pE)^2, (2\pi\N_0)^2 \right)\geq c \text{ for every } \pE\in[\pUmin,\pE_0]
	\end{align*}
	due to the boundedness of $\pT$ away from zero, see Proposition \ref{Tboundsplane}. Since $(2\pi\N_0)^2$ equals the spectrum of $-\partial_\theta^2\colon H^2_\theta\to L^2(]0,1[)$,
	\begin{align*}
	\|z\|_2\leq\frac1c \left\| \left( -\partial_\theta^2 - \lambda\;\pT(\pE)^2 \right)z \right\|_2 \quad\text{for } z\in H^2_\theta ,~\pE\in[\pUmin,\pE_0],
	\end{align*}
	see \cite[Theorem 5.8]{HiSi}. However, this rules out the existence of a sequence $(g_j)_{j\in\N}\subset \mathrm D(\pD^2)$ such that $\|g_j\|_{\pH} = 1$ for $j\in\N$ and $\left\| -\pD^2g_j - \lambda g_j \right\|_{\pH}\to0$ as $j\to\infty$, since in that case,
	\begin{align*}
	\left\| -\pD^2g_j - \lambda g_j \right\|_{\pH}^2 &= \int_{\pOmega_0^{\pE\pv}} \frac{\pT(\pE)^{-3}}{\vert\pvarphi'(\pE,\pv)\vert} \int_0^1 \vert -\partial_\theta^2g_j(\theta,\pE,\pv)-\lambda\,\pT(\pE)^2g_j(\theta,\pE,\pv)\vert^2\diff\theta\diff(\pE,\pv)\\
	&\geq \frac{c^2}{\pT(\pE)^4} \int_{\pOmega_0^{\pE\pv}} \frac{\pT(\pE)}{\vert\pvarphi'(\pE,\pv)\vert} \int_0^1 \vert g_j(\theta,\pE,\pv)\vert^2\diff\theta\diff(\pE,\pv)= \frac{c^2}{\pT(\pE)^4}>0
	\end{align*}
	for $j\in\N$. Therefore, by a weak version of Weyl's criterion \cite[Theorem 5.10]{HiSi} we conclude $\lambda\notin\sigma(-\pD^2)$.
\end{proof}

A simple consequence of the spectral properties of $-\pD^2$ is the following estimate:

\begin{cor}[A Poincar\'e inequality]\label{estgbyDg}
  For every $g\in\mathrm D(\pD)$ with $g\perp\ker(\pD)$,
  \begin{align*}
    \|\pD g\|_{\pH}^2\geq \frac{4\pi^2}{\pT(\pE_0)^2}\;\|g\|_{\pH}^2.
  \end{align*}
\end{cor}
While the above estimate for $g\in\mathrm D(\pD^2)$ follows immediately from Theorem~\ref{transsquarespectrumplane} and the skew-symmetry of $\pD$, we have to mollify an element of $\mathrm D(\pD)$ in some way to get the result there as well. This approximation will be useful later on as well.

\begin{remark}\label{approxplane}
	Let $g\in\mathrm D(\pD)$. To approximate $g$ by smooth functions while keeping $\pD g$ under control, we expand $g$ in its $\theta$-Fourier series, i.e.,
	\begin{align}\label{eq:thetafourierplane}
	g(\theta,\pE,\pv)= \sum_{k=0}^\infty a_k(\pE,\pv)\cos(2k\pi\theta) + \sum_{k=1}^\infty b_k(\pE,\pv)\sin(2\pi k\theta),
	\end{align}
	where the coefficients $a_k$, $b_k$ are given by
	\begin{align*}
	a_0(\pE,\pv) &\coloneqq \int_0^1 g (\theta,\pE,\pv)\diff\theta,\\
	a_k(\pE,\pv) &\coloneqq 2\int_0^1 g (\theta,\pE,\pv)\cos(2\pi k\theta)\diff\theta, \ \ 
	b_k(\pE,\pv) \coloneqq 2\int_0^1 g (\theta,\pE,\pv)\sin(2\pi k\theta)\diff\theta,
	\end{align*}
	for $k\in\N$ and $(\pE,\pv)\in\pOmega_0^{\pE\pv}$. Furthermore, by Lemma~\ref{trans1dplane},
	\begin{align}\label{eq:thetafourierderplane}
	\pD g(\theta,\pE,\pv)= -\frac{2\pi}{\pT(\pE)}\sum_{k=1}^\infty k\,a_k(\pE,\pv)\sin(2k\pi\theta) +\frac{2\pi}{\pT(\pE)}\sum_{k=1}^\infty k\,b_k(\pE,\pv)\cos(2\pi k\theta).
	\end{align}
	Eqns.~\eqref{eq:thetafourierplane} and \eqref{eq:thetafourierderplane}
        both hold as limits in $\pH$, which can be seen by using the properties of the scalar $\theta$-Fourier series and Lebesgue's dominated convergence theorem. Thus we may assume $g$ to be of the form 
	\begin{align}\label{eq:thetafourierplaneapprox}
	g(\theta,\pE,\pv)= \sum_{k=0}^K a_k(\pE,\pv)\cos(2k\pi\theta) + \sum_{k=1}^K b_k(\pE,\pv)\sin(2\pi k\theta)
	\end{align}
	for some $K\in\N$. Observe that $g\perp\ker(\pD)$ is equivalent to $a_0=0$ on $\pOmega_0^{\pE\pv}$, i.e., this property carries over to the approximation. Similarly, $g$ being odd in $v_1$ is equivalent to $a_k=0$ for $k\in\N_0$ and $g$ being odd in $x$ is equivalent to $a_{2k}=0=b_{2k+1}$ for $k\in\N_0$, i.e., these properties carry over too.
	
	To achieve $g\in\mathrm D(\pD^2)$, we have to mollify $g$ in the $(\pE,\pv)$-direction as well. More precisely, we replace $a_k,\,b_k$ by approximations $\tilde a_k,\,\tilde b_k\in C^\infty_c(\inter(\pOmega_0^{\pE\pv}))$ in \eqref{eq:thetafourierplaneapprox} such that the resulting function $\tilde g$ is in $C^2_c(\pOmega_0)$; note that we only have to treat a finite number of coefficients.  Then $\tilde g$ and $\pD\tilde g$ approximate $g$ and $\pD g$ respectively. In particular, by requiring $\tilde a_k=0$ or $\tilde b_k=0$ if $a_k=0$ or $b_k=0$ respectively, the possible
        parity properties of $g$ carry over to $\tilde g$. 
\end{remark} 

We use Theorem~\ref{transsquarespectrumplane} to determine the spectrum of
$\pA=-(\pD^2+\pB)$ by showing that adding $\pB$ does not change the essential spectrum;
see \eqref{E:BARBDEF} for the definition of $\pB$.

\begin{theorem}\label{T:ESSENTIALSPECTRUMABAR}
	$(-\pB)$ is relatively $(-\pD^2)$-compact. Therefore, by the Weyl theorem,
	\begin{align*}
	\sigma_{ess}(\pA) = \sigma_{ess}(-\pD^2).
	\end{align*}
	Together with Theorem~\ref{transsquarespectrumplane}, this implies
	\begin{align}\label{eq:essspecantonovplane}
	\sigma_{ess}(\pA) =  \left( \frac{2\pi\N_0}{\pT([\pUmin,\pE_0])} \right)^2 = 4\pi^2 \left( \frac{\N_0}{\pT([\pUmin,\pE_0])} \right)^2,
	\end{align}
	where $\pA$ denotes the operator on the whole space $\mathrm D(\pD^2)$ with no symmetry restrictions. In addition, 
	\begin{align}
	\sigma_{ess}(\pA\big|_{\p\H}) &=  4\pi^2 \left( \frac{2\N}{\pT([\pUmin,\pE_0])} \right)^2;\label{eq:essspecantonovplaneoddv1x}
	\end{align}
	recall that 
	$\p\H$ denotes the functions in $\pH$ which are odd in $v_1$ and $x$. 
	Here, $\min(\N)=1$, $\N_0=\N\cup\{0\}$.
\end{theorem}
\begin{proof}
	Since $\pB$ is continuous on $\pH$ and $\pD^2$ is self-adjoint with non-empty resolvent set, the claimed relative compactness is equivalent to the operator
	\begin{align}
		-\pB\colon\left( \mathrm D(\pD^2),\;\|\pD^2\cdot\|_{\pH} + \|\cdot\|_{\pH} \right)\to\pH
	\end{align}
	being compact; we refer to the proof of Theorem~\ref{essspecA} for more details. Let $(g_k)_{k\in\N}\subset\mathrm D(\pD^2)$ be such that $(g_k)_{k\in\N}$ and $(\pD^2g_k)_{k\in\N}$ are bounded in $\pH$. Then $(\pD g_k)_{k\in\N}\subset\pH$ is also bounded, see Corollary \ref{estgbyDg}. Thus, $(\pU_{\pD g_k}')_{k\in\N}$ and $(\pU_{\pD g_k}'')_{k\in\N}$ are bounded in $L^2(\R)$ due to \eqref{eq:poth2plane}, and $\supp(\pU_{\pD g_k}'),\,\supp(\pU_{\pD g_k}'')\subset[-\pR_0,\pR_0]$ for $k\in\N$. Therefore, the compact embedding $H^1([-\pR_0,\pR_0])\Subset L^2([-\pR_0,\pR_0])$ yields that $(\pU_{\pD g_k}')_{k\in\N}$ converges strongly in $L^2(\R)$, at least after extracting a subsequence. Representing $\pB$ as in \eqref{eq:bpot1D} then gives
	\begin{align*}
		\|\pB g_k-\pB g_l\|_{\pH}^2 &= \int_{\pOmega_0} \vert\pvarphi(\pE,\pv)\vert\,v_1^2\;\vert\pU_{\pD g_k}'(x)-\pU_{\pD g_l}'(x)\vert^2\diff(x,v)\\&= \int_{\R} \prho_0(x)\,\vert\pU_{\pD g_k}'(x)-\pU_{\pD g_l}'(x)\vert^2\diff x\to0\quad\text{as }k,l\to\infty,
	\end{align*}
	where we have rewritten the $v$-integral as in \eqref{eq:v1intrho}. We therefore conclude that $(\pB g_k)_{k\in\N}$ is strongly convergent in $\pH$.
	
	Translating oddness in $v_1$ or $x$ into the $\theta$-coordinate allows us to conclude the statements for the restriction of $\pA$ onto $\p\H$.
\end{proof}

As in the spherically symmetric case, these sets contain a finite number of gaps iff $\pT$
is non-constant. The precise form of the sets depends on the steady state and the behavior
of the period function $\pT$, see Remark \ref{essspecform}. In any case, the boundedness of
$\pT$ from above (Proposition \ref{Tboundsplane}) again yields the existence of a gap in the
essential spectrum of $\pA$ at the origin which we refer to as the {\em principal gap}.

\section{Kurth solutions}\label{sc:kurth}

In this section we present two families of semi-explicit, time dependent
solutions which are exactly time-periodic and arise by a suitable
perturbation of a corresponding steady state. The first family solves the spherically
symmetric Vlasov-Poisson system and was introduced by R.~Kurth \cite{Ku78}.
The second family is the analogue for the plane symmetric case.
 We are in particular interested
in the relation of the period of the time-periodic solutions close to the
corresponding steady state and the orbital period(s) of the particles in the steady
state configuration itself. The latter will give us a first intuition of where the
eigenvalue of the linearized operator corresponding to the aforementioned oscillations
can be positioned relative to the essential spectrum.


\subsection{The spherically symmetric Kurth family}\label{ssc:kurthradial}

Following Kurth \cite{Ku78} we define,
in the spherically symmetric situation,
\[
f_0(x,v)\coloneqq\frac{3}{4\pi^3}
\left\{
\begin{array}{cl} 
\left(1-|x|^2 - |v|^2 +|x\times v|^2\right)^{-1/2}&,\
\mbox{where}\ (\ldots)>0 \\
&\ \ \mbox{and}\ |x\times v|<1,\\
0&,\ \mbox{else}.
\end{array} \right.
\]
The induced spatial density is
\[
\rho_0 = \frac{3}{4\pi} \id_{B_1(0)},
\]
which in turn induces the spherically symmetric potential
\[
U_0(x) =
\left\{
\begin{array}{cl}
\frac{1}{2} |x|^2 - \frac{3}{2}&,\ |x|\leq 1,\\
-\frac{1}{|x|} &,\ |x|>1,
\end{array} \right.
\]
and the gravitational field
\[
\nabla U_0(x) =
\left\{
\begin{array}{cl}
x&,\ |x|\leq 1,\\
\frac{x}{|x|^3} &,\ |x|>1.
\end{array} \right.
\]
On the support of $f_0$ the particle energy equals
\[
E = \frac{1}{2} |v|^2 + U_0(x) = \frac{1}{2} \left(|v|^2 + |x|^2\right) -\frac{3}{2}. 
\]
Thus, $f_0$ depends on the invariants $E$ and $L$ only and is
therefore a stationary solution of the radial Vlasov-Poisson system
corresponding to the ansatz function
\[
\varphi(E,L)=\frac3{4\pi^3}
\left\{
\begin{array}{cl} 
\left(-2-2E+L\right)^{-1/2}&,\
\mbox{where}\ (\ldots)>0 \ \mbox{and}\ L<1,\\
0&,\ \mbox{else};
\end{array} \right.
\]
we emphasize the fact that $f_0$ is singular at the boundary of its support.
Setting
\begin{align}\label{eq:kurthpert3d}
f(t,x,v) \coloneqq f_0\left(\frac{x}{R(t)},R(t)v - \dot R(t) x\right)
\end{align}
embeds this steady state into a family of time-periodic solutions,
with induced spatial mass density 
\[
\rho(t) = \frac{3}{4\pi} \frac{1}{R^3(t)} \id_{B_{R(t)}}, 
\]
provided the function $R(t)$ solves the differential equation
\begin{equation}\label{eq:R_eq_sph}
\ddot R - \frac{1}{R^3} + \frac{1}{R^2} = 0.
\end{equation}
We supplement this equation with initial data 
\begin{equation}\label{eq:R_eqdata_pl}
R(0)=1,\ \dot R(0)=\gamma.
\end{equation}
For $\gamma =0$ we find that $R(t)=1$ is constant and recover the steady state
$f_0$. For $0<|\gamma|<1$, the solution is time-periodic, which can be seen as follows.
The equation has the conserved quantity
\begin{align*}
\frac{1}{2} \dot R^2 + \frac{1}{2 R^2} - \frac{1}{R} = const
= \frac{1}{2}\gamma^2 - \frac{1}{2} \eqqcolon E_\gamma.
\end{align*}
The behavior of the corresponding potential $\frac{1}{2 R^2} - \frac{1}{R}$
implies that energy levels
$- \frac{1}{2}< E_\gamma < 0$, i.e., $0<|\gamma|<1$,
correspond to closed orbits and non-trivial, time-periodic solutions of~\eqref{eq:R_eq_sph}.
Their periods are given by
\begin{align} \label{eq:talpha_sph}
P(\gamma) 
&= 2\int_{R_-(\gamma)}^{R_+(\gamma)}
\frac{\diff r}{\sqrt{2 E_\gamma +\frac{2}{r} -\frac{1}{r^2}}}\nonumber\\
&= \frac{2}{\sqrt{-2E_\gamma}} \int_{R_-(\gamma)}^{R_+(\gamma)}
\frac{r\diff r}{\sqrt{(R_+(\gamma)-r) (r-R_-(\gamma))}},
\end{align}
where $0<R_-(\gamma) < R_+(\gamma)$ are the two positive roots
of the polynomial $2 E_\gamma r^2 + 2 r -1$. In the integral
in \eqref{eq:talpha_sph},
$R_-(\gamma) \leq r \leq R_+(\gamma)$, and the remaining integral is equal to $\pi$.
Hence the fact
that these roots converge to $1$ as $\gamma\to 0$ implies that
\[
\lim_{\gamma \to 0} P(\gamma) = 2\pi.
\]

On the other hand, a straight-forward calculation shows that the radial particle period
of all the particles in the steady state equals $\pi$, i.e.,
\begin{align*}
T(E,L) = \pi ,\quad (E,L)\in\mathring\Omega_0^{EL}.
\end{align*}
In view of the spectral considerations of the previous section, this means that the
eigenvalue corresponding to the limiting period $2\pi$ lies in the principal gap of
the essential spectrum of the linearized operator. Note that the Kurth steady state
does not satisfy our general assumptions, but the results from Section \ref{ssc:essradial}
are still expected to hold true in the Kurth setting. However, the discrepancy between the
limiting period and the period of the particle trajectories is only present when
restricting the latter to the radial motion. When the particle trajectories are
considered in $(x,v)$-coordinates, all particles have period $2\pi$, caused
by an azimuthal period of $2\pi$. This illustrates that the restriction to the
spherically symmetric setting may be important when searching for isolated eigenvalues
in the principal gap of the essential spectrum.

\subsection{A planar Kurth-type family}\label{ssc:kurthplane}
Let
\begin{align*}
\pf_0(x,v)\coloneqq \frac1{4\pi^2} \left(1-x^2-v_1^2\right)_+^{-1/2} \,\pbeta(\pv)
\end{align*}
for $(x,v)\in\R\times\R^3$, where $\pbeta$ is as specified in Section \ref{ssc:ststp},
in particular $\int_{\R^2}\pbeta=1$; also recall~\eqref{E:+conv}.
Then $\pf_0(x,v)=0$ if $|x|\geq 1$, and for $|x|<1$ the induced spatial
density becomes
\begin{align*}
\prho_0(x)
&=
\int_{\R^3} \pf_0(x,v)\, \diff v
= \frac1{4\pi^2} \int_{v_1^2 < 1-x^2}\left(1-x^2 - v_1^2\right)^{-1/2} \diff v_1\\
&= \frac1{4\pi^2} \int_{x^2}^1\frac{\diff\eta}{\sqrt{1-\eta}\sqrt{\eta-x^2}}= \frac1{4\pi}
\end{align*}
so that
\[
\prho_0 = \frac{1}{4\pi} \id_{]-1,1[}.
\]
This density induces the potential
\begin{align*}
  \pU_0(x)= 2\pi \int_{\R} \vert x-y\vert\, \prho_0(y)\diff y =\left\{ \begin{array}{cl}
    x&,\ x\geq 1,\\
    \frac12\left(1+x^2\right)&,\ -1<x<1,\\
    -x&,\ x\leq  -1,
  \end{array}\right.
\end{align*}
and 
\begin{align*}
  \pU_0'(x)=\left\{ \begin{array}{cl}
    1&,\ x\geq 1,\\
    x&,\ -1<x<1,\\
    -1&,\ x\leq  -1.
  \end{array}\right.
\end{align*}
On the support of $\pf_0$ the particle energy $\pE$ takes the form
\[
\pE(x,v_1) = \frac12 v_1^2 + \pU_0(x) = \frac12\left(1+x^2 + v_1^2\right),
\]
i.e., $\pf_0$ only depends on the conserved quantities $\pE$ and $\pv$ via the ansatz function
\begin{align*}
\pvarphi(\pE,\pv) = \frac1{4\sqrt2\pi^2} \left(1-\pE\right)^{-1/2}_+ \,\pbeta(\pv).
\end{align*}
With the exception of the factor $(4\sqrt2\pi^2)^{-1}$, which we inserted to make $\pf_0$
look similar to the radial case, $\pvarphi$ is exactly of the polytropic
form \eqref{eq:poly1d} with index $k=-\frac12$. In particular, $\pf_0$
indeed induces a stationary solution of the planar Vlasov-Poisson system~\eqref{pl_vlasov}--\eqref{pl_rho}.

As in the spherically symmetric setting, we embed this steady state into a family of time-periodic
solutions
\[
\pf(t,x,v) \coloneqq \pf_0\left(\frac{x}{R(t)}, R(t) v_1 - \dot R(t) x,\pv\right),
\]
where the function $R(t)$ still needs to be determined. This phase space density
induces the spatial density
\[
\prho(t,x) = \frac{1}{R(t)}\, \prho_0\left(\frac{x}{R(t)}\right)
\]
and potential
\[
\pU(t,x) = R(t)\, \pU_0 \left(\frac{x}{R(t)}\right),
\]
so that
\[
\pU'(t,x) = \pU'_0 \left(\frac{x}{R(t)}\right).
\]
Substituting all this into the Vlasov equation and observing that
$\pf_0$ satisfies the stationary Vlasov equation with potential $\pU_0$
we see that $\pf$ satisfies the Vlasov equation with its induced potential $\pU$,
iff
\begin{equation}\label{eq:R_eq_pl}
\ddot R - \frac{1}{R^3} +1 =0.
\end{equation}
We again supplement this with the initial data \eqref{eq:R_eqdata_pl};
the choice $\gamma=0$ recovers the steady state $\pf_0$. For $\gamma\neq 0$
the solution is time-periodic, which can be seen similarly to the radial setting.
Along solutions of \eqref{eq:R_eq_pl} the following quantity is conserved:
\begin{equation}\label{eq:R_energy_pl}
\frac{1}{2} \dot R^2 + \frac{1}{2 R^2} + R = const
= \frac{1}{2}\gamma^2 + \frac{3}{2} =: E_\gamma.
\end{equation}
The corresponding potential $\frac{1}{2 R^2}+R$ has a unique, strict,
global minimum at $R=1$ and becomes infinite both for $R\to 0$ and
$R\to \infty$. Hence every level set of the energy \eqref{eq:R_energy_pl}
with $E_\gamma >\frac{3}{2}$, i.e., $\gamma \neq 0$, corresponds to a closed
orbit and a non-trivial, time-periodic solution of \eqref{eq:R_eq_pl}
with data \eqref{eq:R_eqdata_pl}.

We now want to determine the period of the above time-periodic solutions
in the limit $\gamma \to 0$ and compare it to the periods of the particle
trajectories in the steady state. As to the latter, they are given
as the solutions to
\[
\ddot x = -x,
\]
i.e.,
\[
\pT(\pE) = 2\pi,\quad \frac12<\pE\leq1.
\]
The period of the time-periodic solution depends on $\gamma$ and is given as
\[
P(\gamma) =
2\int_{R_-(\gamma)}^{R_+(\gamma)} \frac{\diff r}{\sqrt{2 E_\gamma -2 r -1/r^2}}
= 2\int_{R_-(\gamma)}^{R_+(\gamma)} \frac{r\diff r}{\sqrt{2 E_\gamma r^2 -2 r^3 -1}}.
\]
Here $0<R_-(\gamma) < R_+(\gamma)$ are the two positive roots
of the cubic polynomial
\[
p_\gamma (r)=2 E_\gamma r^2 -2 r^3 -1;
\]
notice that $p'_\gamma(r)=0$ for $r=0$ and $r=\frac{2}{3}E_\gamma$ and
$p_\gamma(\frac{2}{3}E_\gamma)>0$ if $E_\gamma>\frac{3}{2}$, i.e., $\gamma\neq 0$.
If we denote the third, negative root of $p_\gamma$ by $r^\ast(\gamma)$,
we can
factor the polynomial $p_\gamma$ and conclude that
\[
r^\ast(\gamma) = -\frac{1}{2 R_+(\gamma) R_-(\gamma)}.
\]
Hence,
\[
P(\gamma) 
= \sqrt{2}\int_{R_-(\gamma)}^{R_+(\gamma)}
\frac{r\diff r}
{\sqrt{(R_+(\gamma)-r) (r-R_-(\gamma)) (r-r^\ast(\gamma))}}.
\]
Now we observe that
\[
\int_{R_-(\gamma)}^{R_+(\gamma)} \frac{\diff r}{\sqrt{(R_+(\gamma)-r)(r-R_-(\gamma))}} = \pi.
\]
Using the fact that
the function $r/\sqrt{r-r^\ast(\gamma)}$ is strictly increasing on $[0,\infty[$, we
obtain the estimate
\[
\sqrt{2} \pi
\frac{R_-(\gamma)}{\sqrt{R_-(\gamma)-r^\ast(\gamma)}}
< P(\gamma) <
\sqrt{2} \pi
\frac{R_+(\gamma)}{\sqrt{R_+(\gamma)-r^\ast(\gamma)}}.
\]
For $\gamma \to 0$, both $R_+(\gamma) \to 1$ and
$R_-(\gamma) \to 1$, and thus
\[
\lim_{\gamma \to 0} P(\gamma) = \frac{2\pi}{\sqrt{3}}.
\]
In contrast to the radial case, the limiting period is strictly smaller than the period of
the particle trajectories. In the light of the spectral analysis this means that the
eigenvalue corresponding to the limiting period $2\pi/\sqrt3$ is not in the principal gap
of the planar linearized operator---again assuming that the results from
Section~\ref{ssc:essplane} also apply to the plane symmetric Kurth steady state.
However, when restricting $\pA$ to the space of functions being odd in $v_1$ and $x$,
the aforementioned eigenvalue lies in the principal gap since the non-zero bottom of
the essential spectrum quadruples, see \eqref{eq:essspecantonovplaneoddv1x}.
This shows that assuming oddness in $v_1$ and $x$ can be beneficial in the search for eigenvalues; note that such symmetries are needed for a function to be plane symmetric in the sense of~\eqref{pl_symm_def} anyway.


\section{The spectral gap}\label{sc:gap}

We  show that the spectra of $\A$ and $\pA$ are contained in $[0,\infty[$
(Corollaries~\ref{C:SPECTRUMNONNEGATIVE} and \ref{C:SPECTRUMNONNEGATIVEplane})
and then characterize the respective nullspaces (Corollary~\ref{antonovev0} and
Theorem \ref{gapplane}). In particular, we show that the whole spectra posses a
gap at the origin; in the plane symmetric case we restrict the linearized operator
$\pA$ appropriately to obtain the latter results.

In both the radial and planar setting, the above results are all based on certain forms
of Antonov's coercivity bound.

\subsection{Spherically symmetric case}

Let $\A$ denote the self-adjoint Antonov operator on $\Ltwo $ with domain $\mathrm D(\D^2)$;
similar statements hold true for its restriction to $\Ltwo^{odd}$.
First, we restate Antonov's coercivity bound:
\begin{prop}\label{antonovcoerc}
  For $g\in C^2_{c,r}(\Omega_0)$ odd in $v$,
  \begin{align}\label{eq:antonovcoerc}
    \langle\A g,g\rangle_{\Ltwo} \geq
    \int_{\Omega_0}\frac1{\vert\varphi'(E,L)\vert}\,\frac{U_0'(r)}r \;\vert g(x,v)\vert^2 \diff(x,v).
  \end{align}
\end{prop}
An estimate of the above type was first shown by V.~A.~Antonov \cite{An1961}.  
For recent proofs we refer the reader to \cite[Lemma 1.1]{GuRe2007} or \cite[(4.6)]{LeMeRa11}.
In the former reference the result is only proven for isotropic steady states, but the same
proof can also be applied to show the coercivity bound for general linearly stable models
and then provides a sharper estimate than the one stated above.
We use \eqref{eq:antonovcoerc} to deduce the following two corollaries:

\begin{cor}\label{C:SPECTRUMNONNEGATIVE}
  The quadratic form associated with $\A$ is non-negative on $\mathrm D(\D)$.
  Thus, $\sigma(\A)\subset[0,\infty[$.
\end{cor}
\begin{proof}
  To apply Proposition \ref{antonovcoerc} we split $g\in\mathrm D(\D)$ into its even and
  odd part in $v$, defined by
  \begin{align}\label{eq:gevenoddv}
    g_{\pm}(x,v)\coloneqq \frac12\left(g(x,v)\pm g(x,-v)\right),\quad (x,v)\in\Omega_0.
  \end{align}
  We then use an approximation argument to extend Antonov's coercivity bound to $\mathrm D(\D)$,
  see for example \cite[Proposition 2]{ReSt20}, and obtain $\langle \A(g_-),g_-\rangle_{\Ltwo}\geq0$.
  On the other hand, $\B(g_+)=0$, and therefore $\langle \A(g_+),g_+\rangle_{\Ltwo}\geq0$.
  Since the odd and even subspaces are orthogonal to each other and $\A(g_\pm)=(\A g)_\pm$
  since $\D^2$ preserves $v$-parity, we conclude that 
  \begin{align*}
    \langle \A g,g \rangle_{\Ltwo} = \langle \A(g_+),g_+\rangle_{\Ltwo} +
    \langle \A(g_-),g_-\rangle_{\Ltwo} \geq 0.
  \end{align*}
  For the equivalence of the non-negativity of the quadratic form of $\A$ and
  $\sigma(\A)\subset[0,\infty[$ we refer to \cite[Proposition 5.12]{HiSi}.
\end{proof}

\begin{cor}\label{antonovev0}
	$\ker(\A)=\ker(\D)$.
\end{cor}
\begin{proof}
  The representation \eqref{eq:bpot} of $\B$ immediately yields $\ker(\D)\subset \ker(\A)$.
	
  Conversely, if $\A g=0$ for some $g\in\mathrm D(\D^2)$,
  we split $g$ into its even and odd part w.r.t.~$v$ as in \eqref{eq:gevenoddv} to obtain
  \begin{align}
    0 &= (\A g)_+ = -\D^2 (g_+),\label{eq:gkerneleven}\\
    0 &= (\A g)_- = -\D^2 (g_-) - \B(g_-).\label{eq:gkernelodd}
  \end{align}
  We then extend Proposition \ref{antonovcoerc} to $\mathrm D(\D^2)$
  by approximation to obtain $g_-=0$. For the approximation process,
  note that the weight $\frac{U_0'(r)}r$ is positive and bounded on the radial support,
  in particular $\lim_{r\to0} \frac{U_0'(r)}r= U_0''(0) = \frac{4\pi}3\rho_0(0)<\infty$
  if the steady state does not have an inner radial vacuum region. 
	
  On the other hand, \eqref{eq:gkerneleven} implies $\|\D (g_+)\|_{\Ltwo}^2=0$
  and therefore $g_+\in\ker(\D)$, i.e., overall we conclude that $g\in\ker(\D)$.
\end{proof}

To obtain an oscillating solution of the linearized Vlasov-Poisson system we have to show that there is a strictly 
positive eigenvalue of $\A$ restricted to $\Ltwo^{odd}$. One approach to this problem is to show that
\begin{align}\label{E:LAMBDAONEDEF}
{\lambda_1} \coloneqq \inf_{\substack{g\in\mathrm D(\D)\setminus\{0\}\\g\perp\ker(\D)}} \frac{\langle \A g,g\rangle_{\Ltwo}}{\|g\|_{\Ltwo}^2}
\end{align}
is positive, i.e., $\lambda_1>0$, and that the infimum is obtained by some odd-in-$v$ function. 

While the existence of eigenvalues will be approached in Section~\ref{sc:mathur} by a different method, we next show $\lambda_1>0$. 
Note that such a bound would follow by Antonov's coercivity bound~\eqref{eq:antonovcoerc} if the weight $\frac{U_0'(r)}r$ were bounded away from zero on the radial support.
However, this is not true in general, since in the case of a polytropic ansatz \eqref{E:POLYTROPE} with $L_0>0$, $U_0'$ vanishes on the whole inner vacuum region and the weight
$\frac{U_0'(r)}r$ can not be bounded away from zero on the support.
We pursue an alternative route and consider the following intermediate variational problem:

\begin{prop} \label{helpervariation}
  Let 
  \begin{align*}
    {\tilde \lambda} \coloneqq
    \inf_{\substack{g\in\mathrm D(\D)\\g\notin\ker(\D)}}
    \frac{\langle \A g,g\rangle_{\Ltwo}}{ \|\D g\|_{\Ltwo}^2}
    =  \inf_{\substack{g\in\mathrm D(\D)\\g\notin\ker(\D)}} 1 -
    \frac{\|\partial_xU_{\D g}\|_2^2}{4\pi\|\D g\|_{\Ltwo}^2}.
  \end{align*}
  Then $0<\tilde\lambda<1$ and the infimum is obtained by a minimizer.
\end{prop}
\begin{proof}
  Let $(g_k)_{k\in\N}\subset \mathrm D(\D)$ be a minimizing sequence,
  i.e., $\|\D g_k\|_{\Ltwo}^2 =1$ for $k\in\N$ and $\langle \A g_k,g_k\rangle_{\Ltwo}\to\tilde\lambda$
  as $k\to\infty$.
  
  \noindent
  \textit{Convergence of the potentials.}
  Due to Section \ref{ssc:potradial} in the appendix, $U_{\D g_k}\in C\cap H^2(\R^3)$
  for $k\in\N$ and $(U_{\D g_k})_{k\in\N}$ is bounded in $H^2(\R^3)$ by \eqref{eq:poth2}.
  Let $\psi\in H^2(\R^3)$ be such that
  \begin{align*}
    U_{\D g_k} \weakto \psi \text{ in } H^2(\R^3),~k\to\infty,
  \end{align*}
  after extracting a subsequence.
  \eqref{eq:uprime} yields $\supp (\partial_xU_{\D g_k})\subset B_{R_0}(0)$
  for every $k\in\N$, and therefore $\supp(\nabla\psi)\subset B_{R_0}(0)$.
  Together with the compact embedding $H^2(B_{R_0}(0))\Subset H^1(B_{R_0}(0))$ we obtain
  \begin{align*}
    \partial_xU_{\D g_k} \to \nabla\psi \text{ in } L^2(\R^3;\R^3),~k\to\infty.
  \end{align*}
  \noindent
  \textit{Convergence of $(\D g_k)_{k\in\N}$.}
  $(\D g_k)_{k\in\N}$ is bounded in $\Ltwo $. Thus, there exists $h\in \Ltwo $ such that
  \begin{align*}
    \D g_k\weakto h \text{ in } \Ltwo ,~k\to\infty,
  \end{align*}
  again after extracting a subsequence.
  
  \noindent
  \textit{The connection between the above limits.}
  Using \eqref{eq:A7prime}, it follows from the above step that
  \begin{align*}
    \int_{\R^3} \D g_k(\cdot,v)\diff v
    = \rho_{\D g_k} \weakto \int_{\R^3} h(\cdot,v)\diff v \text{ in } L^2(\R^3),~k\to\infty.
  \end{align*}
  As before, we extend all functions by $0$. Since $\Delta U_{\D g_k} = 4\pi \rho_{\D g_k}$,
  the uniqueness of weak limits in $L^2(\R^3)$ then yields
  \begin{align*}
    \Delta \psi = \int_{\R^3} h(\cdot,v)\diff v \text{ in } L^2(\R^3).
  \end{align*}
  From this equality it follows that $\psi=U_h$, where
  $U_h\coloneqq-\frac1{\vert\cdot\vert}\ast\rho_h$ is the gravitational potential induced by
  $\rho_h\coloneqq\int_{\R^3}h(\cdot,v)\diff v$. To see this, first note
  $\rho_h\in L^1\cap L^2(\R^3)$ since $\supp(\rho_h)\subset B_{R_0}(0)$.
  Then basic potential theory yields $U_h\in L^6(\R^3)$.
  On the other hand, $\psi\in L^6(\R^3)$ by the embedding $H^1(\R^3)\hookrightarrow L^6(\R^3)$.
  Thus, $u\coloneqq U_h-\psi\in L^6(\R^3)$ and $u$ is harmonic.
  The mean value property of harmonic functions then yields $u=0$, i.e.,
  \begin{align*}
    \psi=U_h.
  \end{align*}		
  \noindent
 \textit{The minimizer.} 
 Since $\D g_k\in \mathrm{im}(\D)\subset\ker(\D)^\perp$ for every $k\in\N$ by the
 skew-adjointness of $\D$, we also have $h\perp\ker(\D)$. By Lemma \ref{transinversetauel}
 there exists $g\in \mathrm D(\D)$ such that
 \begin{align*}
   \D g = h.
 \end{align*}
 Then the above convergences imply
 \begin{align*}
   1-\frac1{4\pi}\|\partial_xU_{\D g}\|_2^2
   = 1-\frac1{4\pi}\|\nabla\psi\|_2^2 = \lim_{k\to\infty}\left(1-\frac1{4\pi}\|\partial_x U_{\D g_k}\|_2^2\right) = \tilde\lambda.
 \end{align*}
 If $\|\partial_xU_{\D g}\|_2^2=0$ we would instantly get $\tilde\lambda=1$.
 However, this contradicts the non-triviality of the steady state $f_0$,
 since we can easily choose some $f\in\mathrm D(\D)\setminus\ker(\D)$
 such that $\partial_x U_{\D f}\neq0$.
		
 Thus, $\|\partial_xU_{\D g}\|_2^2>0$, in particular, $\D g\neq0$.
 This means that $g$ can be taken as a test function in the infimum $\tilde\lambda$, leading to
 \begin{align*}
   1-\frac1{4\pi}\|\partial_xU_{\D g}\|_2^2
   = \tilde\lambda \leq 1-\frac{\|\partial_xU_{\D g}\|_2^2}{4\pi\|\D g\|_{\Ltwo}^2},
 \end{align*}
 i.e., $\|\D g\|_{\Ltwo}^2\geq1$. On the other hand, the weak lower semicontinuity of
 $L^2$-norms yields
 \begin{align*}
   \|\D g\|_{\Ltwo}^2 \leq \liminf_{k\to\infty}\|\D g_k\|_{\Ltwo}^2 = 1.
 \end{align*}
 Overall, $\|\D g\|_{\Ltwo}^2=1$ and therefore
 \begin{align*}
   \tilde\lambda=1-\frac1{4\pi}\|\partial_xU_{\D g}\|_2^2
   =\langle\A g,g\rangle_{\Ltwo} = \frac{\langle\A g,g\rangle_{\Ltwo}}{\|\D g\|_{\Ltwo}^2},
 \end{align*}
 i.e., $g$ is the desired minimizer.\qedhere
\end{proof}

The techniques we employed in the above proof are similar to the ones used for
Schr\"odinger type operators~\cite[Chapter~11]{LiLo01} and for the Guo-Lin operator, see \cite[Lemma~3.1]{GuLi08} and \cite[Proposition~4.8]{St19}.

We now apply Proposition~\ref{helpervariation} together with the Poincar\'e inequality
from Corollary~\ref{estimatetranssquare} to estimate $\lambda_1$.

\begin{theorem} \label{antonovgap} 
  The constant $\lambda_1$ defined in~\eqref{E:LAMBDAONEDEF} is strictly positive.
  In particular,
  \begin{align}\label{eq:specAbound}
    \sigma(\A)\setminus\{0\}\subset \left[\lambda_1,\infty\right[,
  \end{align}
  i.e, there is a gap in the spectrum of $\A$.
\end{theorem}
\begin{proof}
  Combining Proposition~\ref{helpervariation} and Corollary \ref{estimatetranssquare} yields
  \begin{align}\label{eq:gapestimate}
    \langle\A g,g\rangle_{\Ltwo} \geq
    \tilde\lambda \| \D g\|_{\Ltwo}^2 \geq
    \frac{4\pi^2}{\sup^2_{\mathring\Omega_0^{EL}}(T)}\,\tilde\lambda\; \| g\|_{\Ltwo}^2
  \end{align}
  for every $g\in\mathrm D(\D)$ with $g\perp\ker(\D)$. Therefore
  \begin{align*}
    \lambda_1\geq\tilde\lambda\frac{4\pi^2}{\sup_{\mathring\Omega_0^{EL} }^2(T)}>0.
  \end{align*}
  To conclude \eqref{eq:specAbound}, recall that we have explicitly characterized the
  essential spectrum of $\A$ in Theorem~\ref{essspecA}. In particular, every element of
  $\sigma(\A)$ within the principal gap $\mathcal G$ has to be an eigenvalue. Since $\A$
  is self-adjoint, eigenfunctions corresponding to different eigenvalues are orthogonal
  to each other and the eigenspace of the eigenvalue $0$ equals $\ker(\D)$, see
  Corollary~\ref{antonovev0}. Thus, applying \eqref{eq:gapestimate} to an eigenfunction
  of a non-zero eigenvalue indeed yields \eqref{eq:specAbound}.
\end{proof}

\begin{remark}
  \begin{itemize}
    \item[(a)]
      Estimate~\eqref{eq:gapestimate} can be viewed as an alternate or
      improved version of the Antonov estimate \eqref{eq:antonovcoerc}.
    \item[(b)]\label{gapoddv}
      Since $\Ltwo^{odd}\subset\ker(\D)^\perp$, the spectrum of the restricted operator
      $\A\colon\Ltwo^{odd}\cap\mathrm D(\D^2)\to\Ltwo^{odd}$ is bounded from below by $\lambda_1>0$.
    \item[(c)]
      A naive approach to show the existence of an eigenvalue within the principal gap of
      $\A$ is to find a function odd in $v$ which minimizes the variational problem
      \eqref{E:LAMBDAONEDEF}, i.e., there holds equality in \eqref{eq:gapestimate}.
      However, the latter estimate contains the two separate variational problems
      Theorem~\ref{essspecA} and Proposition~\ref{helpervariation} with differing minimizers,
      which is why we do not pursue this approach any further.
  \end{itemize}
\end{remark}

\subsection{Plane symmetric case}\label{ssc:gapplane}

We now turn our attention to the planar linearized operator
$\pA\colon\mathrm D(\pD^2)\to\pH$, see Definition~\ref{antonovdef1D}.
First, we prove a coercivity bound analogous to Proposition~\ref{antonovcoerc}
in the radial setting.

\begin{prop}\label{antonovcoerc1d}
  Let $g\in C^2_c(\pOmega_0)$ be odd in $v_1$.
  \begin{enumerate}[label=(\alph*)]
  \item $\langle \pA g,g\rangle_{\pH}\geq0$.
  \item If $g$ is also odd in $x$,
    \begin{align}\label{eq:antonov1doddxv1}
      \langle  \pA g,g\rangle_{\pH}\geq
      3 \int_{\pOmega_0}\frac1{\vert\pvarphi'(\pE,\pv)\vert}\; \frac{\pU_0'(x)}x\;
      \vert g(x,v)\vert^2\diff(x,v).
    \end{align}
  \end{enumerate}
\end{prop}
\begin{proof}
  We proceed similarly to \cite[Proof of Lemma 1.1]{GuRe2007}:
  For every $g\in C^2_c(\pOmega_0)$,
  \begin{align*}
    \langle\pB g,g\rangle_{\pH}
    &= 4\pi \int_{\R} \left( \int_{\R^3}v_1\,g(x,v)\diff v \right)^2\diff x \\
    &\leq 4\pi \int_{\R} \left( \int_{\R^3}v_1^2\,\vert\pvarphi'(\pE,\pv)\vert \diff v \right)
    \left( \int_{\R^3}\frac{\vert g(x,v)\vert^2}{\vert\pvarphi'(\pE,\pv)\vert} \diff v \right)
    \diff x  \\
    &= 4\pi \int_{\pOmega_0} \prho_0(x)\frac{\vert g(x,v)\vert^2}{\vert\pvarphi'(\pE,\pv)\vert}
    \diff(x,v),
  \end{align*}
  where we used \eqref{eq:v1intrho} for the last equality. Thus,
  \begin{align}
    \langle \pA g,g\rangle_{\pH}\geq
    \int_{\pOmega_0} \frac1{\vert\pvarphi'(\pE,\pv)\vert} \left(\vert\pD g(x,v)\vert^2
    - 4\pi\prho_0(x)\,\vert g(x,v)\vert^2\right)\diff(x,v).
  \end{align}
  For (a) consider
  \begin{align*}
    \mu(x,v)\coloneqq \frac1{v_1}g(x,v),\quad (x,v)\in\pOmega_0.
  \end{align*}
  Since $g$ is odd in $v_1$, we obtain $\mu\in C^2_c(\pOmega_0)$. Furthermore, 
  \begin{align*}
    \pD g  = \pD\left(v_1\,\mu\right) = v_1 \pD\mu + \mu\,\pD v_1,
  \end{align*}
  and therefore
  \begin{align*}
    \vert \pD g\vert^2
    &= v_1^2 \vert \pD\mu\vert^2 + v_1 \pD\left(\mu^2\right)\pD v_1 + \mu^2\vert\pD v_1\vert^2  
    = v_1^2 \vert \pD\mu\vert^2 + \pD\left( v_1\mu^2\pD v_1 \right) - v_1\mu^2\pD^2(v_1)  \\
    &= v_1^2 \vert \pD\mu\vert^2 + \pD\left( v_1\mu^2\pD v_1 \right) + 4\pi g^2 \prho_0.
  \end{align*}
  Thus,
  \begin{align*}
    \langle \pA g,g\rangle_{\pH} \geq
    \int_{\pOmega_0} \frac1{\vert\pvarphi'(\pE,\pv)\vert}\,v_1^2\, \vert\pD
    \mu(x,v)\vert^2\diff(x,v)\geq0;
  \end{align*}
  note that there are no boundary terms when we integrate by parts to obtain
  \begin{align}\label{eq:antonovibp}
    \int_{\pOmega_0} \frac1{\vert\pvarphi'(\pE,\pv)\vert}\,
    \pD\left( v_1\mu^2\pD v_1 \right)\diff(x,v)=0.
  \end{align}
  For part (b) we consider 
  \begin{align*}
    \mu(x,v)\coloneqq \frac1{v_1\,x}g(x,v),\quad (x,v)\in\pOmega_0;
  \end{align*}
  $\mu\in C^2_c(\pOmega_0)$, since $g$ is odd in $x$ and $v_1$.
  As in the first part,
  \begin{align*}
    \pD g  = \pD\left(v_1x\,\mu\right) = v_1x\,\pD\mu + \mu\,\pD\left(v_1x\right),
  \end{align*}
  and hence
  \begin{align*}
    \vert \pD g\vert^2
    &= v_1^2x^2 \vert \pD\mu\vert^2
    + \pD\left( v_1x\,\mu^2\pD\left(v_1x\right)\right) - v_1x\,\mu^2\pD^2(v_1x) \\
    &= v_1^2x^2 \vert \pD\mu\vert^2 + \pD\left( v_1x\,\mu^2\pD\left(v_1x\right)\right)
    + 4\pi g^2 \prho_0 + 3 g^2 \frac{\pU_0'}x.
  \end{align*}
  Thus,
  \begin{align*}
    \langle \pA g,g\rangle_{\pH} \geq
    \int_{\pOmega_0} \frac1{\vert\pvarphi'(\pE,\pv)\vert}
    \left(v_1^2x^2\, \vert\pD \mu(x,v)\vert^2 + 3\frac{\pU_0'(x)}x\vert g(x,v)\vert^2 \right)
    \diff(x,v),
  \end{align*}
  where we again integrated by parts similarly to \eqref{eq:antonovibp}.
  The reason why we required $g$ to be odd in $x$ for part (b) is that otherwise
  the occurring boundary terms at $x=0$ do not vanish as they do in the spherically
  symmetric setting.
\end{proof}

We again obtain the non-negativity of $\pA$:

\begin{cor}\label{C:SPECTRUMNONNEGATIVEplane}
  The quadratic form associated with $\pA$ is non-negative on $\mathrm D(\pD)$.
  Thus, $\sigma(\pA)\subset[0,\infty[$.
\end{cor}
\begin{proof}
  We split $g\in\mathrm D(\pD)$ into its even and odd part $g_{\pm}$ in $v_1$
  analogously to \eqref{eq:gevenoddv}.
  Then $\langle \pA (g_+),g_+\rangle_{\pH}\geq0$ since $\pB(g_+)=0$.
  On the other hand, an approximation argument (as in Remark \ref{approxplane})
  allows us to extend Proposition \ref{antonovcoerc1d} (a) onto
  functions in $\mathrm D(\pD)$ which are odd in $v_1$, 
  i.e., $\langle \A (g_-),g_-\rangle_{\pH}\geq0$.
\end{proof}

We now prove that $\pA$, when restricted to functions being odd in $v_1$ and $x$,
possesses a gap in its spectrum as well. Since in the plane symmetric setting
$\prho_0(0)$ is always positive, the proof is easier than in the spherically symmetric case.

\begin{theorem}\label{gapplane}
  There exists $c>0$ such that for all $g\in\mathrm D(\pD)\cap\p\H$,
  \begin{align*}
    \langle \pA g,g\rangle_{\pH} \geq c \| g\|_{\pH}^2.
  \end{align*}
  In particular, the kernel of $\pA\colon\mathrm D(\pD^2)\cap\p\H\to\p\H$ is trivial.
\end{theorem} 
\begin{proof}
  Approximating as in Remark \ref{approxplane} allows us to extend the coercivity
  bound~\eqref{eq:antonov1doddxv1} to $\mathrm D(\pD)\cap\p\H$;
  note that the weight $\frac{\pU_0'(x)}x$ is bounded on $\R$,
  since $\lim_{x\to0}\frac{\pU_0'(x)}x = \pU_0''(0)=4\pi\prho_0(0)<\infty$.
  Furthermore, since $\pU_0'\neq0$ on $\R\setminus\{0\}$ and $\prho_0(0)>0$,
  the weight $\frac{\pU_0'(x)}x$ is also bounded from below on the
  support of the steady state, which immediately implies the above estimate.
\end{proof}

In other words, the spectrum of $\pA\colon\mathrm D(\pD)\cap\p\H\to\p\H$
is bounded away from zero, cf.~\cite[Proposition~5.12]{HiSi}

Note that we do not establish the existence of a spectral gap of the operator $\pA$ on the whole function space $\pH$ with no added symmetry assumptions as we did in the radial case. The reason for that is that we need Antonov's coercivity bound to deduce that the kernel of $\pA$ equals the kernel of $\pD$, and our proof of this result requires oddness in $x$ in the planar case. 

\section{Existence of eigenvalues}\label{sc:mathur}

As explained in the introduction, S.~Mathur~\cite{Ma} used a version of the Birman-Schwinger
principle to show the existence of periodic solutions to the linearized Vlasov-Poisson system
in the presence of an external potential. 
In this section we use this approach to derive a criterion for the existence of periodic oscillations without such an external potential.
More precisely, we are interested in the existence of eigenfunctions
for the operators $\A$ and $\pA$, see Definitions~\ref{antonovdef}
and~\ref{antonovdef1D} respectively. In the radial case, we seek an eigenfunction
which is odd in $v$/$w$, in the planar case the eigenfunction has to be plane
symmetric in the sense of~\eqref{pl_symm_def}. 
This is why we always restrict $\A$ to odd functions in this section, i.e., by $\A$ we denote
\begin{align*}
  \A=\A\big|_{\Ltwo^{odd}} \colon\mathrm D(\D^2) \cap\Ltwo^{odd}\to\Ltwo^{odd},
\end{align*}
$\D^2$ and $\B$ are restricted accordingly; see Section \ref{ssc:operatorradial} for the definitions of these function spaces and operators.

In the planar setting we restrict the linearized operator $\pA$ onto the functions which are odd in $v_1$ and $x$, i.e., by $\pA$ we denote
\begin{align*}
	\pA=\pA\big|_{\p\H} \colon\mathrm D(\pD^2) \cap\p\H\to\p\H,
\end{align*}
$\pD^2$ and $\pB$ are restricted accordingly; the function spaces and operators are defined in Section~\ref{ssc:operatorplane}, see in particular~\eqref{E:HODDDEF} for the definition of $\p\H$. 


\subsection{Mathur's argument and a criterion for the existence of eigenvalues}

We reformulate the eigenvalue problem using a Birman-Schwinger type argument.
For $g\in \Ltwo^{odd}\cap\mathrm D(\D^2)$ and $\lambda\notin\sigma_{ess}(\A)$ let 
\begin{align*}
	h \coloneqq \left(-\D^2 - \lambda \right) g \in \Ltwo^{odd}.
\end{align*} 
Since $\sigma_{ess}(\A) = \sigma_{ess}(-\D^2) = \sigma(-\D^2)$, the resolvent operator
\begin{align*}
	R_{-\D^2}(\lambda) \coloneqq \left(-\D^2 - \lambda \right)^{-1}\colon\Ltwo^{odd}\to\Ltwo^{odd}
\end{align*}
is bounded and we can recover $g$ from $h$ via $g = R_{-\D^2}(\lambda)h$. 
It is clear that $g$ is an eigenfunction of $\A$ with eigenvalue $\lambda$ if and only if
$
	\left(-\D^2 - \lambda\right) g = \B g,
$
i.e., 
\begin{align*}
	h = \B g= \B R_{-\D^2}(\lambda)h.
\end{align*}
In the planar case, we conclude similarly that $g\in\mathrm D(\pD^2)\cap\p\H$ is an eigenfunction of $\pA$ with eigenvalue $\lambda\notin\sigma_{ess}(\pA)$ if and only if
$\left(-\pD^2 - \lambda\right) g = \pB g$, or equivalently $h = \pB R_{-\pD^2}(\lambda) h$ with $h=\left(-\pD^2 - \lambda \right) g$.
It is therefore natural to introduce the $\l$-parametrized families of operators
\begin{align}
Q_\l\coloneqq& \B R_{-\D^2}(\lambda)\colon\Ltwo^{odd}\to\Ltwo^{odd}, \\
\pQ_\l\coloneqq& \pB R_{-\pD^2}(\lambda)\colon\p\H\to\p\H.
\end{align}
Before proving some general properties of these operators in Lemma \ref{Qlambdaprop}
we first discuss the connection between their eigenvalues and the ones of $\A$ and $\pA$
respectively.
We have shown above that $\lambda$ is an eigenvalue of $\A$ ($\pA$) if and only if $1$
is an eigenvalue of $Q_\lambda$ ($\pQ_\l$).
In fact, it is easy to see that the following lemma holds.


\begin{lemma}\label{L:BS}
  Let $\lambda\notin\sigma_{ess}(\A)$ ($\lambda\notin\sigma_{ess}(\pA)$) and $\mu\geq1$.
  Then $\lambda$ is an eigenvalue of $-\D^2 - \frac1\mu\B$ ($-\pD^2 - \frac1\mu\pB$)
  if and only if $\mu$ is an eigenvalue of $Q_\lambda$ ($\pQ_\l$).
\end{lemma}


Lemma~\ref{L:BS} is simply a version of the Birman-Schwinger principle~\cite{LiLo01}.
In particular, the non-negativity of $\B$ ($\pB$) (in the sense of quadratic forms)
allows us to conclude the existence of an eigenvalue of $\A$ ($\pA$) in the principal
gap if we know that there exists an eigenvalue of
$-\D^2 - \frac1\mu\B$ ($-\pD^2 - \frac1\mu\pB$) in the principal gap for
sufficiently large $\mu$. This is shown in the next lemma.

\begin{lemma}\label{L:BSONE}
  \begin{enumerate}[label=(\alph*)]
  \item
    If there exists $\mu\geq1$ such that $-\D^2 - \frac1\mu\B$
    possesses an eigenvalue in the principal gap
    \begin{align}\label{eq:defprincgap}
      \mathcal G= \left]0,\frac{4\pi^2}{\sup^2_{\mathring\Omega_0^{EL}}(T)}\right[,
    \end{align}
    then $\A$ also has an eigenvalue in $\mathcal G$.
  \item
    If there exists $\mu\geq1$ such that $-\pD^2 - \frac1\mu\pB$ possesses an
    eigenvalue in the principal gap 
    \begin{align}\label{eq:defprincgapplanar}
      \p{\mathcal G} := \left]0, 4\frac{4\pi^2}{\pT^2(\pE_0)}\right[,
    \end{align}
    then $\pA$ also has an eigenvalue in $\p{\mathcal G}$.
  \end{enumerate}
\end{lemma}


\begin{proof}
  For part (a), $\B\geq0$ (in the sense of quadratic forms, see \cite[Definition 5.11]{HiSi})
  and $\frac1\mu\leq1$ yield $-\D^2-\frac1\mu\B\geq\A$.
  Note that $-\D^2-\frac1\mu\B$ is self-adjoint since $\D^2$ is self-adjoint and $\B$
  is bounded and symmetric. By assuming the existence of an eigenvalue of $-\D^2-\frac1\mu\B$
  in the principal gap, we have $-\D^2 - \frac1\mu\B<\frac{4\pi^2}{\sup^2(T)} \mathrm{id}$,
  where $<$ is defined as $\not\geq$.
  Thus, transitivity yields $\A<\frac{4\pi^2}{\sup^2(T)} \mathrm{id}$, which means that 
  \begin{align*}
    \sigma(\A) \not\subset \left[ \frac{4\pi^2}{\sup^2_{\mathring\Omega_0^{EL}}(T)}, \infty \right[.
  \end{align*}
  Since we have restricted ourselves to functions odd in $v$, Theorem~\ref{antonovgap}
  (see also Remark~\ref{gapoddv}~(b)) implies $\sigma(\A)\subset]0,\infty[$,
  i.e., we obtain $\sigma(\A)\cap\mathcal G\neq\emptyset$. By Theorem~\ref{essspecA}
  every element of the spectrum of $\A$ within $\mathcal G$
  has to be an (isolated) eigenvalue of $\mathcal A$.
	
	The proof of part (b) is analogous; the positivity of the resulting eigenvalue is ensured by Theorem~\ref{gapplane}. Recall however that we have restricted the operators onto $\p\H$ in the planar setting and the essential spectrum of $\pA$ is given by \eqref{eq:essspecantonovplaneoddv1x}. In particular, the bottom of the essential spectrum is $4\frac{4\pi^2}{\pT^2(\pE_0)}$.
\end{proof}


We now investigate the properties of $Q_\lambda$ and $\pQ_\l$ for fixed $\lambda\in\rho(-\D^2)$ and $\lambda\in\rho(-\pD^2)$ respectively. As usual, $\rho\coloneqq\C\setminus\sigma$ denotes the resolvent set. Since all involved operators are self-adjoint, we always restrict ourselves to real $\lambda$. 


\begin{lemma}\label{Qlambdaprop}
  The operators $Q_\lambda\colon \Ltwo^{odd}\to \Ltwo^{odd}$ and $\pQ_\l:\p\H\to \p\H$ are
  both linear, continuous, and compact.
\end{lemma}


\begin{proof}
  We only prove the claim for $Q_\l$ as the claim for $\pQ_\l$ follows analogously.
  Linearity is obvious, continuity follows by Lemma \ref{bdef} together with the boundedness of the resolvent operator $R_{-\D^2}(\lambda)$ in the case $\lambda\in\rho(-\D^2)$. The compactness of $Q_\lambda$ is equivalent to the relative $(-\D^2)$-compactness of $\B$, see Theorem~\ref{essspecA} (and \cite[Definition 14.1]{HiSi}).
\end{proof}


Since $\text{im}(Q_\l)\subset\text{im}(\mathcal B)$, it is of interest to characterize
measurable functions $G\colon[0,\infty[\to\R$ such that
$g(r,w,L) = \vert\varphi'(E,L)\vert\,w\,G(r)\in \Ltwo$;
recall the definition of $\B$ in \eqref{E:BDEF}.


\begin{lemma} \label{Wintegrability}
  \begin{enumerate}[label=(\alph*)]
  \item If a spherically symmetric $g\colon\Omega_0\to\R$ is of the form 
    $g(r,w,L) = \vert\varphi'(E,L)\vert\,w\,G(r)$, then
    \begin{align*}
      g\in\Ltwo \Leftrightarrow \int_0^\infty r^2\,\rho_0(r)\,G^2(r) \diff r < \infty.
    \end{align*}
    In particular, by the boundedness of $\rho_0$, $\int_0^\infty r^2\,G^2(r) \diff r < \infty$
    implies $g\in \Ltwo$.
  \item If $g\colon\pOmega_0\to\R$ is of the form 
    $g(x,v) = \vert\pvarphi'(\pE,\pv)\vert\,v_1\,G(x)$, then
    \begin{align*}
      g\in \pH \Leftrightarrow \int_{\R} \prho_0(x)\,G^2(x) \diff x < \infty.
    \end{align*}
    In particular, by the boundedness of $\prho_0$,
    $\int_{\mathbb R} G^2(x) \diff x < \infty$ implies $g\in \pH$.
    Furthermore, $g$ is odd in $x$ if and only if $G$ is odd in $x$.
  \end{enumerate}
\end{lemma}
\begin{proof}
  As to part (a), Fubini's theorem shows that $g\in \Ltwo$ is equivalent to 
  \begin{align*}
    \infty > \|g\|_{\Ltwo}^2
    &= 4\pi^2\int_{\Omega_0^r} \vert\varphi'(E,L)\vert\,w^2\,G^2(r) \diff(r,w,L) \\
    &= 4\pi^2 \int_0^\infty G^2(r) \int_0^\infty \int_{\R} w^2\,\vert\varphi'(E,L)\vert
    \diff w \diff L \diff r \\
    &= 4\pi\int_0^\infty r^2\,\rho_0(r)\,G^2(r)\diff r,
  \end{align*}
  where we used \eqref{eq:wintrho} for the last equality.
	
  Similarly, for $g\colon\pOmega_0\to\R$,
  \begin{align*}
    \|g\|_{\pH}^2
    & = \int_{\R}G^2(x)\int_{\R^3}v_1^2\, \vert\pvarphi'(\pE,\pv)\vert\diff v\diff x
    = \int_{\mathbb R}\prho_0(x)\,G^2(x)\diff x
  \end{align*}	
  by \eqref{eq:v1intrho}, and part (b) follows.
\end{proof}


Let 
\begin{align}
S &\coloneqq \left\{ r\geq0\mid\rho_0(r)\neq0 \right\}\subset[0,R_0[, \\
\pS&\coloneqq \left\{ x\in\mathbb R\mid\prho_0(x)\neq0 \right\}=]-\pR_0,\pR_0[,
\end{align}
denote the spatial support of the respective steady states.
Based on the previous lemma, it is natural to introduce the spaces
\begin{align}
		\mathcal F&\coloneqq\left\{ G\colon S\to\R\text{ measurable}\mid \int_S r^2\,\rho_0(r)\,G^2(r)\diff r < \infty \right\},\\
		\mathcal F_1&\coloneqq\left\{ G\colon S\to\R\text{ measurable}\mid \int_S r^2\,G^2(r)\diff r < \infty \right\}\subset \mathcal F, \label{E:FONEDEF}
	\end{align}
	and similarly
	\begin{align}
		\p{\mathcal F}&\coloneqq\left\{ G\colon\pS\to\R\text{ measurable and odd}\mid \int_{\pS} \prho_0(x) G^2(x)\diff x < \infty \right\},\\
		\p{\mathcal F}_1&\coloneqq\left\{ G\colon\pS\to\R\text{ measurable and odd}\mid \int_{\pS} G^2(x)\diff x < \infty \right\}\subset \p{\mathcal F}. \label{E:FONEDEFBAR}
	\end{align}
We treat $\mathcal F, \mathcal F_1, \p{\mathcal F}, \p{\mathcal F}_1$ as (subsets of weighted) $L^2$-spaces, i.e., we identify functions which are equal a.e.~on $S$ or $\pS$ respectively. Their norms $\|\cdot\|_{\mathcal F}$, $\|\cdot\|_{\mathcal F_1}$, $\|\cdot\|_{\p{\mathcal F}}$, $\|\cdot\|_{\p{\mathcal F}_1}$ and scalar products $\langle\cdot,\cdot\rangle_{\mathcal F}$, $\langle\cdot,\cdot\rangle_{\mathcal F_1}$, $\langle\cdot,\cdot\rangle_{\p{\mathcal F}}$, $\langle\cdot,\cdot\rangle_{\p{\mathcal F}_1}$ are defined accordingly.

\begin{defnrem}[The Mathur operators]\label{defmathur}
  \begin{enumerate}[label=(\alph*)]
  \item
    Let $G\in\mathcal F$ and define $g(r,w,L)\coloneqq\vert\varphi'(E,L)\vert\,w\,G(r)$.
    Then by Lemma~\ref{Wintegrability}, $g\in \Ltwo^{odd}$. Since $Q_\lambda g \in\mathrm{im}(\B)$,
    there exists $F\in\mathcal F$ such that
    $(Q_\lambda g)(r,w,L)=\vert\varphi'(E,L)\vert\,w\,F(r)$
    for a.e.~$(r,w,L)\in\Omega_0^r$, and $F$ is uniquely
    determined by $G$. This defines a map
    \begin{align*}
      \mathcal M_\l\colon\mathcal{F}\to \mathcal F,~ G\mapsto F .
    \end{align*}
  \item Let $G\in\p{\mathcal F}$ and define
    $g(x,v)\coloneqq\vert\pvarphi'(\pE,\pv)\vert\, v_1\, G(x)$. Then by Lemma~\ref{Wintegrability},
    $g\in\p\H$. Since $\pQ_\l g\in\mathrm{im}(\pB)$,
    there exists a unique $F\in\p{\mathcal F}$ such that
    $(\pQ_\l g)(x,v)=\vert\pvarphi'(\pE,\pv)\vert\, v_1\, F(x)$ for a.e.~$(x,v)\in\pOmega_0$.
    This defines a map
    \begin{align*}
      \p{\mathcal M}_\l \colon \p{\mathcal{F}}\to \p{\mathcal F},~ G\mapsto F .
    \end{align*}
  \end{enumerate}
\end{defnrem}

These operators were introduced by Mathur~\cite{Ma}, and we refer to $\mathcal M_\l$
$(\p{\mathcal M}_\l)$ as the {\em (planar) Mathur operator}.
The following key integral kernel representation of $\mathcal M_\l$
is essentially also contained in~\cite{Ma}.

\begin{prop}\label{Klambdadef}
  For any $G\in\mathcal F$,
  \begin{align}\label{eq:mathurkernel}
    (\mathcal M_\l G)(r)= \int_S K_\lambda(r,\sigma)\,G(\sigma)\diff\sigma
  \end{align} 
  for a.e.~$r\in S$, where
  \begin{align}\label{E:KLAMBDADEF}
    K_\lambda(r,\sigma)\coloneqq
    \frac{32\pi^2}{r^2} \sum_{k=1}^\infty
    \int_{\Omega_0^{EL}(r)\cap\Omega_0^{EL}(\sigma)}
    \frac{\vert\varphi'(E,L)\vert}{T(E,L)}
    \frac{\sin(2\pi k\theta(r,E,L))\sin(2\pi k\theta(\sigma,E,L))}
    {\frac{4\pi^2}{T(E,L)^2}k^2-\lambda}\diff(E,L)
  \end{align}
  for $r,\sigma>0$,
  \begin{align*}
    \Omega_0^{EL}(r) \coloneqq
    \left\{ (E,L)\in\mathring\Omega_0^{EL} \mid r_-(E,L)<r<r_+(E,L) \right\}, \quad r>0,
  \end{align*}
  and $\theta$ is defined in \eqref{eq:tauofr}.
  Moreover, the map $]0,\infty[^2\ni (r,\sigma)\mapsto K_\l(r,\sigma)\in\mathbb R$ is continuous.
\end{prop}
\begin{proof}
	To apply $Q_\lambda$ on $g\in\Ltwo^{odd}$ defined by $g(r,w,L)=\vert\varphi'(E,L)\vert\,w\,G(r)$, we first have to apply the resolvent operator $R_{-\D^2}(\lambda)$. For this purpose we expand $g$ in its $\theta$-Fourier series (recall that $g$ is odd w.r.t.\ $v$):
	\begin{align} \label{eq:gfourier}
		g(\theta,E,L) = \sum_{k=1}^\infty b_k(E,L) \sin(2\pi k\theta),
	\end{align} 
	where
	\begin{align}
		b_k(E,L) \coloneqq 2\int_0^1 g(\theta,E,L) \sin(2\pi k\theta)\diff\theta,\qquad(E,L)\in\mathring\Omega_0^{EL}.
	\end{align}
Using \eqref{eq:gfourier}---which holds as a limit in $\Ltwo$---now allows us to apply the resolvent operator:
	\begin{align} \label{eq:gresolvent}
		\left(R_{-\D^2}(\lambda) g\right) (\theta,E,L) = \sum_{k=1}^\infty \frac1{\frac{4\pi^2}{T(E,L)^2}k^2-\lambda} b_k(E,L) \sin(2\pi k\theta)
	\end{align}
	as a limit in $\Ltwo$; it can be easily verified that the resolvent operator is of the form \eqref{eq:gresolvent} by applying Corollary~\ref{trans1dsquare} to the Fourier series.

	To apply $\B$, we change variables via \eqref{eq:tauofr} and get
	\begin{align}\label{eq:bkrwl}
		b_k(E,L) &= 4\int_0^{\frac12} g(\theta,E,L) \sin(2\pi k\theta) \diff\theta \nonumber\\
		&= 4 \frac1{T(E,L)} \int_{r_-(E,L)}^{r_+(E,L)} \frac{g(\theta(r,E,L),E,L)}{\sqrt{2E-2\Psi_L(r)}} \sin(2\pi k\theta(r,E,L)) \diff r\nonumber  \\
		&= 4 \frac{\vert\varphi'(E,L)\vert}{T(E,L)} \int_{r_-(E,L)}^{r_+(E,L)} G(r) \sin(2\pi k\theta(r,E,L)) \diff r.	
	\end{align}
	As for  $\B$, recall that for every $f\in \Ltwo^{odd}$,
	\begin{align} \label{eq:Brecall}
		\left(\B f\right)(r,w,L) &= 4\pi^2\vert\varphi'(E,L)\vert\frac w{r^2}\int_0^\infty\int_\R \tilde w f(r,\tilde w,\tilde L)\diff \tilde w\diff \tilde L  \nonumber\\
		&= 8\pi^2 \vert\varphi'(E,L)\vert\frac w{r^2}\int_0^\infty\int_0^\infty \tilde w f(r,\tilde w,\tilde L)\diff \tilde w\diff \tilde L\nonumber\\
		&= 8\pi^2 \vert\varphi'(E,L)\vert\frac w{r^2}\int_{\Omega_0^{EL}(r)} f(r,\sqrt{2\tilde E-2\Psi_{\tilde L}(r)},\tilde L)\diff(\tilde E,\tilde L).
	\end{align}
	Inserting $f=R_{-\D^2}(\lambda)g$ yields the following for the integral in \eqref{eq:Brecall} for $r>0$:
	\begin{align}\label{eq:Klambdafinal}
		&\int_{\Omega_0^{EL}(r)} \left(R_{-\D^2}(\lambda)g\right)(\theta(r,E,L),E,L)\diff(E,L)  \nonumber\\
		&= 4 \int_{\Omega_0^{EL}(r)} \sum_{k=1}^\infty \frac{\vert\varphi'(E,L)\vert}{T(E,L)} \frac{\sin(2\pi k\theta(r,E,L))}{\frac{4\pi^2}{T(E,L)^2}k^2-\lambda}\int_{r_-(E,L)}^{r_+(E,L)}G(\sigma) \sin(2\pi k\theta(\sigma,E,L)) \diff\sigma\diff(E,L)  \nonumber\\
		&= 4\int_0^\infty G(\sigma) \sum_{k=1}^\infty \int_{\Omega_0^{EL}(r)\cap\Omega_0^{EL}(\sigma)}\frac{\vert\varphi'(E,L)\vert}{T(E,L)}\frac{\sin(2\pi k\theta(r,E,L))\sin(2\pi k\theta(\sigma,E,L))}{\frac{4\pi^2}{T(E,L)^2}k^2-\lambda}\diff(E,L)\diff\sigma,
	\end{align}
	where we used \eqref{eq:gresolvent}, \eqref{eq:bkrwl}, and Fubini's theorem. Note that we switched the infinite sum with the $\sigma$-integral and then with the $(E,L)$-integral in the last equality. Switching with the $\sigma$-integral is verified by fixing $(E,L)\in\Omega_0^{EL}(r)$ and observing that the weight $r^2\rho_0(r)$ is bounded away from $0$ on $[r_-(E,L),r_+(E,L)]$ if $L>0$, i.e., $G$ is integrable over $[r_-(E,L),r_+(E,L)]$. Furthermore, recall that the distance of $\lambda$ to the essential spectrum of $-\D^2$ is positive. The second switch will be justified below.
	Together with \eqref{eq:Brecall} we get the desired representation \eqref{eq:mathurkernel} of $\mathcal M_\lambda G$.
	
	To prove the continuity of $(r,\sigma)\mapsto K_\l(r,\sigma)$ we first extend the mapping $\theta$ from \eqref{eq:tauofr} by setting
	\begin{align}
		\theta(r,E,L)\coloneqq0 \text{ for } r>0,~(E,L)\in\mathring\Omega_0^{EL}\setminus\Omega_0^{EL}(r).
	\end{align}
	Then 
	\begin{align}
		(r,\sigma)\mapsto\frac{\vert\varphi'(E,L)\vert}{T(E,L)} \frac{\sin(2\pi k\theta(r,E,L))\sin(2\pi k\theta(\sigma,E,L))}{\frac{4\pi^2}{T(E,L)^2}k^2-\lambda}
	\end{align}
	is continuous on $]0,\infty[^2$ for every $k\in\N$ and $(E,L)\in\mathring\Omega_0^{EL}$. Moreover, there exists a constant $C>0$ depending only on $\lambda$---more precisely, on the distance of $\lambda$ to the essential spectrum of $-\D^2$ (given by Theorem~\ref{transsquarespectrum}), which is positive---such that
	\begin{align}\label{eq:Klambdabound}
		&\left| \sum_{k=1}^n\frac{\vert\varphi'(E,L)\vert}{T(E,L)} \frac{\sin(2\pi k\theta(r,E,L))\sin(2\pi k\theta(\sigma,E,L))}{\frac{4\pi^2}{T(E,L)^2}k^2-\lambda} \right|\nonumber\\
		&\qquad\qquad\qquad\qquad\leq \frac{\vert\varphi'(E,L)\vert}{\inf_{\mathring\Omega_0^{EL}}(T)} C\sum_{k=1}^n\frac1{k^2}\leq\frac{\vert\varphi'(E,L)\vert}{\inf_{\mathring\Omega_0^{EL}}(T)}C\frac{\pi^2}6
	\end{align} 
	for $r,\sigma>0$, $(E,L)\in\mathring\Omega_0^{EL}$ and $n\in\N$; note $\inf_{\mathring\Omega_0^{EL}}(T)>0$ by Proposition~\ref{Tbounded}. This shows that for fixed $(E,L)\in\mathring\Omega_0^{EL}$ the limit
	\begin{align}\label{eq:Klambdaintegrand}
		\sum_{k=1}^\infty\frac{\vert\varphi'(E,L)\vert}{T(E,L)} \frac{\sin(2\pi k\theta(r,E,L))\sin(2\pi k\theta(\sigma,E,L))}{\frac{4\pi^2}{T(E,L)^2}k^2-\lambda}
	\end{align} 
	exists uniformly in $r,\sigma>0$ and \eqref{eq:Klambdaintegrand} defines a continuous function in $(r,\sigma)$. Moreover, since $\vert\varphi'\vert$ is integrable over $\Omega_0^{EL}$ by \eqref{eq:A7prime} (after changing variables via \eqref{eq:actionanglevolume}), we conclude that
	\begin{align}\label{eq:Klambdaclose}
		\int_{\Omega_0^{EL}}\sum_{k=1}^\infty\frac{\vert\varphi'(E,L)\vert}{T(E,L)} \frac{\sin(2\pi k\theta(r,E,L))\sin(2\pi k\theta(\sigma,E,L))}{\frac{4\pi^2}{T(E,L)^2}k^2-\lambda} \diff(E,L)
	\end{align}
	is also continuous in $r,\sigma>0$. The dominated convergence theorem and \eqref{eq:Klambdabound} yield that we can switch the order of the integral and the infinite sum in \eqref{eq:Klambdaclose}, i.e., it follows that the kernel $K_\lambda$ is indeed continuous on $]0,\infty[^2$.
\end{proof}


In the planar case, we can prove an analogous statement by exactly the same method.


\begin{prop}\label{Klambdadefplane}
  For any $G\in\p{\mathcal F}$,
  \begin{align}\label{eq:mathurkernelplanar}
    \p{\mathcal M}_\l(G)(x)= \int_{\pS} \pK_\lambda(x,y)\,G(y)\,\diff y
  \end{align} 
  for a.e.~$x\in \pS$, where
  \begin{align}\label{E:KLAMBDABARDEF}
    \pK_\lambda(x,y)\coloneqq
    32\pi \sum_{k=1}^\infty \int_{\pOmega_0^{\pE}(x)\cap\pOmega_0^{\pE}(y)}
    \frac{\vert\palpha'(\pE)\vert}{\pT(\pE)}
    \frac{\sin(4\pi k\theta(x,\pE))\sin(4\pi k\theta(y,\pE))}
    {\frac{4\pi^2}{\pT(\pE)^2}(2k)^2-\lambda}\diff \pE
  \end{align}
  for $x,y\in\mathbb R$,
  \begin{align*}
    \pOmega_0^{\pE}(x) \coloneqq \left\{ \pE\in]\pUmin,\pE_0[\, \mid x_-(\pE)<x<x_+(\pE) \right\},
    \quad x\in\mathbb R,
  \end{align*}
  and $\theta$ is defined in \eqref{E:THETAPLANARDEF}.
  Moreover, $\R^2\ni (x,y)\mapsto \pK_\l(x,y)\in\R$ is a bounded continuous function.
  Here $\alpha$ is the microscopic equation of state related to $f_0$ via~\eqref{eq:iso1D}.
\end{prop}


\begin{proof}
  Except for obvious changes the proof is almost identical to the one of
  Proposition~\ref{Klambdadef}. The resolvent operator reads
  \begin{align} \label{eq:gresolventbar}
    \left(R_{-\pD^2}(\lambda) g\right)(\theta,\pE, \pv)
    = \sum_{k=1}^\infty \frac1{\frac{4\pi^2}{\pT(\pE)^2}
      (2k)^2-\lambda} b_k(\pE,\pv) \sin(4\pi k\theta),
  \end{align}
  where 
  \begin{align} \label{eq:gfourierplane}
    b_k(\pE,\pv) \coloneqq
    2\int_0^1 g(\theta,\pE,\pv) \sin(4\pi k\theta)\diff\theta\;\;
    \text{ and }\;\;   g(\theta,\pE,\pv) = \sum_{k=1}^\infty b_k(\pE,\pv) \sin(4\pi k\theta)
  \end{align}
	is the $\theta$-Fourier series expansion of $g\in\p\H$ defined by $g(x,v)=\vert\pvarphi'(\pE,\pv)\vert\,v_1\,G(x)$. Recall that $g$ is odd in $x$ and $v_1$, so the Fourier series does not contain any $\cos$-terms nor $\sin(2\pi k\cdot)$-terms with odd $k$, see also Remark~\ref{approxplane}. We then follow the proof above and note that the $\pv$-integral can be computed explicitly since $\pvarphi'(\pE,\pv)=\palpha'(\pE)\,\pbeta(\pv)$.
\end{proof}

For the spectral analysis it is beneficial to treat $\mathcal M_\lambda$ ($\p{\mathcal M}_\l$)
as an operator on $\mathcal F_1$ ($\p{\mathcal F}_1$) instead of $\mathcal F$ ($\p{\mathcal F}$),
since the Mathur operators are symmetric on the former spaces.

\begin{lemma}\label{mathurF1}
  Define $\mathcal M_\lambda\colon\mathcal F_1\to\mathcal F_1$
  and $\p{\mathcal M}_\lambda\colon\p{\mathcal F}_1\to\p{\mathcal F}_1$
  by \eqref{eq:mathurkernel} and \eqref{eq:mathurkernelplanar} respectively.
  These operators are well-defined, linear, bounded, and self-adjoint.
  Furthermore, they are Hilbert-Schmidt operators and therefore compact, see e.g.~\cite[Theorems~VI.22 and~VI.23]{ReSi1}.
\end{lemma}
\begin{proof}
	For all the claimed properties except of the self-adjointness, it suffices to show that the integrals
	\begin{align}
		\int_S\int_S \frac {r^2}{\sigma^2} K_\lambda^2(r,\sigma)\diff\sigma\diff r, \ \ 
		\int_{\pS}\int_{\pS}\pK_\lambda^2(x,y)\diff y\diff x \label{eq:mathurkernelL2}
	\end{align}
	are finite, since for $G\in\mathcal F_1$ we have by Cauchy-Schwarz
	\begin{align*}
		\|\mathcal M_\lambda G\|_{\mathcal F_1}^2 &= \int_S r^2\,\left(\int_S K_\lambda(r,\sigma)\,G(\sigma)\diff\sigma\right)^2\diff r\leq \|G\|_{\mathcal F_1}^2 \int_S\int_S \frac {r^2}{\sigma^2} K_\lambda^2(r,\sigma)\diff\sigma\diff r;
	\end{align*}
	a similar estimate holds in the planar case.
	
	If the spherically symmetric steady state is of polytropic form \eqref{E:POLYTROPE} with $L_0>0$, then the spatial support is bounded away from $0$ due to the presence of an inner vacuum region. In this case \eqref{eq:mathurkernelL2} is obviously finite since the integrand is bounded on $S^2$. In the planar setting, the finiteness of the second integral in~\eqref{eq:mathurkernelL2} follows similarly by the boundedness of the integrand as there are no singular terms in~\eqref{E:KLAMBDABARDEF}. The finiteness of the first integral in~\eqref{eq:mathurkernelL2} for the remaining radial steady states is more challenging but will be shown below. Before that, observe that the symmetry of $\mathcal M_\lambda\colon\mathcal F_1\to\mathcal F_1$ follows from
	\begin{align*}
		\langle\mathcal M_\lambda G,F\rangle_{\mathcal F} = \int_S\int_Sr^2K_\lambda(r,\sigma)\,G(\sigma)\,F(r)\diff\sigma\diff r
	\end{align*}
	and the symmetry of $r^2K_\lambda(r,\sigma)$ in $(r,\sigma)$. The planar case is obvious.
	
	Now consider a radial steady state of King type \eqref{E:KING} or a polytrope \eqref{E:POLYTROPE} with $L_0=0$. In particular, $\varphi(E,L)=\varphi(E)$ in both cases, i.e., the steady state is isotropic. For every $f\in L^2(S\times S)$,
\begin{align*}
  &\left|\int_S\int_S\frac r\sigma K_\lambda(r,\sigma)f(r,\sigma)\diff\sigma\diff r\right|\\
		&\leq32\pi^2\sum_{k=1}^\infty\int_{\Omega_0^{EL}}\frac{\vert\varphi'(E)\vert}{T(E,L)}\frac1{\frac{4\pi^2}{T(E,L)^2}k^2-\lambda}\\&\qquad\qquad\int_{r_-(E,L)}^{r_+(E,L)}\int_{r_-(E,L)}^{r_+(E,L)}\left|\frac{\sin(2\pi k\theta(r,E,L))}r\frac{\sin(2\pi k\theta(\sigma,E,L))}\sigma f(r,\sigma)\right|\diff\sigma\diff r\diff(E,L)\\
		&\leq32\pi^2\|f\|_{L^2(S\times S)}\sum_{k=1}^\infty\int_{\Omega_0^{EL}}\frac{\vert\varphi'(E)\vert}{T(E,L)}\frac1{\frac{4\pi^2}{T(E,L)^2}k^2-\lambda}\int_{r_-(E,L)}^{r_+(E,L)}\frac{\sin^2(2\pi k\theta(r,E,L))}{r^2}\diff r\diff(E,L)
	\end{align*}
	by inserting \eqref{E:KLAMBDADEF}; the following calculations will show that we can switch the infinite sum with the integrals. Thus,
	\begin{align*}
		\int_S&\int_S\frac {r^2}{\sigma^2} K_\lambda^2(r,\sigma)\diff\sigma\diff r\\
		&\leq32\pi^2\sum_{k=1}^\infty\int_{\Omega_0^{EL}}\frac{\vert\varphi'(E)\vert}{T(E,L)}\frac1{\frac{4\pi^2}{T(E,L)^2}k^2-\lambda}\int_{r_-(E,L)}^{r_+(E,L)}\frac{\sin^2(2\pi k\theta(r,E,L))}{r^2}\diff r\diff(E,L)\\
		&\leq \frac{32\pi^2}{\inf_{\mathring\Omega_0^{EL}}(T)}\sum_{k=1}^\infty\int_{\Omega_0^{EL}}\frac{\vert\varphi'(E)\vert}{\frac{4\pi^2}{T(E,L)^2}k^2-\lambda}\int_{r_-(E,L)}^{r_+(E,L)}\frac{\diff r}{r^2} \diff(E,L)\leq C\int_{\Omega_0^{EL}} \frac{\vert\varphi'(E)\vert}{r_-(E,L)}\diff(E,L) 
	\end{align*}
	for some constant $C>0$ depending on the steady state and $\lambda$; note that $T$ is bounded away from $0$ by Proposition~\ref{Tbounded} and that the distance between $\lambda$ and the spectrum of $-\D^2$ is positive. To show that the latter integral is finite, recall that $r_-(E,L)$ is defined in Lemma~\ref{effpot} as a solution of 
	\begin{align*}
		E=\Psi_L(r_-(E,L))=U_0(r_-(E,L))+\frac L{2r_-^2(E,L)},
	\end{align*}
	which implies
	\begin{align*}
		\frac1{r_-(E,L)}=\frac{\sqrt{2E-2U_0(r_-(E,L))}}{\sqrt L}\leq\sqrt2\frac{\sqrt{E_0-U_0(0)}}{\sqrt L}.
	\end{align*}
	Thus,
	\begin{align*}
		\int_{\Omega_0^{EL}} \frac{\vert\varphi'(E)\vert}{r_-(E,L)}\diff(E,L)\leq \sqrt{2E_0-2U_0(0)}\int_0^{L_{max}}\frac{\diff L}{\sqrt L}\;\int_{U_0(0)}^{E_0} \vert\varphi'(E)\vert\diff E <\infty,
	\end{align*}
	where $L_{max}\in]0,\infty[$ is the maximal $L$-value in the steady state support and the finiteness of the latter integral follows since $k>0$ in the polytropic case. We therefore conclude that the first integral in~\eqref{eq:mathurkernelL2} is indeed finite. 
\end{proof}

In the next lemma we show that $\mathcal M_\lambda$ ($\p\M_\lambda$) map $\mathcal F$ ($\p{\mathcal F}$) into the smaller spaces $\mathcal F_1$ ($\p{\mathcal F}_1$).
\begin{lemma}\label{mathurincint}
	$\mathcal M_\lambda(\mathcal F)\subset\mathcal F_1$ and $\p{\mathcal M}_\lambda(\p{\mathcal F})\subset\p{\mathcal F}_1$.
\end{lemma}
\begin{proof}
	Let $F\in\mathcal M_\lambda(\mathcal F)$, i.e., there exists $G\in\mathcal F$ such that $\mathcal M_\lambda G=F$. Define the spherically symmetric functions $f,g\colon\Omega_0\to\R$ by
	\begin{align*}
		f(r,w,L)=\vert\varphi'(E,L)\vert\,w\,F(r),\quad g(r,w,L)=\vert\varphi'(E,L)\vert\,w\,G(r),\quad(r,w,L)\in\Omega_0^r.
	\end{align*}
	Lemma~\ref{Wintegrability} yields $f,g\in\Ltwo^{odd}$, and $Q_\lambda g=f$ by the definition of $\mathcal M_\lambda$, see Definition \ref{defmathur}. In other words,
	\begin{align*}
		f(r,w,L)=\B h(r,w,L) = 4\pi^2\vert\varphi'(E,L)\vert\,\frac w{r^2}\int_0^\infty\int_\R\tilde w\,h(r,\tilde w,\tilde L)\diff\tilde w\diff\tilde L,\quad(r,w,L)\in\Omega_0^r,
	\end{align*}
	where $h\coloneqq R_{-\D^2}(\lambda)g\in\mathrm D(\D^2)\cap\Ltwo^{odd}$. Using the representation of $F$ provided by the formula above, we arrive at
	\begin{align*}
		\|F\|_{\mathcal F_1}^2&=\int_Sr^2\,F^2(r)\diff r=16\pi^4\int_S\frac1{r^2}\,\left(\int_0^\infty\int_{\R}w\,h(r,w,L)\diff w\diff L\right)^2\diff r\\
		&\leq16\pi^4\int_S\frac1{r^2}\,\left(\int_0^\infty\int_\R w^2\,\vert\varphi'(E,L)\vert\diff w\diff L\right)\left(\int_0^\infty\int_\R \frac{h^2(r,w,L)}{\vert\varphi'(E,L)\vert} \diff w\diff L\right)\diff r\\&\leq C\|h\|_\Ltwo^2,
	\end{align*}
	where we used \eqref{eq:wintrho} and the boundedness of $\rho_0$ in the last inequality. In the planar setting the proof is analogous; apply~\eqref{eq:v1intrho} instead of~\eqref{eq:wintrho} for the last step.
\end{proof}

We are now in the position to show the equivalence of the eigenvalues of $Q_\lambda$ ($\pQ_\lambda$) and the ones of the Mathur operator ${\mathcal M}_\lambda$ ($\p{\mathcal M}_\lambda$):

\begin{lemma}\label{spectraequivalent}
	Let $\mu\in\R\setminus\{0\}$.
	\begin{enumerate}[label=(\alph*)]
		\item $\mu$ is an eigenvalue of $Q_\lambda$ if and only if $\mu$ is an eigenvalue of $\M_\lambda$.	
		Here, $\M_\lambda$ can be seen as an operator $\mathcal F\to\mathcal F$ or $\mathcal F_1\to\mathcal F_1$.
		\item $\mu$ is an eigenvalue of $\pQ_\lambda$ if and only if $\mu$ is an eigenvalue of $\p\M_\lambda$,  where $\p\M_\lambda$ can be seen as an operator $\p{\mathcal F}\to\p{\mathcal F}$ or $\p{\mathcal F}_1\to\p{\mathcal F}_1$.
	\end{enumerate}
\end{lemma}
\begin{proof}
	We only prove part (a), similar arguments apply in the planar setting. 
	
	If there exists $g\in \Ltwo^{odd}$ with $Q_\lambda g= \mu g$, then $g\in \mathrm{im}(\B)$ since $\mu\neq0$. Thus, there exists $G\colon S\to\R$ such that $g$ is of the form $g(r,w,L)=\vert\varphi'(E,L)\vert\,w\,G(r)$, and $G\in\mathcal F$ by Lemma~\ref{Wintegrability}. Using Definition~\ref{defmathur}, the eigenvalue equation becomes
	\begin{align}\label{eq:Klambdaev}
		\vert\varphi'(E,L)\vert\,w\,\M_\lambda G(r) = \mu\,\vert\varphi'(E,L)\vert\,w\,G(r),\qquad(r,w,L)\in\Omega_0^r.
	\end{align}
	Hence, $\M_\lambda G=\mu G$. In particular, $G\in\mathcal F_1$ by Lemma~\ref{mathurincint}.
	
	Conversely, let $G\in\mathcal F$ (recall $\mathcal F_1\subset\mathcal F$) be such that $\M_\lambda G=\mu G$ and define $g(r,w,L)=\vert\varphi'(E,L)\vert\,w\,G(r)$. Then $g\in \Ltwo^{odd}$ and \eqref{eq:Klambdaev} holds true. 
\end{proof}

Since $\M_\lambda\colon\mathcal F_1\to\mathcal F_1$ ($\p\M_\lambda\colon\p{\mathcal F}_1\to\p{\mathcal F}_1$) is a symmetric and compact Hilbert-Schmidt operator by Lemma~\ref{mathurF1}, the spectrum of $\M_\lambda$ is real and its largest element is given by
\begin{align}\label{eq:mathurlargestev}
	M_\lambda\coloneqq 
	\sup\left\{ \langle G,\M_\lambda G\rangle_{\mathcal F_1} \mid G\in\mathcal F_1,\,\|G\|_{\mathcal F_1}=1\right\}.
\end{align}
Similarly, the largest element in the spectrum of $\p\M_\lambda$ is
\begin{align}\label{eq:mathurlargestevplanar}
	\p M_\lambda \coloneqq \sup\left\{ \langle G,\p\M_\lambda G\rangle_{\p{\mathcal F}_1} \mid G\in\p{\mathcal F}_1,\,\|G\|_{\p{\mathcal F}_1}=1\right\}.
\end{align}
Furthermore, if $M_\lambda\neq0$ ($\p M_\lambda\neq0$), then $M_\lambda$ ($\p M_\lambda$) is actually an eigenvalue of $\M_\lambda$ ($\p\M_\lambda$).
\begin{theorem}[Criterion for the existence of oscillating modes]\label{T:CRITERION}
	\begin{enumerate}[label=(\alph*)]
		\item Let $f_0$ be a radial steady state of the form~\eqref{E:POLYTROPE} or~\eqref{E:KING}. Then the linearized operator $\A$ possesses an eigenvalue in the principal gap $\mathcal G$---defined in~\eqref{eq:defprincgap}---with associated eigenfunction odd in $v$ if and only if there exists a $\l\in \mathcal G$ such that 
		\begin{align}\label{E:CRITERIONONE}
			M_\lambda\geq1.
		\end{align}
		\item Let $\pf_0$ be a planar steady state as specified in Section~\ref{ssc:ststp}. Then the linearized operator $\pA$ possesses an eigenvalue in the principal gap $\p{\mathcal G}$---defined in \eqref{eq:defprincgapplanar}---with associated eigenfunction odd in $v_1$ and $x$ if and only if there exists a $\l\in \p{\mathcal G}$ such that 
		\begin{align}\label{E:CRITERIONTWO}
			\p M_\lambda\geq1.
		\end{align}
	\end{enumerate}
\end{theorem}
\begin{proof}[Proof of part (a)]
	If $\lambda\in\mathcal G$ is an eigenvalue of $\A$ (restricted to $H^{odd}$, i.e., functions in the spherically symmetric, weighted $L^2$-space $\Ltwo$ which are odd in~$v$), then $1$ is an eigenvalue of $Q_\lambda$ by Lemma~\ref{L:BS}. Thus, Lemma~\ref{spectraequivalent} yields that $1$ is also an eigenvalue of $\mathcal M_\lambda\colon\F_1\to\F_1$, which implies $M_\lambda\geq1$.
	
	Conversely, if there exists $\lambda\in\mathcal G$ such that $M_\lambda\geq1$, then $\M_\lambda$ has an eigenvalue $\mu\geq1$. Lemma~\ref{spectraequivalent} implies that $\mu$ is also an eigenvalue of $Q_\lambda$, and therefore $\lambda$ is an eigenvalue of $-\D^2-\frac1\mu\B$ by Lemma~\ref{L:BS}. Using Lemma~\ref{L:BSONE}, we conclude that $\A$ indeed has an eigenvalue in the principal gap $\mathcal G$.
	
	\noindent\textit{Proof of part (b).} The proof is analogous to the radial setting and uses the planar statements of Lemmas~\ref{L:BS}, \ref{L:BSONE}, and~\ref{spectraequivalent}. However, recall that we restricted $\pA$ to $\p\H$, i.e., functions odd in $v_1$ and $x$, in the planar setting, which causes the top of the principal gap $\p{\mathcal G}$ to quadruple, see \eqref{eq:defprincgapplanar} and also Theorem~\ref{T:ESSENTIALSPECTRUMABAR}.
\end{proof}

Before verifying these criteria for selected steady states,
we state some properties of $M_\lambda$ and $\p M_\lambda$.

\begin{remark}\label{criterionexplicit}
  Using the kernel representation of the radial Mathur operator provided
  by Proposition~\ref{Klambdadef}, the associated quadratic form can be rewritten as
  \begin{align}
    &\langle G,\M_\lambda G\rangle_{\mathcal F_1}
    = \int_S\int_Sr^2\,K_\lambda(r,\sigma)\,G(r)\,G(\sigma)\diff\sigma\diff r\nonumber\\
    &=32\pi^2\sum_{k=1}^\infty
    \int_{\Omega_0^{EL}}\frac{\vert\varphi'(E,L)\vert}{T(E,L)}
    \frac1{\frac{4\pi^2k^2}{T(E,L)^2}-\lambda}
    \left(\int_{r_-(E,L)}^{r_+(E,L)}\sin(2\pi k\theta(r,E,L))\,G(r)\diff r\right)^2 \diff(E,L)
    \label{eq:mathurquadform}
  \end{align}
  for $G\in\mathcal F_1$. Since $\lambda$ is in the principal gap, the latter integral is
  obviously non-negative, i.e., the spectrum of the Mathur operator is non-negative as well.
  Hence $M_\lambda$ coincides with the operator norm of
  $\M_\lambda$, cf.~\cite[Theorem~5.14]{HiSi}.
  Furthermore, $M_\lambda$ increases in $\lambda$,
  which can for example be seen in the integral \eqref{eq:mathurquadform}. 
  Theorem~\ref{T:CRITERION} now allows us to translate the existence of
  an oscillating mode corresponding to an eigenvalue of $\A$ in the principal gap into a condition
  on the size of
  \begin{align}\label{eq:mathursupsup}
  	M\coloneqq\sup_{\lambda\in\mathcal G}M_\lambda=\sup_{\lambda\in\mathcal G}\|\M_\lambda\|_{\F_1\to\F_1}:
  \end{align}
  \begin{enumerate}[label=(\alph*)]
  \item If $M>1$, then $\A$ possesses at least one eigenvalue in the principal gap.
  \item If $M<1$, then $\A$ has no eigenvalues in the principal gap.
  \item In the case $M=1$, the existence of an eigenvalue of $\A$ in the
    principal gap depends on whether or not the supremum \eqref{eq:mathursupsup} is attained.
  \end{enumerate}
	
  Similar statements hold true in the plane symmetric setting.
  In particular, the spectrum of $\p\M_\lambda\colon\p\F_1\to\p\F_1$ is non-negative,
  and 
  \begin{align}
  	\p M\coloneqq\sup_{\lambda\in\p{\mathcal G}}\p M_\lambda=\sup_{\lambda\in\p{\mathcal G}}\|\p\M_\lambda\|_{\p\F_1\to\p\F_1}
  \end{align}
  being larger (smaller) than $1$ implies the existence one (no) eigenvalue(s) of $\pA$
  in the principal gap. 
\end{remark}

\subsection{Examples of linear oscillations}\label{ssc:examples}

We now apply Theorem~\ref{T:CRITERION} to give a class of examples of steady states which allow for pulsating modes. A particularly simple-minded approach is to identify steady states for which $M_\lambda$ ($\p M_\lambda$) tends to infinity as $\lambda$ approaches the top of the principal gap.

\subsubsection{Linear oscillations in the planar case}

In the plane symmetric setting we are able to analytically show the existence of linearly pulsating modes for a large class of steady state models by pursuing the approach discussed above. 


\begin{theorem}\label{T:EXAMPLEPLANAR}
Let $\pf_0$ be a planar steady state of polytropic form~\eqref{eq:poly1d} with $\frac12<k\leq1$ or of King type~\eqref{eq:king1d}. Then the associated linearized operator $\pA$---restricted to $\p\H$, i.e., functions in $\pH$ which are odd in $v_1$ and $x$---possesses an eigenvalue in the principal gap $\p{\mathcal G}=\left]0,4\frac{4\pi^2}{\pT^2(\pE_0)}\right[$.
\end{theorem}


\begin{proof}
For $G\in\p{\mathcal F}_1$, the quadratic form in \eqref{eq:mathurlargestevplanar} is
\begin{align}
	&\langle G,\p{\M}_\lambda G\rangle_{\p{\mathcal F}_1} = \int_{\pS} \int_{\pS} \pK_\lambda(x,y)\,G(x)\,G(y)\diff x\diff y \nonumber \\
	&=32\pi \sum_{k=1}^\infty \int_{\pS}\int_{\pS}  \int_{\pOmega_0^{\pE}(x)\cap\pOmega_0^{\pE}(y)}\frac{\vert\palpha'(\pE)\vert}{\pT(\pE)}\frac{\sin(4\pi k\theta(x,\pE))
	\sin(4\pi k\theta(y,\pE))}{\frac{4\pi^2}{\pT(\pE)^2}(2k)^2-\lambda}\diff\pE\,G(x)\,G(y)\diff x \diff y \nonumber \\
	&=32\pi\sum_{k=1}^\infty  \int_{\pUmin}^{\pE_0}\frac{\vert\palpha'(\pE)\vert}{\pT(\pE)}\frac1{\frac{4\pi^2}{\pT(\pE)^2}(2k)^2-\lambda} \left( \int_{x_-(\pE)}^{x_+(\pE)} \sin(4\pi k\theta(x,\pE))\,G(x)\diff x \right)^2 \diff \pE,
\end{align}
see Proposition~\ref{Klambdadefplane} for the definition of $\pK_\lambda$. Note that the exchange of the infinite sum and the integration can be justified similarly to the radial case, see for example a related argument in the proof of Proposition~\ref{Klambdadef}.
In particular, for any $\lambda$ in the principal gap, i.e., $0<\lambda<4\frac{4\pi^2}{\pT^2(\pE_0)}$,
we conclude that for all non-zero $G\in\p{\mathcal F}_1$, 
\begin{align}
  \p M_\lambda&\geq\frac{\langle G,\p\M_\lambda G\rangle_{\p{\mathcal F}_1} }{\|G\|_{\p{\mathcal F}_1}^2}
  \nonumber \\
  &\geq\frac{32\pi}{\|G\|_{\p{\mathcal F}_1}^2}\int_{\pUmin}^{\pE_0}\frac{\vert\palpha'(\pE)\vert}{\pT(\pE)}\frac1{\frac{16\pi^2}{\pT(\pE)^2}-\lambda}\left( \int_{x_-(\pE)}^{x_+(\pE)} \sin(4\pi \theta(x,\pE))\,G(x)\diff x\right)^2\diff\pE.\label{eq:Glambdaestimatebar}
	\end{align}
Consider a neighborhood $\pN_\eta \coloneqq[\pE_0-\eta,\pE_0]$ of the cut-off energy $\pE_0$
for some sufficiently small parameter $0<\eta<\pE_0-\pUmin$.
By an easy continuity argument we can choose a closed interval
$\pI\subset]x_-(\pE_0),0[=]-\pR_0,0[$ such that
\begin{align}\label{E:SINLOWERBOUND}
  \sin (4\pi\theta(x,\pE)) \geq\frac12 \quad\text{ for } \pE\in \pN_\eta, \ x\in \pI,
\end{align} 
if $\eta>0$ is sufficiently small, in particular,
$\pI\subset]x_-(\pE),0[$ for $\pE\in\pN_\eta$. More precisely, $\pI$ can be constructed as follows:
Let $x_{1/4}(\pE)$ be defined by
\begin{align*}
	\theta( x_{1/4}(\pE),\pE )=\frac18,\quad\text{i.e.,}\quad x_{1/4}(\pE)=X(\frac18\pT(\pE),x_-(\pE),0),
\end{align*}
where $X$ gives the $x$-component of the solutions of the stationary planar characteristic system \eqref{eq:charsystplanar}. Obviously, $x_{1/4}$ is continuous, and $\pI$ can be chosen as a neighborhood of $x_{1/4}(\pE_0)$ after possibly reducing $\eta$. Now let
\begin{align}
	G\colon\pS\to\R,\quad G(x)\coloneqq\begin{cases}
		1, & x\in \pI, \\
		-1, & -x\in \pI,\\
		0,&\text{else}.
	\end{cases}
\end{align}
Clearly, $G$ is odd and $G\in\p\F_1$ with $\|G\|_{\p{\mathcal F}_1}^2 = 2|\pI|$. Moreover,
\begin{align*}
	\int_{x_-(\pE)}^{x_+(\pE)}\sin(4\pi\theta(x,\pE))\,G(x)\diff x\geq\vert\pI\vert
\end{align*}
for any $\pE\in\pN_\eta$; observe that the symmetry of $\pU_0$ implies $4\pi\theta([x_-(\pE),0],\pE)=[0,\pi]$ and $4\pi\theta([0,x_+(\pE)],\pE)=[\pi,2\pi]$ as well as $\theta(x,\pE)+\theta(-x,\pE)=\frac12$ for $x\in[x_-(\pE),0]$, recall the definition \eqref{E:THETAPLANARDEF} of $\theta$.
Plugging $G$ into~\eqref{eq:Glambdaestimatebar} yields
\begin{align*}
	\p M_\lambda\geq16\pi\vert\pI\vert\int_{\pN_\eta}\frac{\vert\palpha'(\pE)\vert}{\pT(\pE)}\frac1{\frac{16\pi^2}{\pT(\pE)^2}-\lambda}\diff\pE.
\end{align*}
We now let $\lambda\to\sup\p{\mathcal G}=\frac{16\pi^2}{\pT(\pE_0)^2}$---notice that $\pI$
is independent of $\lambda$---and conclude by the monotone convergence theorem that
\begin{align}
	\limsup_{\l\to \frac{16\pi^2}{\pT(\pE_0)^2}} \p M_\lambda&\geq16\pi\vert\pI\vert\int_{\pN_\eta}\frac{\vert\palpha'(\pE)\vert}{\pT(\pE)}\frac1{\frac{16\pi^2}{\pT(\pE)^2}-{\frac{16\pi^2}{\pT(\pE_0)^2}}} \diff\pE\nonumber\\
	&=\frac{\vert\pI\vert}\pi\int_{\pN_\eta}\frac{\vert\palpha'(\pE)\vert}{\pT(\pE_0)-\pT(\pE)}\,\frac{\pT(\pE)\,\pT(\pE_0)^2}{\pT(\pE)+\pT(\pE_0)}\diff\pE\nonumber\\
	&\geq C\int_{\pN_\eta}\frac{\vert\palpha'(\pE)\vert}{\pT(\pE_0)-\pT(\pE)}\diff\pE\label{eq:mathursupsupestimate}
\end{align}
for some $C>0$; recall that $\pT$ is bounded and bounded away from $0$ on
$[\pUmin,\pE_0[$ by Proposition~\ref{Tboundsplane}.
To show that the latter integral is infinite, we expand the denominator of its integrand.
By the mean value theorem, for each $\pE\in\pN_\eta$ there exists $\pE^*\in\pN_\eta$ such that
$\pT(\pE_0)-\pT(\pE)=\pT'(\pE^*)\,(\pE_0-\pE)$.
Since $\pT'$ is continuous by Lemma~\ref{Tderplane}, 
we obtain $\pT(\pE_0)-\pT(\pE)\leq C\,(\pE_0-\pE)$ for $\pE\in\pN_\eta$ and some constant $C>0$ depending on the steady state and $\eta$. Inserting this estimate in \eqref{eq:mathursupsupestimate} then yields
\begin{align*}
	\limsup_{\l\to \frac{16\pi^2}{\pT(\pE_0)^2}} \p M_\lambda&\geq C\int_{\pN_\eta}\frac{\vert\palpha'(\pE)\vert}{\pE_0-\pE}\diff\pE\\
	&=C\begin{cases}k\int_{\pE_0-\eta}^{\pE_0}\left(\pE_0-\pE\right)^{k-2}\diff\pE,&\palpha \text{ is polytropic with }\frac12<k\leq1\\
		\int_{\pE_0-\eta}^{\pE_0}\frac{e^{\pE_0-\pE}}{\pE_0-\pE}\diff\pE,&\palpha\text{ is of King type}
	\end{cases}\\
	&=\infty.
\end{align*}
The claim now follows by part (b) of Theorem~\ref{T:CRITERION}.
\end{proof}

\begin{remark}
	\begin{enumerate}[label=(\alph*)]
		\item Since $\pT'>0$ on $]\pUmin,\infty[$, we can not expect to show that the integral \eqref{eq:mathursupsupestimate} is infinite for polytropes with exponent $k>1$ by expanding $\pT$ to higher order.
		\item For the above proof---and everything in the preceding sections---the exact polytropic~\eqref{eq:poly1d} or King type~\eqref{eq:king1d} structure of $\palpha$ is not essential. Only  the behavior of $\palpha$ near the cut-off energy $\pE_0$ matters for our argument. 
	\end{enumerate}	
\end{remark}




\subsubsection{Linear oscillations in the radial case}

In order to apply the same idea as in the proof of Theorem~\ref{T:EXAMPLEPLANAR} we need
to know where the period function $\mathring\Omega_0^{EL}\ni(E,L)\mapsto T(E,L)$ attains
its maximal values. This is an involved technical question about the behavior of solutions
to semi-linear radial ODEs of
Lane-Emden type, which will be rigorously addressed in future work.
However, numerical calculations conclusively show that for a wide range
of steady states of the form~\eqref{E:POLYTROPE}--\eqref{E:KING} the period function has
the property that
\begin{align}\label{E:TMAX}
  \sup_{\mathring\Omega_0^{EL}}(T) = T(E_0,L_0),
\end{align}
where $L_0\geq0$ is the lowest occurring $L$-value in the steady state support.
Note that $T(E_0,L_0)$ may formally not be defined by Definition~\ref{defT},
but $T$ can easily be extended to $L=0$ by replacing $\Psi_L$ with $U_0$ in the definition. 

In order to Taylor-expand $T$ near its maximal value similar to the planar setting,
we further require that
\begin{align}\label{E:Tdiffbar}
  T\text{ is differentiable and }\partial_ET,\,\partial_LT\text{ are bounded near }(E_0,L_0).
\end{align}
The latter can be shown by explicitly computing the derivatives of $T$ with respect to $E$ and $L$,
but we choose to leave out the proof of this rather technical statement. 

One way to validate~\eqref{E:TMAX} is to show
\begin{align}\label{E:TPROPERTIES}
	\pa_E T(E,L)& \geq0, \ \ 
	\pa_L T(E,L) \leq0, \  \quad\text{ for } (E,L)\in\mathring\Omega_0^{EL},
\end{align}
and numerical computations indicate that \eqref{E:TPROPERTIES} is indeed true for a wide
range of steady states.
We discuss this matter in more detail at the end of Section~\ref{ssc:reg} in the appendix,
and give explicit parameters of a steady state for which~\eqref{E:TMAX} and~\eqref{E:TPROPERTIES}
have been verified numerically in Remark~\ref{radialparameterexample}.

\begin{theorem}\label{T:EXAMPLERADIAL}
Let $f_0$ be a radial steady state of polytropic form~\eqref{E:POLYTROPE} with parameters $k,l,L_0$ satisfying
\begin{align}\label{eq:critparameters}
	L_0>0,\ \ k>0,\ \ l>-1,\ \ k<l+\frac72,\ \ k+l+\frac12\geq0,\ \ k+l\leq0.
\end{align} 	
Assume further that the assumptions~\eqref{E:TMAX} and~\eqref{E:Tdiffbar} hold. Then the associated linearized operator $\A$---restricted to $\Ltwo^{odd}$, i.e., functions in $\Ltwo$ which are odd in $v$---possesses an eigenvalue in the principal gap $\mathcal G=\left]0,\frac{4\pi^2}{\sup^2(T)}\right[$. 
\end{theorem}


\begin{proof}
Just like in the proof of Theorem~\ref{T:EXAMPLEPLANAR} we can show that 
\begin{align}\label{E:MLAMBDALOWERBOUND}
	& M_\lambda 
	\geq \frac{32\pi^2}{\|G\|_{\mathcal F_1}^2} \int_{\Omega_0^{EL}}\frac{\vert\varphi'(E,L)\vert}{T(E,L)}\frac1{\frac{4\pi^2}{T(E,L)^2}-\lambda} \left( \int_{r_-(E,L)}^{r_+(E,L)} \sin(2\pi k\theta(r,E,L))\,G(r)\diff r \right)^2 \diff(E,L)
	\end{align}
for any $G\in\mathcal F_1\setminus\{0\}$ and $\lambda\in\mathcal G$. Letting $\lambda\to\frac{4\pi^2}{\sup^2(T)}=\frac{4\pi^2}{T(E_0,L_0)^2}$, using the monotone convergence theorem, and the boundedness of $T$ from above and away from zero (see Proposition~\ref{Tbounded}), we conclude that
\begin{multline}
\limsup_{\l\to \frac{4\pi^2}{\sup^2(T)}}M_\lambda\nonumber\\ \geq \frac C{\|G\|_{\mathcal F_1}^2} \int_{\Omega_0^{EL}}  \frac{\vert\varphi'(E,L)\vert}{T(E_0,L_0)-T(E,L)}\left( \int_{r_-(E,L)}^{r_+(E,L)} \sin(2\pi\theta(r,E,L))\,G(r)\diff r \right)^2 \diff(E,L)
\end{multline}
for some constant $C>0$ depending on the steady state.

Consider a closed neighborhood $N_\eta$ of the $T$-maximizer $(E_0,L_0)$ of the form
\begin{align*}
  N_\eta \coloneqq [E_0-\eta,E_0]\times[L_0,L_0+\eta]
\end{align*}
for a sufficiently small $\eta>0$ such that $\mathring N_\eta\subset\Omega_0^{EL}$.
Next, choose a non-empty, closed interval $I\subset S$ such that
\begin{align}
	\sin (2\pi\theta(r,E,L)) \geq\frac12 \quad\text{ for } (E,L)\in N_\eta,~r\in I.
\end{align} 
To convince ourselves of the existence of such an interval $I$, let $r_{1/2}(E,L)>0$ be defined by
\begin{align}
  \theta( r_{1/2}(E,L),E,L )=\frac14, \quad\text{i.e.,}\quad r_{1/2}(E,L)=R(\frac14T(E,L),r_-(E,L),0,L),
\end{align} 
where $R$ is the $r$-component of the solutions of the stationary radial characteristic
system~\eqref{eq:charsystrad};
note that $r_{1/2}(E,L)$ need not coincide with $r_L$.
Since $T$ is continuous by Lemma~\ref{partperiodcontinuous}, $r_{1/2}$ is continuous in $(E,L)$, and $I$ can be chosen as a sufficiently small interval around $r_{1/2}(E_0,L_0)>0$
(after possibly reducing $\eta$). 
Now let $ G\coloneqq\id_I \in \mathcal F_1$. 
Analogously to~\eqref{eq:mathursupsupestimate} we conclude that 
\begin{align}\label{E:LBMLAMBDA}
 \limsup_{\l\to \frac{4\pi^2}{T(E_0,L_0)^2}} M_\lambda
 &\ge  C \int_{E_0-\eta}^{E_0}\int_{L_0}^{L_0+\eta}\frac{\vert\varphi'(E,L)\vert}{T(E_0,L_0)-T(E,L)} \diff L \diff E.
\end{align}
Taylor-expanding the denominator in this integral gives
\begin{align}
T(E_0,L_0)-T(E,L) 
= &c_E\,(E_0-E) - c_L\,(L-L_0)\nonumber\\& + o_{(E,L)\to(E_0,L_0)}\left(\vert(E,L)-(E_0,L_0)\vert\right),
\end{align}
where we denote $c_E\coloneqq\partial_ET(E_0,L_0),\ c_L\coloneqq\partial_LT(E_0,L_0)<\infty$;
recall the assumption~\eqref{E:Tdiffbar}.
Thus, after choosing a possibly smaller $\eta$, we use \eqref{E:LBMLAMBDA}
and the polytropic structure \eqref{E:POLYTROPE} of the steady state model to obtain
\begin{align}
 \limsup_{\l\to \frac{4\pi^2}{T(E_0,L_0)^2}} M_\lambda
 & \ge C  \int_{E_0-\eta}^{E_0}\int_{L_0}^{L_0+\eta} \frac{(E_0-E)^{k-1}(L-L_0)^l}{E_0-E + L-L_0}
 \diff L\diff E \notag \\
& = C \int_0^\eta\int_0^\eta \frac{x^{k-1}y^l}{x+y} \diff y\diff x,
\end{align}
where we used the obvious change of variables $x=E_0-E$ and $y=L-L_0$.
The latter integral is infinite precisely when $k+l\leq0$.
Together with part (a) of Theorem~\ref{T:CRITERION} this concludes the proof.
\end{proof}

\begin{remark}\label{radialparameterexample}
  The steady state parameter conditions \eqref{eq:critparameters} are for example
  satisfied with
  \begin{align}
    k=\frac12,\quad l=-\frac12,\quad L_0>0\text{ arbitrary},
  \end{align}
  and numerical computations---e.g.\ with $L_0=\frac1{10}$ and $E_0-U_0(0)=1$---clearly
  show that in this case the monotonicity assumptions~\eqref{E:TPROPERTIES} are valid.
\end{remark}

\appendix
\section{Auxiliary results on potentials}\label{sc:potential}
In this section we discuss the properties of potentials induced by functions of the form $\D g$
for $g\in\mathrm D(\D)$ and $\pD g$ for $g\in\mathrm D(\pD)$ respectively.

\subsection{Potentials in the spherically symmetric case}\label{ssc:potradial}

For $g\in\mathrm D(\D)$, let
\begin{align*}
\rho = \rho_{\D g} \coloneqq \int_{\R^3} \D g(\cdot,v)\diff v.
\end{align*}
We extend all functions by $0$ on $\R^3\times\R^3$. Using \eqref{eq:A7prime} yields
\begin{align*}
\|\rho\|_2^2&= \int_{\R^3}\left(\int_{\R^3} \D g(x,v)\diff v\right)^2\diff x\leq C\int_{\R^3}\int_{\R^3} \frac{\vert \D g(x,v)\vert^2}{\vert\varphi'(E,L)\vert}\diff v\diff x= C\|\D g\|_{\Ltwo}^2,
\end{align*}
and $\supp(\rho)\subset B_{R_0}(0)$, i.e., $\rho\in L^1\cap L^2(\R^3)$. 
Furthermore, the integral of $\rho$ vanishes, since
\begin{align*}
	\int_{\R^3}\rho(x)\diff x=\langle \vert\varphi'(E,L)\vert,\D g\rangle_{\Ltwo}
\end{align*}
and $\vert\varphi'(E,L)\vert\in\ker(\D)\subset\Ltwo$ by \eqref{eq:A7prime} as well as
$\D g\in\mathrm{im}(\D)\perp\ker(\D)$ by the skew-adjointness of $\D$.

Now let $U\coloneqq U_\rho=U_{\D g}$ be the gravitational potential induced by $\D g$, i.e., 
\begin{align*}
U(x) = -\int_{\R^3}\frac{\rho(y)}{\vert x-y\vert} \diff y ,\quad x\in\R^3.
\end{align*}
Since $\rho\in L^1\cap L^2(\R^3)$, basic potential theory yields $U \in C(\R^3)$ with $\lim_{\vert x\vert\to\infty} U(x)=0$ and $\|U\|_\infty \leq C\|\rho\|_2$ as well as $\partial_x U\in L^2(\R^3)$ with
\begin{align*}
\| \partial_x U\|_2\leq C\|\rho\|_{\frac65}\leq C\|\rho\|_2,
\end{align*}
and obviously $\Delta U = 4\pi\rho\in L^2(\R^3)$. In particular, all derivatives exist in the
weak sense. Furthermore, for all $x\in\R^3$ with $\vert x\vert\geq2R_0$ the vanishing integral
of $\rho$ implies that
\begin{align*}
\vert U(x)\vert &= \left| \int_{\R^3}\frac{\rho(y)}{\vert x-y\vert} \diff y  \right| = \left| \int_{B_{R_0}(0)}\frac{\rho(y)}{\vert x-y\vert} \diff y - \int_{B_{R_0}(0)}\frac{\rho(y)}{\vert x\vert} \diff y \right| \\
&\leq \int_{B_{R_0}(0)} \vert\rho(y)\vert \frac{\left| \vert x\vert - \vert x-y\vert \right|}{\vert x\vert \cdot \vert x-y\vert} \diff y \leq \int_{B_{R_0}(0)} \vert\rho(y)\vert \frac{2R_0}{\vert x\vert^2} \diff y= \frac{2 R_0\|\rho\|_1}{\vert x\vert^2}.
\end{align*}
Thus,
\begin{align*}
\int_{\vert x\vert\geq 2R_0} \vert U(x)\vert^2\diff x\leq 4 R_0^2\|\rho\|_1^2 \int_{\vert x\vert\geq 2R_0} \frac{\diff x}{\vert x\vert^4} = 8 \pi R_0\|\rho\|_1^2 \leq C \|\rho\|_2^2.
\end{align*}
Overall we obtain $U\in H^2(\R^3)$ with
\begin{align} \label{eq:poth2}
\| U\|_{H^2(\R^3)}\leq C\|\rho\|_2 \leq C\|\D g\|_{\Ltwo},
\end{align}
where $C>0$ only depends on the fixed steady state $f_0$. Related arguments have also been used in the proof of Theorem~1.1 in \cite{GuLi08}. 

Lastly, $U$ inherits the symmetry of $\rho$, $\D g$ and $g$, i.e., we can write $U(x)=U(\vert x\vert)$. Then, 
\begin{align} \label{eq:uprime}
U_{\D g}'(r) = U'(r) = \frac{4\pi^2}{r^2} \J(g)(r) = \frac{4\pi^2}{r^2} \int_0^\infty\int_\R w\, g(r,w,L)\diff w\diff L 
\end{align}
for a.e.~$r>0$. For $g\in C^{\infty}_{c,r}(\Omega_0)$ this follows by integrating the radial Poisson equation, since
\begin{align*}
\rho(r) = \frac\pi{r^2} \int_0^\infty\int_\R\D g(r,w,L)\diff w\diff L = \frac\pi{r^2}\;\partial_r\left( \int_0^\infty\int_\R w\,g(r,w,L)\diff w\diff L \right).
\end{align*}
\eqref{eq:uprime} can be extended to $\mathrm D(\D)$ by using the approximation scheme from \cite{ReSt20,St19} together with \eqref{eq:wintrho}.

\subsection{Potentials in the plane symmetric case}\label{ssc:potplane}

In the plane symmetric case, we let $g\in \mathrm D(\pD)$ and
\begin{align*}
\prho(x)=\prho_{\pD g}(x)\coloneqq \int_{\R^3} \pD g(x,v)\diff v,\quad x\in\R,
\end{align*}
where we again extend all functions by $0$ to $\R\times\R^3$.
By \eqref{eq:A7prime1d}, $\|\prho\|_2\leq C\|\pD g\|_{\pH}$
with some constant $C>0$ depending only on the steady state.
Obviously, $\mathrm{supp}(\prho)\subset[-\pR_0,\pR_0]$, i.e.,
$\prho\in L^1\cap L^2(\R)$.
As in the spherically symmetric case,
$\int_{\R}\prho=\langle\vert\pvarphi'(\pE,\pv)\vert,\pD g\rangle_{\pH} =0$
by the skew-adjointness of $\pD$.

Now let $\pU=\pU_{\prho}=\pU_{\pD g}$ be the potential induced by $\pD g$,
cf.~\eqref{pl_poisson}.
Since $\prho$ vanishes outside of $[-\pR_0,\pR_0]$ and $\int_\R \rho =0$,
\eqref{pl_dxU} implies that
\[
U'(x) = 2\pi \int_{-R_0}^{R_0} \sign(x-y)\, \rho(y) \diff y
= 4\pi \int_{-R_0}^{x} \rho(y) \diff y,
\]
and $\pU'(x)=0$ for
$x\in \R\setminus[-R_0,R_0]$. Hence
\begin{align*}
  \|\pU'\|_2^2=\int_{-\pR_0}^{\pR_0} \vert\pU'(x)\vert^2\diff x
  \leq 16\pi^2\pR_0^2 \, \|\prho\|_2^2.
\end{align*}
Altogether,
\begin{align}\label{eq:poth2plane}
\|\pU'\|_2+\|\pU''\|_2\leq C\|\prho\|_2\leq C\|\pD g\|_{\pH}^2.
\end{align}
Lastly, we obtain the representation
\begin{align}\label{eq:uprimeplane}
  \pU_{\pD g}'(x) = 4\pi\int_{-\pR_0}^x \int_{\R^3}\pD g(y,v)\diff v\diff y
  = 4\pi \int_{\R^3} v_1\,g(x,v)\diff v,\quad x\in\R
\end{align}
for smooth $g$, which once again can be extended onto $\mathrm D(\pD)$
by approximation, for example as in Remark \ref{approxplane}, using the
bound \eqref{eq:v1intrho}.

\section{Properties of the radial period function}\label{sc:T}

This appendix is devoted to the properties of the period function $T$ in the case
of a spherically symmetric equilibrium, since it gives the essential spectrum of
the linearized operator $\A$ and is a crucial quantity for the existence of
oscillating modes. Recall that
\begin{align*}
T(E,L) = 2 \int_{r_-(E,L)}^{r_+(E,L)} \frac{\diff r}{\sqrt{2E-2\Psi_L(r)}}
\end{align*}
for $L>0$ and $\Psi_L(r_L) < E < 0$. In particular, $T(E,L)$ is
defined for $(E,L)\in\mathring\Omega_0^{EL}$.
In the two upcoming sections we prove the following result which is e.g.~needed
in Section~\ref{ssc:essradial}:
\begin{prop}\label{Tbounded}
  There exist $c,C>0$ such that $c\leq T(E,L)\leq C$ for $(E,L)\in\mathring\Omega_0^{EL}$.
\end{prop}
Afterwards we discuss the regularity and potential monotonicity of the period function. 

\subsection{An upper bound on $T$}\label{ssc:Tupper}

First we recall the bound \eqref{eq:Tboundnaive}, i.e.,
\begin{align*}
T(E,L) \leq 2\pi \frac{\|f_0\|_1^2}{E^2\sqrt L}
\end{align*}
for $L>0$ and $\Psi_L(r_L)<E<0$. This estimate has earlier been used in
\cite{ReSt20,St19}, where the reader may also find a rather straight-forward proof.
In particular, it shows that in the case of a polytropic shell steady state, i.e.,
the ansatz function is of the form \eqref{E:POLYTROPE} with $L_0>0$, $T$ is bounded
from above on the whole set $\mathring\Omega_0^{EL}$.

Unfortunately, the boundedness from above is harder to show in the case
where no inner vacuum region exists, for example in the case of isotropic steady states.
In order to handle these models we use the maximum principle to estimate $T$ as follows:

\begin{lemma}\label{Tupperboundmaxprinc}
	Consider an isotropic steady state, i.e., $\varphi$ is of the form \eqref{E:KING} or \eqref{E:POLYTROPE} with $L_0=0=l$, in particular $\supp(\rho_0)=[0,R_0]$ and $\rho_0(0)>0$. Suppose that there exist $S,c>0$ such that $\rho_0\geq c$ on $[0,S]\subset[0,R_0]$. Then 
	\begin{align}\label{eq:effpotmaxprinc}
	E-\Psi_L(r)\geq \frac{2\pi}3 c \frac{(r_+(E,L)- r)(r-r_-(E,L))(r+r_+(E,L)+r_-(E,L))}r
	\end{align}
	for $(E,L)\in\mathring\Omega_0^{EL}$ and $r>0$ such that $0<r_-(E,L)\leq r\leq r_+(E,L)\leq S$. This leads to
	\begin{align}
	T(E,L) \leq \sqrt{\frac{3\pi}c}
	\end{align}
	for $(E,L)\in\mathring\Omega_0^{EL}$ satisfying $r_+(E,L)\leq S$.
\end{lemma}
\begin{proof}
	Fix $(E,L)\in\mathring\Omega_0^{EL}$ with $r_+(E,L)\leq S$ and let
	\begin{align*}
	U_c(r)\coloneqq&- \frac{2\pi}3 c \frac{(r_+(E,L)- r)(r-r_-(E,L))(r+r_+(E,L)+r_-(E,L))}r  \nonumber \\
	=& \frac{2\pi}3 c \left( r^2 - \left(r_+-r_-\right)^2 + r_-r_+ + \frac1r r_-r_+(r_++r_+) \right)
	\end{align*}
	for $r\in[r_-(E,L),r_+(E,L)]$, where we used the abbreviation $r_\pm= r_\pm(E,L)$. Obviously, $U_c(r_\pm(E,L)) = 0= E-\Psi_L(r_\pm(E,L))$. Applying the (radial) Laplacian $\Delta = \left(\partial_r^2+\frac2r\partial_r\right)$ yields
	\begin{align*}
	\Delta U_c(r) &= 4\pi c,\\
	\Delta \Psi_L (r) &= \Delta U_0(r) + \frac L{r^4} = 4\pi\rho_0(r) + \frac L{r^4}
	\end{align*}
	for $r\in[r_-(E,L),r_+(E,L)]$. Thus, by the choice of $c$,
	\begin{align*}
	\Delta \left( U_c + E - \Psi_L \right) < 0 \text{ on } [r_-(E,L),r_+(E,L)].
	\end{align*}
	By the maximum principle 
	we therefore conclude that
	\begin{align*}
	U_c + E - \Psi_L > 0 \text{ on } ]r_-(E,L),r_+(E,L)[,
	\end{align*}
	which shows \eqref{eq:effpotmaxprinc}. Inserting this into the definition of $T$ yields
	\begin{align*}
	T(E,L) &= \sqrt2 \int_{r_-(E,L)}^{r_+(E,L)} \frac{\diff r}{\sqrt{E-\Psi_L(r)}} \\ 
	&\leq \sqrt{\frac3{c\pi}} \int_{r_-(E,L)}^{r_+(E,L)} \frac{\diff r}{\sqrt{(r_+(E,L)-r)(r-r_-(E,L))}} = \sqrt{\frac{3\pi}c}.\qedhere
	\end{align*}
\end{proof}

The bound \eqref{eq:Tboundnaive} shows that for any choice of $L_1>0$
the period function $T$ is bounded on the set $L\geq L_1$.
For any choice of $U_0(0)<E_1<E_0$
orbits corresponding to $(E,L)$ with $E\leq E_1$ are radially restricted
to some interval $[0,R_1]$ with $R_1<R_0$. For an isotropic steady
state, $\rho_0$ is bounded away from zero on such an interval, and
the previous lemma shows that the period function $T$ is
bounded on the set $ E \leq E_1$.
The next lemma closes the remaining gap:

\begin{lemma}\label{T_gap}
	Consider a steady state of the form \eqref{E:KING} or \eqref{E:POLYTROPE} with $L_0=0=l$, in particular $\supp(\rho_0)=[0,R_0]$. Then there exist $U_0(0) < E_1 < E_0$ and $L_1>0$ and a constant $C>0$ such that
	\[
	T(E,L) \leq C \text{ for } E_1\leq E \leq E_0,\ 0< L \leq L_1.
	\]
	In particular, $L_1$ and $E_1$ can be chosen such that $\Psi_L(r_L)<E$ for $E_1\leq E \leq E_0,\ 0< L \leq L_1$, i.e., $T$ is well-defined for points as above.
\end{lemma}

\begin{proof}
	For $L>0$ let $E_L\coloneqq\Psi_L(r_L) < E \leq E_0$ and $\epsilon >0$ such that $E - \epsilon >E_L$;
	$\epsilon$ will be specified more precisely below. 
	Then
	\[
	r_-(E,L) < r_-(E-\epsilon,L) < r_+(E-\epsilon,L) < r_+(E,L).
	\]
	{\em Step 1.}
	We estimate the time a particle takes to travel from
	$r_-(E,L)$ to $r_-(E-\epsilon,L)$. To do so we
	first show that $\Psi_L$ is convex on the interval $]0,r_L]$,
	which contains the radii specified above.
	This follows from
	\[
	\Psi_L''(r) = 4 \pi \rho_0(r) - \frac{2}{r} U_0'(r) + 3 \frac{L}{r^4}
	\]
	and the fact that by definition of $r_L$,
	\[
	\Psi_L' (r) = U_0'(r) - \frac{L}{r^3} < 0
	\]
	on $]0,r_L[$. For any $r\in [r_-(E,L),r_-(E-\epsilon,L)]$ we let
	\[
	\alpha := \frac{r-r_-(E,L)}{r_-(E-\epsilon,L)-r_-(E,L)}.
	\]
	Then
	\[
	\Psi_L(r) \leq (1-\alpha) \Psi_L(r_-(E,L)) + \alpha  \Psi_L(r_-(E-\epsilon,L))
	= E - \alpha \epsilon.
	\]
	Hence
	\begin{align*}
	&
	\int_{r_-(E,L)}^{r_-(E-\epsilon,L)} \frac{\diff r}{\sqrt{E-\Psi_L(r)}}\\
	& \qquad
	\leq \frac{\sqrt{r_-(E-\epsilon,L)-r_-(E,L)}}{\sqrt{\epsilon}}
	\int_{r_-(E,L)}^{r_-(E-\epsilon,L)} \frac{\diff r}{\sqrt{r-r_-(E,L)}}\\
	& \qquad
	= 2 \frac{r_-(E-\epsilon,L)-r_-(E,L)}{\sqrt{\epsilon}}.
	\end{align*}
	{\em Step 2.}
	Here we estimate the time a particle takes to travel from $r_-(E-\epsilon,L)$
	to $r_+(E-\epsilon,L)$. On this interval, $\Psi_L(r) \leq E-\epsilon$,
	and hence
	\[
	\int_{r_-(E-\epsilon,L)}^{r_+(E-\epsilon,L)} \frac{\diff r}{\sqrt{E-\Psi_L(r)}}
	=  \frac{r_+(E-\epsilon,L)-r_-(E-\epsilon,L)}{\sqrt{\epsilon}}.
	\]
	{\em Step 3.}
	Here we estimate the time a particle takes to travel from $r_+(E-\epsilon,L)$
	to $r_+(E,L)$, which is the crucial part. Let
	\[
	\mu := \min\{ \Psi_L'(r) \mid r\in [r_+(E-\epsilon,L),r_+(E,L)]\}.
	\]
	Clearly,
	$\Psi_L(r) \leq E - \mu \, (r_+(E,L) -r)$ for
	$r\in [r_+(E-\epsilon,L),r_+(E,L)]$.
	Since $\mu >0$ this implies that
	\begin{align*}
	\int_{r_+(E-\epsilon,L)}^{r_+(E,L)} \frac{\diff r}{\sqrt{E-\Psi_L(r)}}
	&\leq
	\int_{r_+(E-\epsilon,L)}^{r_+(E,L)} \frac{\diff r}{\sqrt{\mu (r_+(E,L) -r)}}\\
	&= 2 \frac{\sqrt{r_+(E,L)-r_+(E-\epsilon,L)}}{\sqrt{\mu}},
	\end{align*}
	so it remains to estimate $\mu$
	independently of $E$ and $L$.
	
	We observe that $r_+(E,L)\to R_0>0$ as $(E,L) \to (E_0,0)$, and $r_L \to 0$
	as $L\to 0$. Hence there exist $U_0(0) < \tilde E_1 < E_0$
	and $L_1 > 0$ such that
	\[
	r_+(E,L) > r_{2 L_1} > r_{L_1} \geq r_L
	\]
	as well as $E_L<E$ for all $\tilde E_1 \leq E \leq E_0$ and $0 < L \leq L_1$.
	We define
	\[
	\epsilon \coloneqq \frac{E_0 - \tilde E_1}2,\quad E_1 \coloneqq \frac{E_0 + \tilde E_1}2
	\]
	Consider any $E_1\leq E \leq E_0$ and $0< L \leq L_1$. Then
	$E-\epsilon \geq \tilde E_1$, and hence $r_+(E-\epsilon,L) > r_{2L_1}$
	and for $r\in [r_+(E-\epsilon,L),r_+(E,L)]$,
	\begin{align*}
	\Psi_L'(r)
	&= U_0'(r) - \frac{L}{r^3} \geq
	U_0'(r) - \frac{L_1}{r^3} = \Psi_{L_1}'(r) \\
	&\geq
	\min\{\Psi_{L_1}'(s) \mid s\in [r_{2 L_1},R_0]\} >0.
	\end{align*}
	The latter constant depends only on the given steady state and the
	parameter $L_1$ and provides the required lower bound on $\mu$.
	
	All three steps together yield the desired estimate for $T(E,L)$,
	where we note that also $\epsilon$, which enters in the estimates
	from the first two steps, has now been chosen to depend only
	on the parameters $L_1$ and $\tilde E_1=2 E_1 - E_0$.
\end{proof}

\subsection{A lower bound on $T$}\label{ssc:Tlower}

\begin{lemma}
	For all $(E,L)\in\mathring\Omega_0^{EL}$,
	\begin{align*}
	T(E,L) \geq \left(4\pi\|\rho_0\|_\infty+3\frac L{r_L^4}\right)^{-\frac12}.
	\end{align*}
\end{lemma}
\begin{proof}
  First,
  \begin{align*}
    T(E,L)
    &= 2\int_{r_-(E,L)}^{r_+(E,L)} \frac{\diff r}{\sqrt{2E-2\Psi_L(r)}}
    \geq \sqrt2 \frac{r_+(E,L)-r_-(E,L)}{\sqrt{E-\Psi_L(r_L)}}\\
    &\geq \sqrt2\frac{r_+(E,L)-r_L}{\sqrt{E-\Psi_L(r_L)}}.
  \end{align*}
  We now apply the mean value theorem to the mapping
  $]\Psi_L(r_L),0[\ni \eta\mapsto r_+(\eta,L)$, continuously extended by
  $r_+(\Psi_L(r_L),L)\coloneqq r_L$.
  Note that $r_+(\cdot,L)$ is differentiable on $]\Psi_L(r_L),0[$ with
  \begin{align*}
    \frac{\partial r_+}{\partial E} (\eta, L)
    = \frac1{\Psi_L'(r_+(\eta,L))} ,\qquad \eta\in]\Psi_L(r_L),0[
  \end{align*}
  by the inverse function theorem. Hence there exists $\eta\in]\Psi_L(r_L),E[$ such that
    \begin{align*}
      r_+(E,L)-r_L = \frac{E-\Psi_L(r_L)}{\Psi_L'(r_+(\eta,L))}.
    \end{align*}
    Thus,
    \begin{align*}
      T(E,L) &\geq\sqrt2 \frac{\sqrt{{E-\Psi_L(r_L)}}}{\Psi_L'(r_+(\eta,L))} \geq
      \sqrt2 \frac{\sqrt{{\eta-\Psi_L(r_L)}}}{\Psi_L'(r_+(\eta,L))} 
      = \sqrt2 \left( \frac{\Psi_L(r_+(\eta,L)) - \Psi_L(r_L)}{(\Psi_L')^2(r_+(\eta,L)) - (\Psi_L')^2(r_L)} \right)^{\frac12},
	\end{align*}
	where we used $\Psi_L'(r_L)=0$ and $\Psi_L'(r_+(\eta,L))>0$ in the last equality. By the (extended) mean value theorem, there exists $s\in]r_L,r_+(\eta,L)[$ such that
	\begin{align*}
	2\Psi_L'(s)\Psi_L''(s)\left( \Psi_L(r_+(\eta,L)) - \Psi_L(r_L) \right) = \Psi_L'(s) \left( (\Psi_L')^2(r_+(\eta,L)) - (\Psi_L')^2(r_L) \right).
	\end{align*}
	We therefore conclude that
	\begin{align*}
	T(E,L) &\geq \sqrt2 \left( \frac{\Psi_L'(s)}{2\Psi_L'(s)\Psi_L''(s)} \right)^{\frac12} = \frac1{\sqrt{\Psi_L''(s)}},
	\end{align*}
	note that $\Psi_L'>0$ on $]r_L,\infty[$ and the above equality yields $\Psi_L''(s)>0$.
    Moreover,
	\begin{align*}
	\Psi_L''(s) &= U_0''(s)+ 3\frac L{s^4} = -2 \frac{m_0(s)}{s^3} + 4\pi\rho_0(s)+3\frac L{s^4}
	\leq 4\pi\|\rho_0\|_\infty+ 3\frac L{s^4},
	\end{align*}
	and
	\begin{align*}
	T(E,L) &\geq \left(4\pi\|\rho_0\|_\infty+3\frac L{s^4}\right)^{-\frac12} \geq \left(4\pi\|\rho_0\|_\infty+3\frac L{r_L^4}\right)^{-\frac12}.\qedhere
	\end{align*}
\end{proof} 

In order to obtain the boundedness of $T$
away from zero when arbitrary small $L$s are in the steady state support,
we have to ensure that the continuous function $]0,\infty[\ni L\mapsto \frac L{r_L^4}$
is bounded for $L\to0$; note that $r_L\to0$ as $L\to0$.

\begin{lemma}\label{limitLoverrL4}
  It holds that      
  \[
  \lim\limits_{L\to0} \frac L{r_L^4} = \frac{4\pi}3 \rho_0(0).
  \]
\end{lemma}
\begin{proof}
  Recall that $r_L$ is the unique solution of the equation $rm_0(r)=L$.
  By the implicit function theorem, the mapping $]0,\infty[\ni L\mapsto r_L$
  is differentiable with
  \begin{align*}
    \frac{\partial r_L}{\partial L} (L) = \frac1{m_0(r_L)+4\pi r_L^3\rho_0(r_L)} ,\quad L>0.
  \end{align*}
  Therefore, by l'Hospital's rule,
  \begin{align*}
    \lim\limits_{L\to0} \frac L{r_L^4} =
    \lim\limits_{L\to0} \frac{m_0(r_L)+4\pi r_L^3\rho_0(r_L)}{4r_L^3} =
    \lim\limits_{L\to0} \frac{m_0(r_L)}{4r_L^3} + \pi \rho_0(0).
  \end{align*}
  Applying l'Hospital's rule once again we arrive at
  \begin{align*}
    \lim\limits_{L\to0} \frac{m_0(r_L)}{4r_L^3} =
    \lim\limits_{r\to0} \frac{m_0(r)}{4r^3} =
    \lim\limits_{r\to0} \frac{4\pi r^2\rho_0(r)}{12r^2} = \frac\pi3 \rho_0(0). &\qedhere
  \end{align*}
\end{proof}

Since $\rho_0(0)<\infty$ if the steady state possesses no inner vacuum region, the previous two lemmata yield the boundedness of the period function $T$ away from zero on $\mathring{\Omega}^{EL}$.

\begin{remark}
	Another approach to bound the period function $T$ from below and above is to extend it continuously onto the boundary of $\mathring \Omega_0^{EL}$.
\end{remark}

\subsection{Regularity \& monotonicity of $T$}\label{ssc:reg}

\begin{lemma}\label{partperiodcontinuous}
	$T$ is continuous on $\{ (E,L)\in]-\infty,0[\times]0,\infty[\mid\Psi_L(r_L)<E \}$.
\end{lemma}
\begin{proof}
	The change of variables $r=s\left(r_+(E,L)-r_-(E,L)\right)+r_-(E,L)$ yields
	\begin{align*}
	T(E,L) = \sqrt2 &\left(r_+(E,L)-r_-(E,L)\right) \\
	&\cdot \int_0^1 \frac{\diff s}{\sqrt{E-\Psi_L\left[ s\left(r_+(E,L)-r_-(E,L)\right)+r_-(E,L) \right]}}.
	\end{align*}
	Since $r_\pm$ are continuous by Lemma~\ref{effpot}, it remains to show that the latter integral is also continuous. This can be seen by the dominated convergence theorem, since
	\begin{multline*}
	\frac1{\sqrt{E-\Psi_L\left[ s\left(r_+(E,L)-r_-(E,L)\right)+r_-(E,L) \right]}}\\
	\leq \frac{\sqrt2 M_0^2}{\sqrt L\, E_0^2}\,\frac1{\left(r_+(E,L)-r_-(E,L)\right)\sqrt{s(1-s)}}
	\end{multline*}
	for $(E,L)\in]-\infty,0[\times]0,\infty[$ with $\Psi_L(r_L)<E$ and $0\leq s\leq1$ by Lemma~\ref{effpot}. Note that $\left(r_+-r_-\right)^{-1}$ is locally bounded and $\int_0^1\frac{\diff s}{\sqrt{s(1-s)}}=\pi<\infty$.
\end{proof}

In fact, one can show that $T$ is even differentiable. The $E$-derivative can be computed similarly to the plane symmetric case (see Proposition \ref{Tderplane}), the $L$-derivative can be computed by related techniques. Unfortunately, the non-negativity of $\partial_ET$ is harder to show in the spherically symmetric case than in the planar setting.

Nonetheless, numerical simulations clearly indicate $\partial_ET\geq0$ and $\partial_LT\leq0$
for a wide range of steady states, including the ones satisfying the assumptions of
Theorem~\ref{T:EXAMPLERADIAL}, see Remark~\ref{radialparameterexample} for an explicit example.

We again emphasize that whether period functions as the one above are monotonous is an
involved question and has been widely studied \cite{Bo05,Sc85}, especially in the context of bifurcation theory for Hamiltonian ODEs~\cite{Ch85,ChWa86,HaKo1991}. In particular,
numerical computations show that neither the monotonicities~\eqref{E:TPROPERTIES} nor~\eqref{E:TMAX} are true for general radial states; for example in the case of a polytropic~\eqref{E:POLYTROPE} steady state with large $l$.
This illustrates that~\eqref{E:TMAX} can not be shown by a general argument but depends on the exact form of the steady state, i.e., on the parameters $k$, $l$, $L_0$ in the polytropic case.
The rigorous monotonicity properties of $T$ will be treated in future work. 


%
%
%

\end{document}